\documentclass{article}

\usepackage{amsfonts} 
\usepackage{amsmath} 
\usepackage{amssymb} 
\usepackage{amsthm} 
	\newtheorem{thm}{Theorem}[section]

	\newtheorem{cor}[thm]{Corollary}
	\newtheorem{defn}[thm]{Definition}
	\newtheorem{lem}[thm]{Lemma}
	
	\newtheorem{prop}[thm]{Proposition}
	
\usepackage{authblk}
\usepackage{arydshln} 

\usepackage[margin=3cm]{geometry} 
\usepackage{hyperref} 
\usepackage{multicol}

\usepackage{graphicx} 
\usepackage{epstopdf}

\usepackage{tikz}
\usetikzlibrary{3d, calc, arrows, decorations.markings, decorations.pathmorphing, positioning, decorations.pathreplacing}
\makeatletter
\tikzset{nomorepostaction/.code={\let\tikz@postactions\pgfutil@empty}}
\makeatother

\makeatletter
\tikzoption{canvas is xy plane at z}[]{%
  \def\tikz@plane@origin{\pgfpointxyz{0}{0}{#1}}%
  \def\tikz@plane@x{\pgfpointxyz{1}{0}{#1}}%
  \def\tikz@plane@y{\pgfpointxyz{0}{1}{#1}}%
  \tikz@canvas@is@plane
}
\makeatother
\tikzset{xyp/.style={canvas is xy plane at z=#1}}
\tikzset{xzp/.style={canvas is xz plane at y=#1}}
\tikzset{yzp/.style={canvas is yz plane at x=#1}}
\tikzset{mypersp/.style={x={(0:1cm)},y={(55:0.75cm)},z={(90:1cm)}}}

\newcommand{\drawcube}{
\draw [xzp=0] (0,0) -- (0,1) -- (1,1) -- (1,0) -- cycle;
\draw [yzp=1] (0,0) -- (0,1) -- (1,1) -- (1,0) -- cycle;
\draw [xyp=1] (0,0) -- (0,1) -- (1,1) -- (1,0) -- cycle;
\draw [densely dotted, yzp=0] (0,0) -- (0,1) -- (1,1) -- (1,0) -- cycle;
\draw [densely dotted, xzp=1] (0,0) -- (0,1) -- (1,1) -- (1,0) -- cycle;
\draw [densely dotted, xyp=0] (0,0) -- (0,1) -- (1,1) -- (1,0) -- cycle;
}
\newcommand{\drawsigns}{
\draw [xyp=0] (-0.1,-0.1) node {$+$};
\draw [xyp=0] (-0.2, 1.1) node {$-$};
\draw [xyp=0] (0.9,-0.1) node  {$-$};
\draw [xyp=0] (1.1,1.1) node  {$+$};
\draw [xyp=1] (-0.1,-0.1) node {$+$};
\draw [xyp=1] (-0.2, 1.1) node {$-$};
\draw [xyp=1] (0.9,-0.1) node  {$-$};
\draw [xyp=1] (1.1,1.1) node  {$+$};
}

\newcommand{\drawvertices}{
\draw [xyp=0, color=black!50!green, fill=none] (0,0) circle (2pt);
\draw [xyp=0, color=black!50!green, fill=none] (1,1) circle (2pt);
\draw [xyp=1, color=black!50!green, fill=none] (0,0) circle (2pt);
\draw [xyp=1, color=black!50!green, fill=none] (1,1) circle (2pt);
}

\newcommand{\labelvertices}{
\draw [xyp=0.5] (-0.1,-0.1) node {$v$};
\draw [xyp=0.5] (1.1,1.1,) node {$w$};
}
\newcommand{\bottomdiagonal}{
\draw [xyp=0, color=black!50!green, ultra thick] (0,0) -- (1,1);
\draw [xyp=0, color=black!50!green, fill=black!50!green] (0,0) circle (2pt);
\draw [xyp=0, color=black!50!green, fill=black!50!green] (1,1) circle (2pt);
}
\newcommand{\topdiagonal}{
\draw [xyp=1, color=black!50!green, ultra thick] (0,0) -- (1,1);
\draw [xyp=1, color=black!50!green, fill=black!50!green] (0,0) circle (2pt);
\draw [xyp=1, color=black!50!green, fill=black!50!green] (1,1) circle (2pt);
}

\newcommand{\rightbyp}
{
\fill [gray!30!white, draw=none, yzp=1, opacity=0.5] (0,0) -- (1,0) -- (1,1) -- (0,1) -- cycle;
\draw [red, ultra thick, yzp=1] (0.5,0) arc (0:90:0.5);
\draw [red, ultra thick, yzp=1] (0.5,1) arc (-180:-90:0.5);
}
\newcommand{\rightclean}
{
\draw [red, ultra thick, yzp=1] (0.5,0) arc (180:90:0.5);
\draw [red, ultra thick, yzp=1] (0.5,1) arc (0:-90:0.5);
}
\newcommand{\frontclean}
{
\draw [red, ultra thick, xzp=0] (0.5,0) arc (180:90:0.5);
\draw [red, ultra thick, xzp=0] (0.5,1) arc (0:-90:0.5);
}
\newcommand{\frontbyp}
{
\fill [gray!30!white, draw=none, xzp=0, opacity=0.5] (0,0) -- (1,0) -- (1,1) -- (0,1) -- cycle;
\draw [red, ultra thick, xzp=0] (0.5,0) arc (0:90:0.5);
\draw [red, ultra thick, xzp=0] (0.5,1) arc (-180:-90:0.5);
}
\newcommand{\leftclean}
{
\draw [red, ultra thick, densely dotted, yzp=0] (0.5,0) arc (0:90:0.5);
\draw [red, ultra thick, densely dotted, yzp=0] (0.5,1) arc (-180:-90:0.5);
}
\newcommand{\leftbyp}
{
\fill [gray!30!white, draw=none, yzp=0, opacity=0.5] (0,0) -- (1,0) -- (1,1) -- (0,1) -- cycle;
\draw [red, ultra thick, densely dotted, yzp=0] (0.5,0) arc (180:90:0.5);
\draw [red, ultra thick, densely dotted, yzp=0] (0.5,1) arc (0:-90:0.5);
}
\newcommand{\backclean}
{
\draw [red, ultra thick, densely dotted, xzp=1] (0.5,0) arc (0:90:0.5);
\draw [red, ultra thick, densely dotted, xzp=1] (0.5,1) arc (-180:-90:0.5);
}
\newcommand{\backbyp}
{ 
\fill [gray!30!white, draw=none, xzp=1, opacity=0.5] (0,0) -- (1,0) -- (1,1) -- (0,1) -- cycle;
\draw [red, ultra thick, densely dotted, xzp=1] (0.5,0) arc (180:90:0.5);
\draw [red, ultra thick, densely dotted, xzp=1] (0.5,1) arc (0:-90:0.5);
}
\newcommand{\topoff}
{
\draw [red, ultra thick, xyp=1] (0.5,0) arc (0:90:0.5);
\draw [red, ultra thick, xyp=1] (0.5,1) arc (-180:-90:0.5);
}
\newcommand{\topon}
{
\draw [red, ultra thick, xyp=1] (0.5,0) arc (180:90:0.5);
\draw [red, ultra thick, xyp=1] (0.5,1) arc (0:-90:0.5);
}
\newcommand{\bottomon}
{
\draw [red, ultra thick, densely dotted, xyp=0] (0.5,0) arc (180:90:0.5);
\draw [red, ultra thick, densely dotted, xyp=0] (0.5,1) arc (0:-90:0.5);
}
\newcommand{\bottomoff}
{
\draw [red, ultra thick, densely dotted, xyp=0] (0.5,0) arc (0:90:0.5);
\draw [red, ultra thick, densely dotted, xyp=0] (0.5,1) arc (-180:-90:0.5);
}
\newcommand{\strandbackgroundshading}
{
\draw [draw=none, fill=gray!10!white] (0,0) -- (0,0.5) -- (1,0.5) -- (1,0) -- cycle;
\draw [draw=none, fill=gray!10!white] (0,1) -- (0,1.5) -- (1,1.5) -- (1,1) -- cycle;
}

\newcommand{\cubestrandsetup}
{
\draw [ultra thick, color=black!50!green] (0,0) -- (0,0.5);
\draw [ultra thick, color=black!50!green] (0,1) -- (0,1.5);
\draw (0,0.25) to [bend left=90] (0,1.25);
\draw (-0.15,0.15) node {$v$};
\draw (-0.15,1.35) node {$w$};
}
\newcommand{\lefton}
{
\draw [color=black!50!green, fill=black!50!green] (0,0.25) circle (2 pt);
\draw [color=black!50!green, fill=black!50!green] (0,1.25) circle (2 pt);
}
\newcommand{\leftoff}
{
\draw [color=black!50!green, fill=none] (0,0.25) circle (2 pt);
\draw [color=black!50!green, fill=none] (0,1.25) circle (2 pt);
}
\newcommand{\righton}
{
\draw [color=black!50!green, fill=black!50!green] (1,0.25) circle (2 pt);
\draw [color=black!50!green, fill=black!50!green] (1,1.25) circle (2 pt);
}
\newcommand{\rightoff}
{
\draw [color=black!50!green, fill=none] (1,0.25) circle (2 pt);
\draw [color=black!50!green, fill=none] (1,1.25) circle (2 pt);
}
\newcommand{\aftervused}
{
\draw [draw=none, fill=gray!50!white] (0,0.25) -- (0,0.5) -- (1,0.5) -- (1,0.25) -- cycle;
}
\newcommand{\beforevused}
{
\draw [draw=none, fill=gray!50!white] (0,0.25) -- (0,0) -- (1,0) -- (1,0.25) -- cycle;
}
\newcommand{\afterwused}
{
\draw [draw=none, fill=gray!50!white] (0,1.25) -- (0,1.5) -- (1,1.5) -- (1,1.25) -- cycle;
}
\newcommand{\beforewused}
{
\draw [draw=none, fill=gray!50!white] (0,1.25) -- (0,1) -- (1,1) -- (1,1.25) -- cycle;
}

\newcommand{\To}{\longrightarrow}

\newcommand{\0}{{\bf 0}}

\newcommand{\A}{\mathcal{A}}

\newcommand{\Cat}{\mathcal{C}}

\newcommand{\Z}{\mathbb{Z}}
\newcommand{\ZZ}{\mathcal{Z}}

\DeclareMathOperator{\Gr}{Gr}

\DeclareMathOperator{\Int}{Int}
\DeclareMathOperator{\inv}{inv}
\DeclareMathOperator{\Inv}{Inv}

\DeclareMathOperator{\Ob}{Ob}

\DeclareMathOperator{\supp}{supp}

\begin{document}

\title{Strand algebras and contact categories} 

\author{Daniel V. Mathews} 
\affil{School of Mathematical Sciences, Monash University \\
\texttt{Daniel.Mathews@monash.edu}}



\maketitle

\begin{abstract}
We demonstrate an isomorphism between the homology of the strand algebra of bordered Floer homology, and the category algebra of the contact category introduced by Honda. This isomorphism provides a direct correspondence between various notions of Floer homology and arc diagrams, on the one hand, and contact geometry and topology on the other. In particular, arc diagrams correspond to quadrangulated surfaces, idempotents correspond to certain basic dividing sets, strand diagrams correspond to contact structures, and multiplication of strand diagrams corresponds to stacking of contact structures. The contact structures considered are cubulated, and the cubes are shown to behave equivalently to local fragments of strand diagrams.
\end{abstract}


\tableofcontents

\section{Introduction}

In this paper we prove an isomorphism providing a new link between 3-dimensional contact topology and Heegaard Floer theory. These two subjects are well known to be closely related: for instance, contact structures on 3-manifolds yield elements of Heegaard Floer homology \cite{HKM09, HKMContClass, OSContact}, which behave rather nicely (e.g. \cite{HKM08}), and there is much other evidence of deep connections (e.g. \cite{Cooper15, EVVZ14, Honda_Tian}). More specifically, Zarev showed that, in the context of bordered sutured Heegaard Floer theory, the homology of the strand algebra has a description in terms of sutured Floer homology of a thickened surface $\Sigma \times [0,1]$ \cite{Zarev10}. We will show that the homology of this strand algebra can be interpreted directly in terms of contact structures on a thickened surface $\Sigma \times [0,1]$. 

\begin{thm}
\label{thm:main_thm}
Let $\ZZ$ be an arc diagram corresponding to a quadrangulated surface $(\Sigma,Q)$. Then there is an isomorphism of unital $\Z_2$-algebras
\[
CA (\Sigma,Q)
\cong
H(\A(\ZZ)),
\]
where $CA (\Sigma,Q)$ is the contact category algebra of $(\Sigma,Q)$, $\A(\ZZ)$ is the strand algebra of $\ZZ$, and $H(\A(\ZZ))$ is its homology.
\end{thm}

We will define all these notions as we proceed. The proof is direct and explicit, encapsulated in a local correspondence, depicted in figure \ref{fig:tight_cubes}, between fragments of strand diagrams, and dividing sets drawn on cubes. It comes as part of a larger set of correspondences, summarised in figure \ref{fig:dictionary}, translating between various notions of contact geometry, on the one hand, and notions of strand algebras, on the other. The contact structures we consider are \emph{cubulated}, and this work suggests more generally that an approach to 3-dimensional contact topology based on cubulation may be useful. (This notion of cubulation is far simpler than the one which has been so fruitful in recent years in 3-manifold topology, e.g. \cite{Agol_ICM_virtual, Borgeron-Wise12, Calegari13_virtual, Wise12}.)

The contact-topology side of the isomorphism involves the \emph{contact category} introduced by Honda \cite{HonCat}, which has been studied by the author \cite{Me09Paper} and a formal version of which has been studied by Cooper \cite{Cooper15}. A contact category $\Cat(\Sigma,F)$ is defined for any \emph{marked surface} $(\Sigma,F)$ consisting of a compact oriented surface $\Sigma$ with signed points on its boundary. Roughly, objects of $\Cat(\Sigma,F)$ are dividing sets $\Gamma$ on $\Sigma$, which describe contact structures in a small product neighbourhood of $\Sigma$, and morphisms are contact structures on $\Sigma \times [0,1]$ which near $\Sigma \times \{0\}$ and $\Sigma \times \{1\}$ are prescribed by source and target objects $\Gamma_0, \Gamma_1$. 
Composition of morphisms stacks these contact structures on top of each other. In \cite{Me12_itsy_bitsy} we discussed a natural type of \emph{quadrangulation} of a marked surface, and in \cite{Me14_twisty_itsy_bitsy} we showed such quadrangulations are equivalent to a graph-theoretic structure which we called a \emph{tape graph}. A quadrangulation naturally yields \emph{basic} dividing sets, which take a standard form on each square of the quadrangulation. (Related notions appear in work of Zarev \cite{Zarev10} and Honda--Tian \cite{Honda_Tian}.) The dividing sets basic with respect to a quadrangulation $Q$ yield a subcategory $\Cat(\Sigma,Q)$ of $\Cat(\Sigma,F)$, and the contact category algebra $CA (\Sigma,Q)$ appearing in theorem \ref{thm:main_thm} is essentially its $\Z_2$ category algebra.

On the Heegaard-Floer side of the isomorphism, we have the differential graded algebra $\A(\ZZ)$ defined by Zarev \cite{Zarev09} in developing the theory of bordered sutured Floer homology. This theory is a generalisation of both the bordered Floer homology of Lipshitz--Ozsv\'{a}th--Thurston \cite{LOT08} and the sutured Floer homology of Juh\'{a}sz \cite{Ju06}, which in turn are generalisations of Heegaard Floer homology \cite{OS04Prop, OS04Closed, OS06} to 3-manifolds with boundary. A bordered sutured 3-manifold, roughly speaking, is a 3-manifold with boundary where some of the boundary is sutured. In bordered sutured Floer theory, the boundary surface is described by an \emph{arc diagram} $\mathcal{Z}$, which can be regarded as the boundary data of a 3-manifold Heegaard decomposition. The \emph{strand algebra} $\A(\ZZ)$ associated to $\mathcal{Z}$ generalises the algebra associated to a pointed matched circle in bordered Floer homology \cite{LOT08}, and is generated by certain \emph{strand diagrams} related to the arc diagram $\mathcal{Z}$. Roughly, multiplication is defined by concatenating such diagrams, and the differential resolves intersections in such diagrams.

In proving theorem \ref{thm:main_thm}, we will show that there is a natural correspondence between notions arising on both sides. We will show that an arc diagram $\mathcal{Z}$ is equivalent to a quadrangulated surface $(\Sigma,Q)$. We will show that idempotents of $\A(\ZZ)$ correspond to basic dividing sets on $(\Sigma,Q)$, which are also the elementary dividing sets of \cite[sec. 6.1]{Zarev10}. We will show that matched-pair fragments of $\ZZ$ correspond to the squares of $Q$, or the cubes of the corresponding cubulation of $\Sigma \times [0,1]$. We will show that strand diagrams representing elements of $H(\A(\ZZ))$ correspond to cubulated contact structures on $\Sigma \times [0,1]$, and that all contact structures in $CA(\Sigma,Q)$ are of this form. And we will show that concatenation of strand diagrams corresponds to stacking the contact cubes of cubulated contact structures.

As discussed in \cite{LOT08}, strand diagrams are designed to encode Reeb chords arising as asymptotics of holomorphic curves near the boundary of a 3-manifold in bordered Floer homology. The multiplication and differential in the strand algebra are designed so as to describe the behaviour of rigid holomorphic curves in a cylinder between Reeb chords. It is perhaps surprising that the homology of this algebra should so precisely encode something ostensibly quite different, such as contact structures. Recently, Honda and Tian in \cite{Honda_Tian} have found embeddings of contact categories of discs into a homotopy category of bounded cochain complexes of finitely projective left modules over a ring isomorphic to the homology of a strand algebra, which indicates these connections run deeper still.

Several results exist in the literature on the homology of the strand algebra. In \cite{LOT11_Bimodules} Lipshitz--Ozsv\'{a}th--Thurston gave an explicit description of this homology, in the case of a pointed matched circle. We rely upon a generalisation of this description in proving theorem \ref{thm:main_thm}. In \cite{Zarev10}, Zarev showed that $H(\A(\mathcal{Z}))$ is isomorphic to the direct sum of the sutured Floer homology of $\Sigma \times [0,1]$, with various sets of sutures, ranging over basic dividing sets on $\Sigma \times \{0\}$ and $\Sigma \times \{1\}$. We write $M(\Gamma_0, \Gamma_1)$ to denote $\Sigma \times [0,1]$ with sutures $\Gamma_0, \Gamma_1$ drawn on $\Sigma \times \{0\}, \Sigma \times \{1\}$ respectively (full details are given in definition \ref{defn:M}). Combining Zarev's isomorphism with theorem \ref{thm:main_thm}, we obtain the following.
\begin{cor}
\label{cor:SFH_dimension}
Let $(\Sigma, Q)$ be a quadrangulated surface. Then there is an isomorphism of unital $\Z_2$ algebras
\[
CA(\Sigma,Q) \cong \bigoplus_{\Gamma_0, \Gamma_1 \text{ basic}}
SFH \left( -M(\Gamma_0, \Gamma_1) \right),
\]
where the sum is taken over all pairs of dividing sets $(\Gamma_0, \Gamma_1)$ basic with respect to $Q$. In particular, $SFH(-M(\Gamma_0, \Gamma_1))$ has dimension equal to the number of isotopy classes of tight contact structures on $M(\Gamma_0, \Gamma_1)$.
\end{cor}
As discussed by Zarev in \cite{Zarev10}, the multiplication on the right hand side of the isomorphism is given by gluing maps on $SFH$. Moreover, Zarev asserts that the gluing map agrees with contact cobordism maps of \cite{HKM08}, although the proof has not yet appeared. So we expect the isomorphism of corollary \ref{cor:SFH_dimension} in fact sends each contact structure in $CA(\Sigma, Q)$ to the corresponding contact invariant (in the sense of \cite{HKM09}) in $SFH$.

This paper is structured as follows. In section \ref{sec:contact_categories} we develop the necessary contact geometry, starting from background on convex surfaces and ending at a description of the contact category algebra. In section \ref{sec:strands_algebra} we develop the strand algebra and describe its homology, relying on work of Lipshitz--Ozsv\'{a}th--Thurston and Zarev. In section \ref{sec:correspondence} we prove theorem \ref{thm:main_thm}, describing the correspondence between the strand algebra and contact category. Finally in section \ref{sec:SFH} we consider sutured Floer homology and prove corollary \ref{cor:SFH_dimension}.

\medskip

\noindent {\bf Acknowledgments.} I would like to thank Rumen Zarev for introducing me to this topic. The author is supported by Australian Research Council grant DP160103085.

\section{Contact categories and quadrangulated surfaces}
\label{sec:contact_categories}

We begin by discussing the contact geometry on one side of the isomorphism of theorem \ref{thm:main_thm}.

\subsection{Convex surfaces and dividing sets}

\begin{defn}
\label{def:marked_surface}
A \emph{marked surface} is a pair $(\Sigma,F)$ consisting of an orientable compact surface $\Sigma$, together with a finite set $F \subset \partial \Sigma$ of signed points, such that
\begin{enumerate}
\item every component of $\Sigma$ has nonempty boundary,
\item every boundary component of $\Sigma$ contains points of $F$, and
\item around each boundary component of $\Sigma$, the points of $F$ alternate in sign.
\end{enumerate}
\end{defn}
The boundary $\partial \Sigma$ is cut by $F$ into arcs; we can orient these arcs positively and  negatively, in alternating fashion. Following \cite{Me12_itsy_bitsy}, the positive arcs $C_+$ and negative arcs $C_-$ are oriented so that $C_\pm \subset \pm \partial \Sigma$ and $\partial C_\pm = -F$ as oriented manifolds.

A word on terminology: in \cite{Zarev09} Zarev called such $(\Sigma, F)$ a ``sutured surface"; in \cite{Me12_itsy_bitsy} we called $(\Sigma, F)$ a ``sutured background", and called a ``sutured surface" such a surface with further curves were drawn on $\Sigma$. To avoid confusion, here we avoid ``sutured surface" and use ``marked surface". 

\begin{defn}
A \emph{dividing set} on a marked surface $(\Sigma,F)$ is an oriented 1-manifold $\Gamma$ properly embedded in $\Sigma$, such that $\partial \Gamma = F$ as oriented 0-manifolds, and such that $\Gamma$ cuts $\Sigma$ into alternating positive and negative components, $\Sigma \backslash \Gamma = R_+ \sqcup R_-$, where $R_\pm$ are oriented as $\pm \Sigma$, and $\partial R_+ = C_+ \cup \Gamma$, $\partial R_- = C_- \cup \Gamma$.
\end{defn}
Thus $\Gamma$ cuts $\Sigma$ coherently into positive and negative components. We regard dividing sets as equivalent if they are isotopic through dividing sets, and in practice we elide the distinction between dividing sets and their equivalence classes.

A dividing set $\Gamma$ on $(\Sigma,F)$ determines a contact structure $\xi_\Gamma$ on $\Sigma \times [0,1]$. Letting $X$ denote the unit vector in the $[0,1]$ direction, this $\xi_\Gamma$ is invariant in the $X$ direction, and $\partial \Sigma \times \{\cdot\}$ is Legendrian. Here $\Gamma$ is the set of points of each $\Sigma \times \{\cdot\}$ where $X \in \xi_\Gamma$. Proceeding along a Legendrian boundary component $C$ of $\Sigma$, the contact planes of $\xi_\Gamma$ rotate by $\pi$ for each successive point of $C \cap \Gamma$, so that $\xi_\Gamma$ makes $-\frac{1}{2} |C \cap \Gamma|$ full twists along $C$, relative to $S$.

More generally, a \emph{convex surface} $\Sigma$ in a contact 3-manifold $(M, \xi)$ is an embedded surface with a transverse contact vector field $X$. When $\Sigma$ has boundary, we require it to be Legendrian and for $\xi$ to have negative twisting along each boundary component with respect to $\Sigma$. A convex surface $\Sigma$ has a dividing set, which is the locus of points where $X \in \xi$. Every embedded surface in $(M, \xi)$ is $C^\infty$-close to a convex surface. The dividing set determines the germ of the contact structure near $\Sigma$, in an appropriate sense. See \cite{Gi91} for details.

For us a \emph{sutured 3-manifold} is a 3-manifold $M$ with a dividing set $\Gamma$ drawn on its boundary; it can be regarded as a sutured 3-manifold in the sense of \cite{Gabai83}. We regard $\Gamma$ as a prescribed contact structure near $\partial M$. A contact structure on $(M, \Gamma)$ is a contact structure on $M$ with this boundary condition.

A dividing set $\Gamma$ on $\Sigma$ is \emph{tight} if the contact structure it determines near $\Sigma$ is tight. If $\Sigma$ is a sphere, $\Gamma$ is tight if and only if $\Gamma$ is connected; otherwise, $\Gamma$ is tight if and only if it contains no contractible closed curves \cite{Hon00I}.

The quantity $\frac{1}{2}|F| - \chi(\Sigma)$ turns out to be a useful measure of the complexity of a marked surface $(\Sigma,F)$; we call it the \emph{index} of $(\Sigma,F)$, denoted $I(\Sigma,F)$. The index is additive on connected components, and $I(\Sigma,F) \geq 0$. A connected $(\Sigma,F)$ has $I(\Sigma, F)=0$ if and only if it is a disc with two marked points or \emph{bigon}, and has
$I(\Sigma, F)=1$ if and only if it is a \emph{square}, i.e. a disc with 4 marked points.

If $\Sigma$ is convex in $(M, \xi)$ with dividing set $\Gamma$, then the Euler class $e(\xi)$ evaluates on $[\Sigma] \in H_2(M)$ as $\chi(R_+)- \chi(R_-)$. The same applies when $\Sigma$ has boundary and $e(\xi)$ is a relative Euler class. We therefore define the Euler class of a dividing set $\Gamma$ to be $e(\Gamma) = \chi(R_+) - \chi(R_-) \in \Z$. One can show that $e(\Gamma) \equiv I(\Sigma,F)$ mod $2$. Moreover, if $\Gamma$ is tight then $|e(\Gamma)| \leq I(\Sigma, F)$. This is essentially the \emph{Thurston-Bennequin inequality}; see e.g. \cite{Bennequin83, ElMartinet}, and \cite{Me12_itsy_bitsy, Me16ContactCategories} for an account in the present context.

A properly embedded curve in a convex surface $\Sigma$ is \emph{Legendrian realisable} if, by a small isotopy of $\Sigma$ in an invariant neighbourhood, it can be made Legendrian \cite{Hon00I}. The \emph{Legendrian realisation principle} says that a properly embedded curve $c$ is Legendrian realisable if and only if it is transverse to $\Gamma$ and \emph{nonisolating} in the sense that every component of $\Sigma \backslash (\Gamma \cup c)$ has boundary intersecting $\Gamma$. 

\subsection{Quadrangulations}
\label{sec:quadrangulations}

A set of \emph{vertices} $V = V_+ \sqcup V_-$ on a marked surface $(\Sigma,F)$ consists of signed points on $\partial \Sigma$, with one vertex of $V_\pm$ in each boundary arc $C_\pm$ \cite{Me12_itsy_bitsy, Me14_twisty_itsy_bitsy}. Vertices are uniquely determined up to isotopy in $\partial \Sigma \backslash F$. As we proceed around an oriented boundary component of $\Sigma$, we pass vertices and marked points in the cyclic order $V_+, F_-, V_-, F_+, \ldots$. Thus each arc of a dividing set on $(\Sigma,F)$ runs from $F_-$ to $F_+$, $V_+ \subset R_+$ and $V_- \subset R_-$. See figure \ref{fig:neighbourhood_of_dividing_curve}. Indeed, with $\Sigma$ given, $V$ determines $F$, and $F$ determines $V$, up to an appropriate sense of isotopy. So we may regard the structures $(\Sigma,F)$, $(\Sigma,V)$ and $(\Sigma,F,V)$ as equivalent to each other, and interchangeable; we refer to any and all of them as \emph{marked surfaces}.

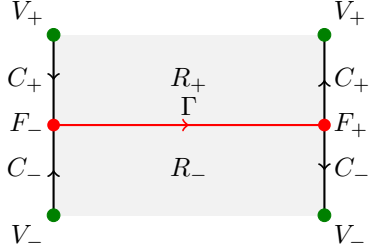
\begin{figure}
\begin{center}

\begin{tikzpicture}[
scale=1.2, 
fill = gray!10,
decomposition/.style={thick, draw=green!50!black}, 
vertex/.style={draw=green!50!black, fill=green!50!black},
suture/.style={thick, draw=red, postaction={nomorepostaction,decorate, decoration={markings,mark=at position 0.5 with {\arrow{>}}}}},
midarrow/.style={thick, postaction={nomorepostaction,decorate, decoration={markings,mark=at position 0.5 with {\arrow{>}}}}} ]

\coordinate [label = below left:{$V_-$}] (bl) at (0,0);
\coordinate [label = left:{$F_-$}] (ml) at (0,1);
\coordinate [label = above left:{$V_+$}] (tl) at(0,2);
\coordinate [label = below right:{$V_-$}] (br) at (3,0);
\coordinate [label = right:{$F_+$}] (mr) at (3,1);
\coordinate [label = above right:{$V_+$}] (tr) at (3,2);

\fill (0,0) rectangle (3,2);

\draw [suture] (ml) to node [midway, above] {$\Gamma$} (mr);
\draw [midarrow] (bl) to node [midway, left] {$C_-$} (ml);
\draw [midarrow] (tl) to node [midway, left] {$C_+$} (ml);

\draw [midarrow] (mr) to node [midway, right] {$C_-$} (br);
\draw [midarrow] (mr) to node [midway, right] {$C_+$} (tr);

\draw (1.5,0.5) node {$R_-$};
\draw (1.5,1.5) node {$R_+$};

\foreach \point in {tl, bl, tr, br}
\fill [vertex] (\point) circle (2pt);
\foreach \point in {ml, mr}
\fill [red] (\point) circle (2pt);

\end{tikzpicture}

\end{center}
\caption{Neighbourhood of a dividing curve in a marked surface $(\Sigma, F, V)$.}
\label{fig:neighbourhood_of_dividing_curve}
\end{figure}

There is a natural type of arc along which to cut a marked surface $(\Sigma, F, V)$, compatible with decomposition of dividing sets \cite{Me12_itsy_bitsy}.
\begin{defn}
A \emph{decomposing arc} on a marked surface $(\Sigma,V)$ is a properly embedded arc in $\Sigma$ with one endpoint in $V_+$ and one endpoint in $V_-$.
\end{defn}
Cutting $(\Sigma,V)$ along a decomposing arc yields a surface $\Sigma'$; the vertices $V$ naturally provide a set of vertices $V'$ on $\Sigma'$, so we have a marked surface $(\Sigma',V')$. The decomposing arc is \emph{trivial} if it cuts a bigon off $(\Sigma, V)$: in this case it connects two adjacent vertices by a boundary-parallel arc. Any connected marked surface $(\Sigma,V)$ other than a bigon or square has a nontrivial decomposing arc and thus we can successively decompose along nontrivial arcs until we arrive at a collection of squares; hence the following definition.
\begin{defn}
\label{def:quadrangulation}
A \emph{quadrangulation} $Q$ of a marked surface $(\Sigma,V)$ is a set of decomposing arcs which cut $(\Sigma,V)$ into a set of squares.
\end{defn}
See figure \ref{fig:punctured_torus_decomposition}  for an example. We denote a quadrangulated surface by $(\Sigma,Q)$, $(\Sigma,F,Q)$ of equivalent. In practice we may refer to a quadrangulation $Q$ either by the set of decomposing arcs, or by the complementary squares on $\Sigma$. We regard quadrangulations as equivalent if their decomposing arcs are isotopic and, as with dividing sets, elide the distinction between them and their equivalence classes.

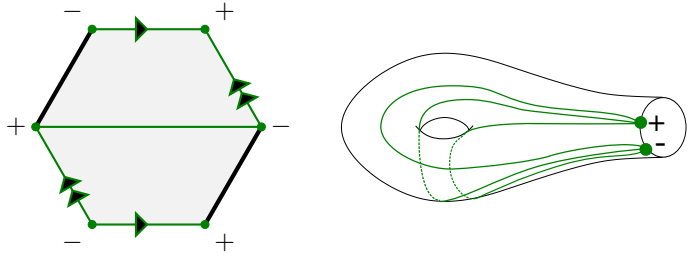
\begin{figure}
\begin{center}

\begin{tabular}{c}
\begin{tikzpicture}[
scale=1.5, 
boundary/.style={ultra thick}, 
decomposition/.style={thick, draw=green!50!black}, 
vertex/.style={draw=green!50!black, fill=green!50!black},
>=triangle 90, 
decomposition glued1/.style={thick, draw=green!50!black, postaction={nomorepostaction,decorate, decoration={markings,mark=at position 0.5 with {\arrow{>}}}}},
decomposition glued2/.style={thick, draw = green!50!black, postaction={nomorepostaction, decorate, decoration={markings,mark=at position 0.5 with {\arrow{>>}}}}}
]

\coordinate [label = right:{$-$}] (0) at (0:1);
\coordinate [label = above right:{$+$}] (1) at (60:1);
\coordinate [label = above left:{$-$}] (2) at (120:1);
\coordinate [label = left:{$+$}] (3) at (180:1);
\coordinate [label = below left:{$-$}] (4) at (240:1);
\coordinate [label = below right:{$+$}] (5) at (300:1);

\fill [gray!10] (0) -- (1) -- (2) -- (3) -- (4) -- (5) -- cycle;

\draw [decomposition glued2] (0) -- (1);
\draw [decomposition glued1] (2) -- (1);
\draw [boundary] (2) -- (3);
\draw [decomposition glued2] (4) -- (3);
\draw [decomposition glued1] (4) -- (5);
\draw [boundary] (5) -- (0);
\draw [decomposition] (0) to (3);

\foreach \point in {0, 1, 2, 3, 4, 5}
\fill [vertex] (\point) circle (1pt);

\end{tikzpicture}
\end{tabular}
\begin{tabular}{c}
\includegraphics[scale=0.4]{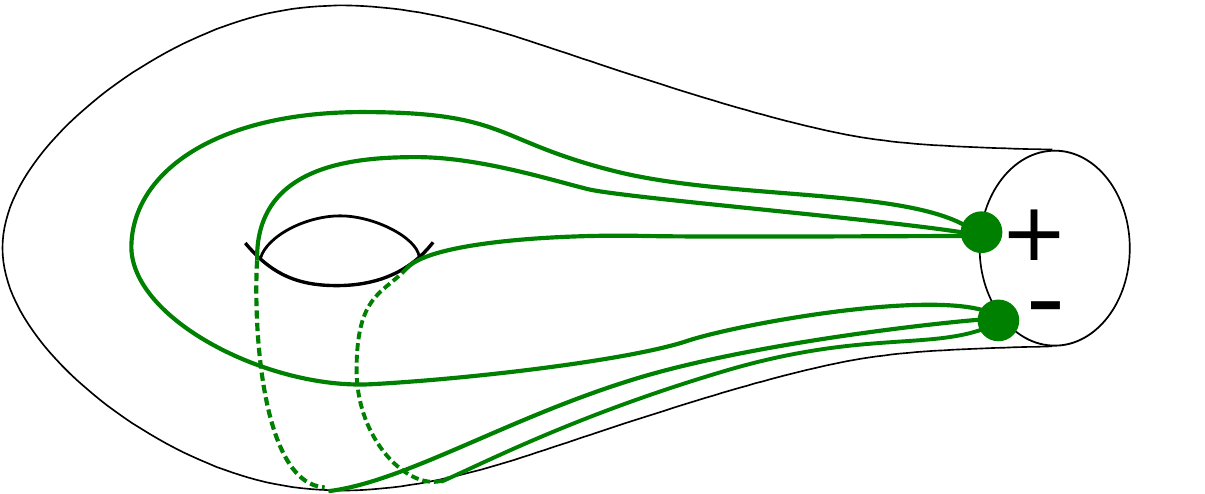}
\end{tabular}
\caption{Two equivalent views of a quadrangulation of a punctured torus with two vertices.}
\label{fig:punctured_torus_decomposition}
\end{center}
\end{figure}

The number of arcs and squares in a quadrangulation of $(\Sigma, V)$ is determined by the topology of $\Sigma$, and $|V|$: the number of arcs is is $\frac{1}{2}|V| - 2\chi(\Sigma)$, and the number of squares is the index $I(\Sigma,V)$ \cite[prop. 4.2]{Me12_itsy_bitsy}.

On a square there are precisely two tight dividing sets, which we call \emph{standard}: see figure \ref{fig:standard_squares}. They have Euler classes $1$ and $-1$, which we call \emph{standard positive} and \emph{standard negative} respectively. 

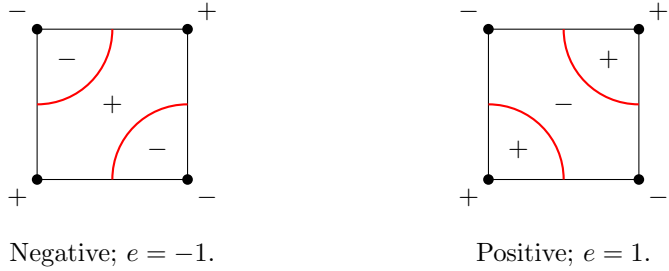
\begin{figure}
\begin{center}

\begin{tikzpicture}[
scale=2, 
suture/.style={thick, draw=red}, 
]

\coordinate [label = above left:{$-$}] (1tl) at (0,1);
\coordinate [label = above right:{$+$}] (1tr) at (1,1);
\coordinate [label = below left:{$+$}] (1bl) at (0,0);
\coordinate [label = below right:{$-$}] (1br) at (1,0);

\draw (1bl) -- (1br) -- (1tr) -- (1tl) -- cycle;
\draw [suture] (0.5,0) to [bend left=45] (1,0.5);
\draw [suture] (0,0.5) to [bend right=45] (0.5,1);
\draw (0.2,0.8) node {$-$};
\draw (0.5,0.5) node {$+$};
\draw (0.8,0.2) node {$-$};
\draw (0.5,-0.5) node {Negative; $e=-1$.};

\coordinate [label = above left:{$-$}] (2tl) at (3,1);
\coordinate [label = above right:{$+$}] (2tr) at (4,1);
\coordinate [label = below left:{$+$}] (2bl) at (3,0);
\coordinate [label = below right:{$-$}] (2br) at (4,0);

\draw (2bl) -- (2br) -- (2tr) -- (2tl) -- cycle;
\draw [suture] (3.5,1) to [bend right=45] (4,0.5);
\draw [suture] (3.5,0) to [bend right=45] (3,0.5);
\draw (3.8,0.8) node {$+$};
\draw (3.5,0.5) node {$-$};
\draw (3.2,0.2) node {$+$};
\draw (3.5,-0.5) node {Positive; $e=1$.};

\foreach \point in {1bl, 1br, 1tl, 1tr, 2bl, 2br, 2tl, 2tr}
\fill [black] (\point) circle (1pt);

\end{tikzpicture}

\caption{Standard dividing sets on the square.}
\label{fig:standard_squares}
\end{center}
\end{figure}

A dividing set $\Gamma$ on a quadrangulated surface can always be made transverse to the decomposing arcs of the quadrangulation; moreover each decomposing arc $a$ intersects $\Gamma$ an odd number of times. The situation is nicest when $|a \cap \Gamma| = 1$, for then $\Gamma$ also yields a dividing set on the marked surface obtained by cutting along $a$, without having to add any extra vertices.
\begin{defn}
Let $(\Sigma,F,Q)$ be a quadrangulated surface. A dividing set on $(\Sigma,F)$ which restricts to a standard dividing set on each square of $Q$ is called \emph{basic} with respect to $Q$. 
\end{defn}
When the quadrangulation is understood we will simply refer to $\Gamma$ as \emph{basic}. (Their contact invariants form a basis for $SFH(\Sigma \times S^1, F \times S^1)$; see \cite{Me12_itsy_bitsy, Me14_twisty_itsy_bitsy}.) Thus there are $2^{I(\Sigma,F)}$ basic dividing sets on any quadrangulation of $(\Sigma,F)$. The Euler class of a basic $\Gamma$ is given by the number of positive minus the number of negative squares. Basic dividing sets are tight; indeed, they are \emph{nonconfining}, meaning that every component of $\Sigma \backslash \Gamma$ intersects $\partial \Sigma$. Indeed, any nonconfining dividing set on a marked surface $(\Sigma,F)$ without bigon components is basic with respect to some quadrangulation. See \cite{Me12_itsy_bitsy} for details.

In \cite{Me14_twisty_itsy_bitsy} we defined the notion of \emph{tape graph} and discussed its relation to quadrangulations. A \emph{tape graph} is a finite graph with a total ordering of the half-edges incident to each vertex. (A ribbon graph, by comparison, has a cyclic ordering at each vertex.) Like a ribbon graph, a tape graph can naturally be thickened into an oriented surface with boundary. 

Given a quadrangulated surface $(\Sigma, Q)$, draw the diagonal in each square connecting the positive vertices. These vertices and diagonals form an embedded graph in $\Sigma$ called the \emph{positive spine} $G_Q^+$ of the quadrangulation $Q$. The positive spine naturally has the structure of a tape graph, since at each vertex the incident half-edges (of diagonals) are ordered clockwise by the orientation on $\Sigma$. Indeed, $\Sigma$ is the thickening of $G_Q^+$, and $\Sigma$ deformation retracts onto $G_Q^+$ \cite[lem. 4.2]{Me14_twisty_itsy_bitsy}. As we will see in section \ref{sec:arc_diagrams_tape_graphs}, tape graphs are closely related to arc diagrams.

\subsection{Corners and rounding}

It will be crucial in our constructions to \emph{smooth} and \emph{sharpen} a surface along a curve \cite{Hon00I}. Let $c$ a properly embedded Legendrian curve in a convex surface $\Sigma$ with dividing set $\Gamma$. Then by an isotopy of $\Sigma$ we may ``sharpen" $\Sigma$ along $c$ into a Legendrian corner. The resulting surface is no longer smooth, but can be regarded (near $c$) as two smooth convex surfaces meeting along a common Legendrian boundary $c$.

Conversely, if $\Sigma$ is a convex surface with a 1-dimensional corner along a simple closed Legendrian curve, then an isotopy makes $\Sigma$ into a smooth convex surface with $c$ an embedded Legendrian curve. The effect of such a smoothing or sharpening isotopy on the dividing set is shown in figure \ref{fig:smoothing_sharpening}. Thus, dividing sets \emph{interleave} along a corner. 

We can therefore broaden our definitions of convex surface and dividing set to include those obtained by sharpening along embedded Legendrian simple  closed curves. We can sharpen and round corners at will, with a well-defined effect on the dividing set. We also broaden our definition of a sutured 3-manifold to include such corners. Euler classes can still be evaluated, after rounding corners.

\begin{figure}
\begin{center}
\begin{tikzpicture}
\draw (-3,0) node {\includegraphics[scale=0.3]{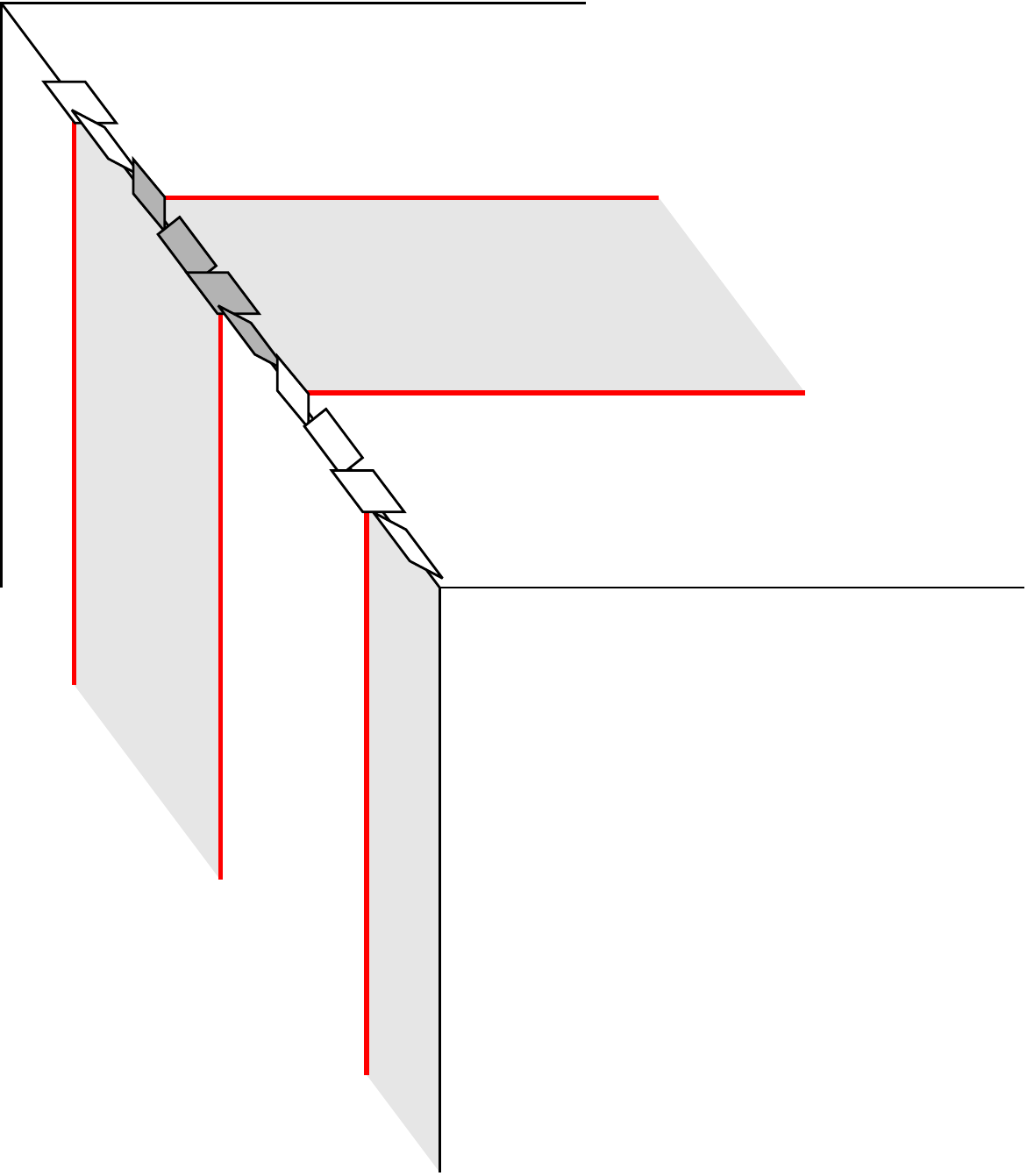}};
\draw [<->] (-0.5,0) -- (0.5,0);
\draw (3,0) node {\includegraphics[scale=0.3]{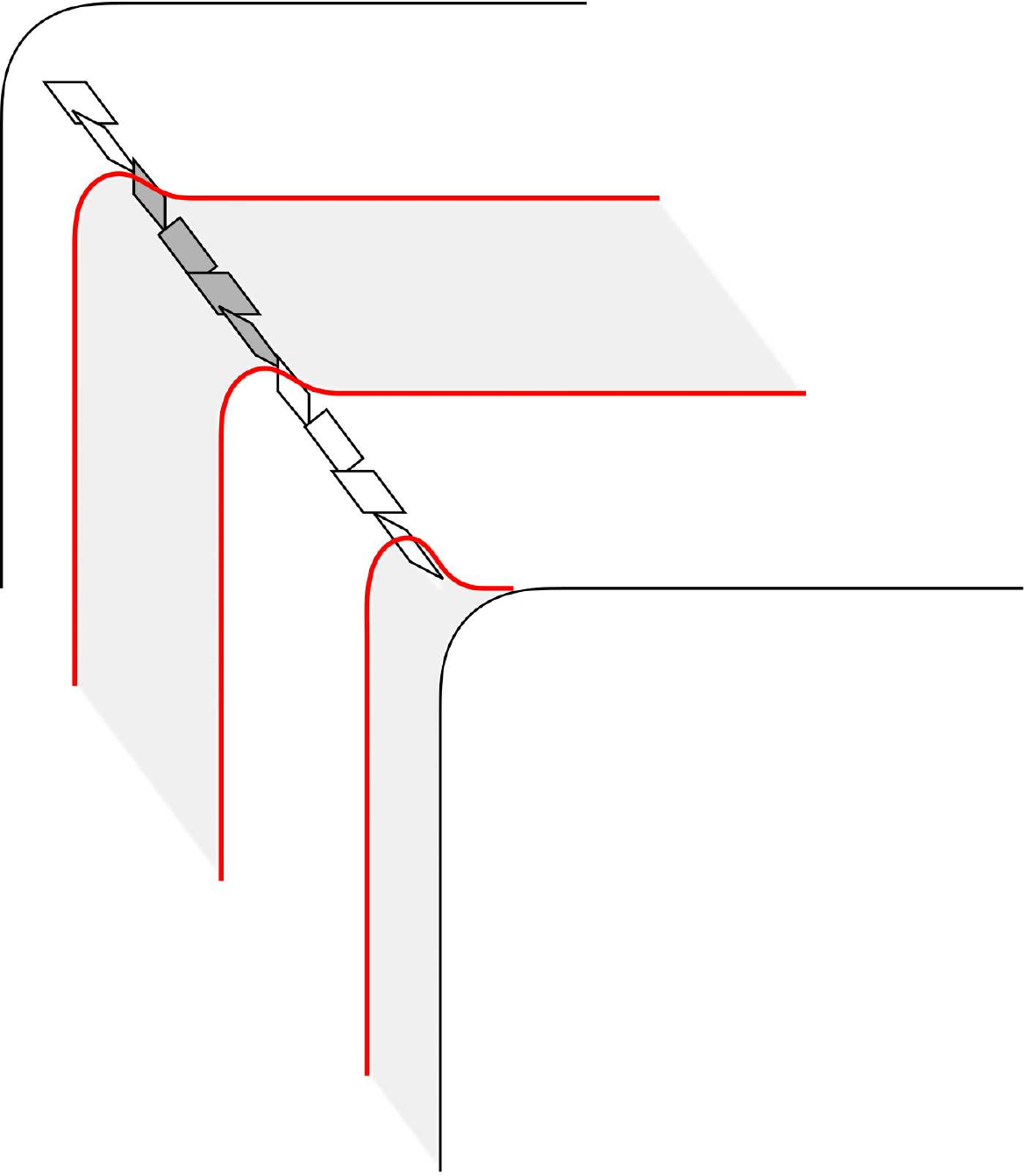}};
\end{tikzpicture}
\caption{Smoothing a corner, or conversely, sharpening along a Legendrian curve.}
\label{fig:smoothing_sharpening}
\end{center}
\end{figure}

Since a dividing set determines the germ of the contact structure near a surface, if we have two convex surfaces $S_1, S_2 \subset \partial M$ in the boundary of a contact 3-manifold $(M, \xi)$, whose dividing sets can be glued, then the contact structures can also be glued. If $S_1, S_2$ have corners along their Legendrian boundaries, then they can be glued so that the dividing sets connect smoothly across the corners, yielding a smooth dividing set on the glued manifold.

\subsection{Contact categories}

We will be concerned with contact structures on product manifolds of the form $\Sigma \times [0,1]$. In the case where $\Sigma$ is a disc, the contact geometry of such manifolds (and some surprisingly deep related algebraic and combinatorial structures) was studied in \cite{Me09Paper, Me10_Sutured_TQFT, Me12_elementary, Me16ContactCategories}.

For such a manifold, we will always describe its boundary as consisting of three parts: the \emph{top} $\Sigma \times \{1\}$, the \emph{bottom} $\Sigma \times \{0\}$, and the \emph{side} $\partial \Sigma \times [0,1]$. The $[0,1]$ direction is called \emph{vertical}: increasing and decreasing $[0,1]$ directions are \emph{up} and \emph{down}. We orient $\Sigma \times [0,1]$ so that the induced boundary orientation agrees with $\Sigma$ along $\Sigma \times \{1\}$, and disagrees along $\Sigma \times \{0\}$. There are corners along $\partial \Sigma \times \{0,1\}$. We consider contact structures on $\Sigma \times [0,1]$ with prescribed dividing sets on the top and bottom, and a vertical dividing set $\{\cdot\} \times [0,1]$ on the side. Because of the corners, it is natural to use a marked surface $(\Sigma, F, V)$, with top and bottom dividing sets having boundary $F$, and side dividing set having boundary $V$. This leads to the following definition.
\begin{defn}
\label{defn:M}
Let $\Gamma_0, \Gamma_1$ two dividing sets on the marked surface $(\Sigma,F,V)$. The sutured 3-manifold $M(\Gamma_0, \Gamma_1)$ is 
\[
M(\Gamma_0, \Gamma_1) = \Big( \Sigma \times [0,1], \; (-\Gamma_0) \times \{0\} \cup V \times [0,1] \cup \Gamma_1 \times \{1\} \Big).
\]
\end{defn}
Here $V \times [0,1]$ is oriented: $V_+ \times [0,1]$ is oriented upwards, and $V_- \times [0,1]$ downwards. On $\Sigma \times \{0\}$, the signs of $R_+, R_-$ are interchanged by the orientation reversal. The Euler class of the dividing set on $M(\Gamma_0, \Gamma_1)$ is $e(\Gamma_1) - e(\Gamma_0)$, so contact structures exist on $M(\Gamma_0, \Gamma_1)$ if and only if $e(\Gamma_0) = e(\Gamma_1)$.

Such contact structures can be \emph{stacked} as follows. Let $\Gamma_0, \Gamma_1, \Gamma_2$ be dividing sets on $(\Sigma, F)$, let $\xi_0$ be a contact structure on $M(\Gamma_0, \Gamma_1)$, and let $\xi_1$ be a contact structure on $M(\Gamma_1, \Gamma_2)$. Since the top dividing set of $\xi_0$ agrees with the bottom dividing set of $\xi_1$, we may glue these two faces together and obtain a contact structure on $\Sigma \times [0,2]$, which of course is homeomorphic to $\Sigma \times [0,1]$. Thus we obtain a contact structure on $M(\Gamma_0, \Gamma_2)$, which we say is obtained by \emph{stacking $\xi_1$ on $\xi_0$}.

The notion of a \emph{contact category} was introduced by Honda \cite{HonCat} and has been discussed in several papers \cite{Cooper15, Honda_Tian, Me09Paper, Me12_elementary, Me16ContactCategories}. We gave a rough idea of a contact category in the introduction; we now make it precise. The only subtleties are that we keep Euler classes separate (since the existence of a contact structure on $M(\Gamma_0, \Gamma_1)$ implies $e(\Gamma_0) = e(\Gamma_1)$); and overtwisted contact structures, regarded as trivial (their classification reduces to homotopy theory of plane fields \cite{ElOT}), are collapsed into ``zero" objects and morphisms. We refer to \emph{zero} and \emph{overtwisted} objects interchangeably; similarly to zero and overtwisted morphisms.
\begin{defn}
Let $(\Sigma,F)$ be a marked surface, and $e$ an integer such that $e \equiv I(\Sigma,F)$ mod $2$ and $|e| \leq I(\Sigma,F)$. The \emph{contact category of $(\Sigma,F)$ with Euler class $e$}, denoted $\Cat_e(\Sigma,F)$, consists of the following.
\begin{enumerate}
\item The objects are equivalence classes of tight dividing sets $\Gamma$ on $(\Sigma,F)$ such that $e(\Gamma) = e$, together with a zero object $*_e$.
\item The morphisms are as follows.
\begin{enumerate}
\item
If $\Gamma_0, \Gamma_1$ are tight objects, the morphisms $\Gamma_0 \To \Gamma_1$ are the isotopy classes of tight contact structures on $M(\Gamma_0, \Gamma_1)$, together with a zero morphism $\0_{\Gamma_0,\Gamma_1}$.
\item
If $X,Y$ are two objects, at least one of which is $*_e$, then there is a single zero morphism $\0_{X,Y} : X \To Y$.
\end{enumerate}
\item
The identity morphisms are as follows.
\begin{enumerate}
\item If $\Gamma$ is a tight object, the identity $1_\Gamma : \Gamma \To \Gamma$ is the isotopy class of the contact structure $\xi_\Gamma$ on $M(\Gamma, \Gamma)$ invariant in the $[0,1]$ direction.
\item The identity $1_* : *_e \To *_e$ is $\0_{*_e,*_e}$.
\end{enumerate}
\item
The composition of morphisms is as follows.
\begin{enumerate}
\item
The composition of two tight morphisms $\Gamma_0 \stackrel{\xi_0}{\To} \Gamma_1 \stackrel{\xi_1}{\To} \Gamma_2$ is the isotopy class of the contact structure on $M(\Gamma_0, \Gamma_2)$ given by stacking $\xi_0$ and $\xi_1$, if this contact structure is tight; otherwise it is $\0_{\Gamma_0,\Gamma_2}$.
\item
The composition of two morphisms $X \stackrel{f}{\To} Y \stackrel{g}{\To} Z$, where at least one of $f,g$ is a zero morphism, is $\0_{X,Z}$.
\end{enumerate}
\end{enumerate}
\end{defn}

\begin{defn}
The \emph{contact category} of a marked surface $(\Sigma,F)$, denoted $\Cat(\Sigma,F)$, is the disjoint union of the $\Cat_e (\Sigma,F)$.
\end{defn}
Nonzero objects or morphisms are also referred to as \emph{tight}. Following the standard abuse, we obscure the distinction between contact structures and their equivalence classes; hopefully no confusion will result.

Thus, the objects and morphisms are dividing sets and contact structures, but everything overtwisted becomes zero. Indeed $\Cat(\Sigma,F)$ is the ``quotient", in an appropriate sense, of an ``unreduced" contact category containing all contact structures, by an overtwisted subcategory; see \cite{Me16ContactCategories}.

Given a quadrangulation $Q$ of $(\Sigma,F,V)$, the basic dividing sets, being tight, form a subset of the tight objects of $\Cat(\Sigma,F)$; we call them \emph{basic objects} with respect to $Q$.
\begin{defn}
Let $(\Sigma,F,Q)$ be a quadrangulated marked surface.

If $e$ is an integer such that $e \equiv I(\Sigma,F)$ mod $2$ and $|e| \leq I(\Sigma,F)$, the \emph{(basic) contact category} of $(\Sigma,Q)$ of Euler class $e$, denoted $\Cat_e (\Sigma,Q)$, is the full subcategory of $\Cat_e (\Sigma,F)$ on the basic objects with respect to $Q$. 

The \emph{(basic) contact category} of $(\Sigma,Q)$, denoted $\Cat(\Sigma,Q)$, is the union of the $\Cat_e (\Sigma,Q)$.
\end{defn}

Given a category $\Cat$ and a base ring $R$ (commutative with $1$), we may form an $R$-algebra $R \Cat$, called the \emph{category algebra}, as follows. As an $R$-module, $R \Cat$ is free with basis given by the morphisms of $\Cat$. The product of two basis elements $f,g$ (morphisms) is then defined to be their composition in $\Cat$, if it is well-defined (i.e. $f,g$ are composable); otherwise the product is defined to be zero. Extending by linearity we obtain an associative $R$-algebra. 

For each object $X$ of $\Cat$, the identity morphism $1_X$ is an idempotent in $R \Cat$. For distinct objects $X,Y$, we have $1_X 1_Y = 1_Y 1_X = 0$, so these idempotents are orthogonal. If $\Cat$ has finitely many objects, then the sum of all identity morphisms is a multiplicative identity element of $R \Cat$. For two objects $X, Y$, the $R$-submodule $1_X R \Cat 1_Y$ has basis the morphisms $X \to Y$, and we have the decomposition $R\Cat = \bigoplus_{X,Y \in \Ob(\Cat)} 1_X \; R\Cat \; 1_Y$.
If $\Cat = \sqcup_e \Cat_e$ is a disjoint union of subcategories then we obtain $R \Cat = \bigoplus_e R \Cat_e$. See \cite{Webb07} for details. 

If $\Cat$ has some morphisms designated \emph{zero} morphisms, such that any composition of morphisms involving a zero morphism is also a zero morphism, then the $R$-submodule of $R\Cat$ generated by zero morphisms is a two-sided ideal, so we can take a quotient of $R\Cat$ by this ideal. The zero morphisms in fact become $0$ in the quotient algebra. 

We will be interested in the case when $R = \Z_2$ and $\Cat = \Cat(\Sigma,Q)$ or $\Cat_e (\Sigma, Q)$. On the quadrangulated surface $(\Sigma, Q)$ there are only finitely many (indeed precisely $2^{I(\Sigma,F)}$) basic dividing sets, hence $\Cat (\Sigma,Q)$ has finitely many objects. And for any basic $\Gamma_0, \Gamma_1$, there are only finitely many isotopy classes of tight contact structures on $M(\Gamma_0, \Gamma_1)$: this follows from the decomposition into cubes which we will discuss as we proceed, or see \cite{Colin_Giroux_Honda03, Colin_Giroux_Honda09} for general finiteness results. In any case, $\Cat(\Sigma,Q)$ contains finitely many morphisms, so $\Z_2 \Cat(\Sigma,Q)$ is finitely generated over $\Z_2$ and has a multiplicative identity. The overtwisted/zero morphisms have the property that any composition involving a zero morphism is also zero, so we make the following definition.
\begin{defn}
Let $(\Sigma,F,Q)$ be a quadrangulated marked surface and $e$ an integer such that $e \equiv I(\Sigma,F)$ mod $2$ and $|e| \leq I(\Sigma,F)$.
\begin{enumerate}
\item
The \emph{contact category algebra} of $(\Sigma,Q)$ of Euler class $e$, denoted $CA_e (\Sigma,Q)$, is the quotient of $\Z_2 \Cat_e (\Sigma,Q)$ by the ideal generated by overtwisted morphisms.
\item
The \emph{contact category algebra} of $(\Sigma,Q)$, denoted $CA (\Sigma,Q)$, is the quotient of $\Z_2 \Cat (\Sigma,Q)$ by the ideal generated by overtwisted morphisms.
\end{enumerate}
\end{defn}
Alternatively, $CA(\Sigma,Q)$ can be defined as the direct sum of the $\Z_2$-algebras $CA_e (\Sigma,Q)$.

The contact category algebra $CA(\Sigma,Q)$ is the algebra appearing in theorem \ref{thm:main_thm}. 

\subsection{From quadrangulations to cubulations}
\label{sec:cubulated_contact_structures}

Let $(\Sigma, Q)$ be a quadrangulated surface, with decomposing arcs $A_1, \ldots, A_j$ and squares $Q_1, \ldots, Q_k$. Then $\Sigma \times [0,1]$ is a union of \emph{cubes} $Q_i \times [0,1]$. Each cube $Q_i \times [0,1]$ has top and bottom boundary a square, and side boundary consisting of four square faces. Some of the side faces of cubes are part of the boundary of $\Sigma \times [0,1]$; others (namely the $A_i \times [0,1]$) are glued in pairs. We refer to this decomposition into cubes as the \emph{cubulation} of $\Sigma \times [0,1]$ corresponding to the quadrangulation $Q$, and denote it by $Q \times [0,1]$; we denote the \emph{cubulated 3-manifold} by $(\Sigma \times [0,1], Q \times [0,1])$ or $(\Sigma, Q) \times [0,1]$. Just as we can refer to the quadrangulation $Q$ by its arcs $A_i$ or squares $Q_i$, we can refer to the corresponding cubulation $Q \times [0,1]$ by its glued faces $A_i \times [0,1]$ or cubes $Q_i \times [0,1]$. Obviously there are many ways to glue cubes together to obtain a 3-manifold, but for present purposes a cubulation refers only to a $(\Sigma, Q) \times [0,1]$ for some quadrangulation $Q$ of $\Sigma$. For us, cubulations are just thickened quadrangulations.

We will think of our cubes as having convex boundary, but we will need to round and sharpen various corners. By default, when we refer to a cube, we actually mean a \emph{rounded cube}. If we Legendrian realise and then make a corner along the boundary of one of the six faces, we say the cube has that face \emph{sharpened}. We can sharpen any single face, and we can simultaneously sharpen two opposite faces, such as the top and bottom faces, but we will not sharpen any adjacent faces simultaneously.

A rounded cube is of course smooth and so it makes sense to speak of a smooth dividing set on its boundary. On each square face then we may draw the curves of a standard dividing set; joining these curves up across the rounded edges, we obtain a dividing set on the rounded cube. We call a cube with such a dividing set a \emph{standard cube}, or a cube with a \emph{standard dividing set}. If some faces of such a cube are sharpened, we still refer to the cube as standard. Note that adjacent vertices (or what remains of them after rounding) have opposite signs with respect to the dividing set.

When we sharpen a face of a standard cube, we always do so in such a way that the top and bottom dividing sets appear standard. The effect is shown in figure \ref{fig:cube_sharpening}.

\begin{figure}
\begin{center}
\begin{tikzpicture}[mypersp, scale=2]
\draw [yzp=1, rounded corners=5mm] (0,0) -- (0,1) -- (1,1) -- (1,0) -- cycle;
\draw [yzp=0, rounded corners=5mm] (0,0) -- (0,1) -- (1,1) -- (1,0) -- cycle;
\draw [xzp=1, rounded corners=5mm] (0,0) -- (0,1) -- (1,1) -- (1,0) -- cycle;
\draw [xzp=0, rounded corners=5mm] (0,0) -- (0,1) -- (1,1) -- (1,0) -- cycle;
\draw [xyp=1, rounded corners=5mm] (0,0) -- (0,1) -- (1,1) -- (1,0) -- cycle;
\draw [xyp=0, rounded corners=5mm] (0,0) -- (0,1) -- (1,1) -- (1,0) -- cycle;
\bottomoff
\backclean
\leftclean
\frontclean
\draw [red, ultra thick, yzp=1] (0.5,0) arc (0:90:0.5);
\draw [red, ultra thick, yzp=1] (0.5,1) arc (-180:-90:0.5);
\topon
\end{tikzpicture}
\begin{tikzpicture}[mypersp, scale=2]
\draw [xyp=1, rounded corners=5mm] (0,0) -- (0,1) -- (1,1) -- (1,0) -- cycle;
\draw [xyp=0, rounded corners=5mm] (0,0) -- (0,1) -- (1,1) -- (1,0) -- cycle;
\draw [xzp=1] (0.2,0) -- (0.2,1);
\draw [xzp=1] (0.9,-0.1) -- (0.9,0.9);
\draw [xzp=0] (0.1,0.1) -- (0.1,1.1);
\draw [xzp=0] (0.95,0.05) -- (0.95,1.05);
\draw [red, ultra thick, xzp=0] (0.1,0.1) -- (0.1,1.1);
\draw [red, ultra thick, densely dotted, xzp=1] (0.2,0) -- (0.2,1);
\draw [red, ultra thick, yzp=1] (0.3,0) to [bend left=90] (0.7,0);
\draw [red, ultra thick, yzp=1] (0.3,1) to [bend right=90] (0.7,1);
\topon
\bottomoff
\end{tikzpicture}
\begin{tikzpicture}[mypersp, scale=2]
\draw [yzp=1, rounded corners=5mm] (-0.05,-0.1) -- (-0.05,1.1) -- (1.05,1.1) -- (1.05,-0.1) -- cycle;
\draw [yzp=0, rounded corners=5mm] (0,0) -- (0,1) -- (1,1) -- (1,0) -- cycle;
\draw [xzp=1, rounded corners=5mm] (1,1) -- (0,1) -- (0,0) -- (1,0);
\draw [xzp=0, rounded corners=5mm] (1,1) -- (0,1) -- (0,0) -- (1,0);
\draw [xyp=1, rounded corners=5mm] (1,1) -- (0,1) -- (0,0) -- (1,0);
\draw [xyp=0, rounded corners=5mm] (1,1) -- (0,1) -- (0,0) -- (1,0);
\draw [red, ultra thick, xyp=1] (0.5,0) arc (180:90:0.4);
\draw [red, ultra thick, xyp=1] (0.5,1) arc (0:-90:0.5);
\draw [red, ultra thick, densely dotted, xyp=0] (0.5,0) arc (0:90:0.5);
\draw [red, ultra thick, densely dotted, xyp=0] (0.5,1) arc (-180:-90:0.55);
\leftclean
\backclean
\draw [red, ultra thick, xzp=0] (0.5,0) arc (180:90:0.48);
\draw [red, ultra thick, xzp=0] (0.5,1) arc (0:-90:0.5);
\draw [red, ultra thick, yzp=1] (0,0) to [bend left=45] (1,0);
\draw [red, ultra thick, yzp=1] (0,1) to [bend right=45] (1,1);
\end{tikzpicture}

\caption{Left: A standard rounded cube. Centre: Top and bottom faces are sharpened. Right: The right side face is sharpened.}
\label{fig:cube_sharpening}
\end{center}
\end{figure}
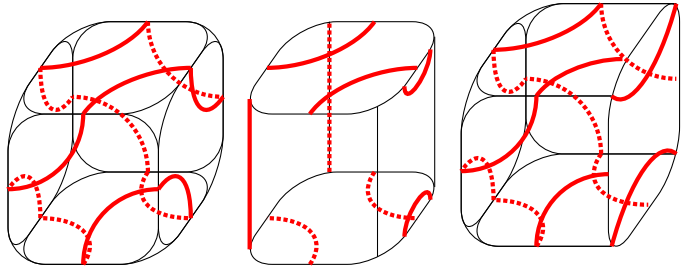

A rounded standard cube has a dividing set on its boundary sphere, which may or may not be connected, or equivalently, tight; we will investigate which standard cubes are tight in section \ref{sec:tightness_of_cubes}. Figure \ref{fig:more_cube_examples} shows some possibilities. Since it is topologically a 3-ball, a standard cube with tight dividing set has a unique isotopy class of tight contact structure \cite{ElMartinet}. 

\begin{defn}
Let $(\Sigma,Q)$ be a quadrangulated surface, so that $(\Sigma, Q) \times [0,1]$ is a cubulated 3-manifold. A \emph{cubulated contact structure} on $(\Sigma,Q) \times [0,1]$ is a contact structure such that each cube $Q_i \times [0,1]$ has a standard dividing set.
\end{defn}
Since the top face of a cube in a cubulated contact structure is a square with a standard dividing set, the dividing set on $\Sigma \times \{1\}$ is basic with respect to $Q$; and since the bottom faces are also standard squares, the dividing set on $\Sigma \times \{0\}$ is also basic with respect to $Q$.

There is a technical issue with this definition, which should be mentioned. Each $Q_i \times [0,1]$ is a bona fide cube with 12 corner edges and 8 corner vertices, but in our scheme of rounding and sharpening faces of cubes we only allow sharpening along simple closed curves. However, we can successively round corners as we decompose $\Sigma \times [0,1]$, so that in the end every cube is rounded, and we can make sense of a dividing set being standard. Precisely, we begin by rounding the corners $\partial \Sigma \times \{0,1\}$ of $\Sigma \times [0,1]$. Then we can cut along a rounded $Q_i \times [0,1]$ and obtain a manifold with two corners along the smooth curves given by the rounded boundary of $Q_i \times [0,1]$. We round these corners, and cut again. By rounding corners at each stage, we eventually  obtain rounded cubes. This process is illustrated in figure \ref{fig:cube_gluing}.

As it turns out, \emph{all} tight contact structures with basic dividing sets are cubulated.
\begin{lem}
\label{lem:all_contact_strs_cubulated}
Let $(\Sigma, Q)$ be a quadrangulated surface, let $\Gamma_0, \Gamma_1$ be basic dividing sets, and let $\xi$ be a tight contact structure on $M(\Gamma_0, \Gamma_1)$. Then $\xi$ is isotopic to a cubulated contact structure on $(\Sigma,Q) \times [0,1]$ where every cube has a tight contact structure.
\end{lem}

\begin{proof}
As $\Gamma_0, \Gamma_1$ are tight, each cube of $Q \times [0,1]$ has a standard dividing set on its top and bottom faces. Rounding the corners $\partial \Sigma \times \{0,1\}$ of $M(\Gamma_0, \Gamma_1)$, the vertical dividing set on the side boundary of $M(\Gamma_0, \Gamma_1)$ naturally provides each unglued side face with a standard dividing set, as illustrated in figure \ref{fig:vertical_side_boundary}.

\begin{figure}
\begin{center}
\begin{tikzpicture}[mypersp, scale=2]
\draw [xyp=1, rounded corners=5mm] (0,0) -- (0,1) -- (2,1) -- (2,0) -- cycle;
\draw [xyp=0, rounded corners=5mm] (0,0) -- (0,1) -- (2,1) -- (2,0) -- cycle;
\draw [yzp=1] (0,0) -- (0,1) -- (1,1) -- (1,0) -- cycle;
\draw [red, ultra thick, densely dotted, xzp=1] (0.2,0) -- (0.2,1);
\draw [red, ultra thick, densely dotted, xzp=1] (1,0) -- (1,1);
\draw [red, ultra thick, xzp=1] (1.9,-0.1) -- (1.9,0.9); 
\draw [red, ultra thick, xzp=0] (1.95,0.05) -- (1.95,1.05);
\draw [red, ultra thick, xzp=0] (0.1,0.1) -- (0.1,1.1);
\draw [red, ultra thick, xzp=0] (1,0) -- (1,1);
\draw [red, ultra thick, densely dotted, xyp=0] (0.5,0) -- (0.5,0.3);
\draw [red, ultra thick, densely dotted, xyp=0] (0,0.5) -- (0.3,0.5);
\draw [red, ultra thick, densely dotted, xyp=0] (0.5,0.7) -- (0.5,1);
\draw [red, ultra thick, xyp=1] (0.5,0) -- (0.5,0.3);
\draw [red, ultra thick, xyp=1] (0,0.5) -- (0.3,0.5);
\draw [red, ultra thick, xyp=1] (0.5,0.7) -- (0.5,1);
\begin{scope}[xshift=1cm]
\draw [red, ultra thick, densely dotted, xyp=0] (0.5,0.7) -- (0.5,1);
\draw [red, ultra thick, densely dotted, xyp=0] (0.7,0.5) -- (1,0.5);
\draw [red, ultra thick, densely dotted, xyp=0] (0.5,0) -- (0.5,0.3);
\draw [red, ultra thick, xyp=1] (0.5,0.7) -- (0.5,1);
\draw [red, ultra thick, xyp=1] (0.7,0.5) -- (1,0.5);
\draw [red, ultra thick, xyp=1] (0.5,0) -- (0.5,0.3);
\end{scope}
\draw [<->] (2.2,1.5) -- (3,1.5);
\begin{scope}[xshift=4cm]
\draw [yzp=1, rounded corners=5mm] (-0.05,-0.1) -- (-0.05,1.1) -- (1.05,1.1) -- (1.05,-0.1) -- cycle;
\draw [yzp=0, rounded corners=5mm] (0,0) -- (0,1) -- (1,1) -- (1,0) -- cycle;
\draw [xzp=1, rounded corners=5mm] (1,1) -- (0,1) -- (0,0) -- (1,0);
\draw [xzp=0, rounded corners=5mm] (1,1) -- (0,1) -- (0,0) -- (1,0);
\draw [xyp=1, rounded corners=5mm] (1,1) -- (0,1) -- (0,0) -- (1,0);
\draw [xyp=0, rounded corners=5mm] (1,1) -- (0,1) -- (0,0) -- (1,0);
\draw [yzp=2, rounded corners=5mm] (0,0) -- (0,1) -- (1,1) -- (1,0) -- cycle;
\draw [xzp=1, rounded corners=5mm] (1,1) -- (2,1) -- (2,0) -- (1,0);
\draw [xzp=0, rounded corners=5mm] (1,1) -- (2,1) -- (2,0) -- (1,0);
\draw [xyp=1, rounded corners=5mm] (1,1) -- (2,1) -- (2,0) -- (1,0);
\draw [xyp=0, rounded corners=5mm] (1,1) -- (2,1) -- (2,0) -- (1,0);
\backclean
\leftclean
\frontclean
\draw [red, ultra thick, densely dotted, xyp=0] (0.5,0) -- (0.5,0.3);
\draw [red, ultra thick, densely dotted, xyp=0] (0,0.5) -- (0.3,0.5);
\draw [red, ultra thick, densely dotted, xyp=0] (0.5,0.7) -- (0.5,1);
\draw [red, ultra thick, xyp=1] (0.5,0) -- (0.5,0.3);
\draw [red, ultra thick, xyp=1] (0,0.5) -- (0.3,0.5);
\draw [red, ultra thick, xyp=1] (0.5,0.7) -- (0.5,1);
\end{scope}
\begin{scope}[xshift=5cm]
\backclean
\rightclean
\frontclean
\draw [red, ultra thick, densely dotted, xyp=0] (0.5,0.7) -- (0.5,1);
\draw [red, ultra thick, densely dotted, xyp=0] (0.7,0.5) -- (1,0.5);
\draw [red, ultra thick, densely dotted, xyp=0] (0.5,0) -- (0.5,0.3);
\draw [red, ultra thick, xyp=1] (0.5,0.7) -- (0.5,1);
\draw [red, ultra thick, xyp=1] (0.7,0.5) -- (1,0.5);
\draw [red, ultra thick, xyp=1] (0.5,0) -- (0.5,0.3);
\end{scope}
\end{tikzpicture}
\caption{Rounding corners, the vertical dividing set on the side boundary of $M(\Gamma_0, \Gamma_1)$ becomes a standard unused dividing set on each side face.}
\label{fig:vertical_side_boundary}
\end{center}
\end{figure}
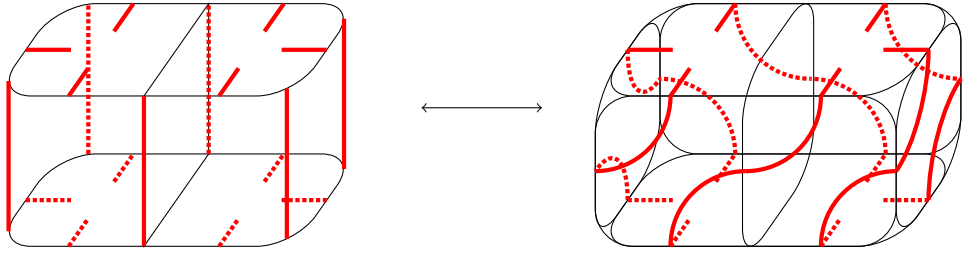

Consider a decomposing arc $A$; so $A \times I$ is a glued face of the cubulation. After rounding corners, $\partial (A \times I)$ intersects the dividing set on $\partial M(\Gamma_0, \Gamma_1)$ in four points, so can be Legendrian realised. Then $A \times I$ (suitably rounded), after a small isotopy, can be made convex. The dividing set on (rounded) $A \times I$ has four endpoints and, as $\xi$ is tight, must consists of two arcs. After cutting along this surface and rounding edges, as in figure \ref{fig:cube_gluing}, we have standard dividing sets on the resulting side faces.

Proceeding in this way, we cut $M(\Gamma_0, \Gamma_1)$ along rounded surfaces, close to $A \times I$, into a collection of standard cubes. As $\xi$ is tight, each standard cube has a tight contact structure. 
\end{proof}

A priori, a cubulated contact structure might be overtwisted, even though all cubes have tight contact structures. However, as we will see next, the tightness of each cube implies the tightness of the entire contact structure.

We have seen how a cubulated $(\Sigma, Q) \times [0,1]$ can be cut into standard cubes, as illustrated in figure \ref{fig:cube_gluing}, read from bottom to top. If we read it from top to bottom, reversing the process, we have a recipe for gluing faces of standard cubes.

To glue two rounded standard cubes along a pair of faces, we sharpen those faces, and then glue them together, identifying dividing sets, so as to obtain a manifold with smooth boundary.

As illustrated at the top of figure \ref{fig:cube_gluing}, two faces of rounded standard cubes can be glued together if and only if the dividing sets on those faces appear to \emph{disagree!} Effectively they differ by a $90^\circ$ rotation, but in the sharpening process the dividing sets are each shifted $45^\circ$, so disagreement of dividing sets on a rounded cube implies agreement once those faces are sharpened. If we are looking at rounded standard cubes, we say two sides can be \emph{validly identified} if their dividing sets precisely disagree in this way; sharpening the faces and gluing, we call a \emph{valid identification} of faces.

\begin{figure}
\begin{center}
\begin{tikzpicture}[mypersp, scale=1.5]
\draw [yzp=1, rounded corners=5mm] (0,0) -- (0,1) -- (1,1) -- (1,0) -- cycle;
\draw [yzp=0, rounded corners=5mm] (0,0) -- (0,1) -- (1,1) -- (1,0) -- cycle;
\draw [xzp=1, rounded corners=5mm] (0,0) -- (0,1) -- (1,1) -- (1,0) -- cycle;
\draw [xzp=0, rounded corners=5mm] (0,0) -- (0,1) -- (1,1) -- (1,0) -- cycle;
\draw [xyp=1, rounded corners=5mm] (0,0) -- (0,1) -- (1,1) -- (1,0) -- cycle;
\draw [xyp=0, rounded corners=5mm] (0,0) -- (0,1) -- (1,1) -- (1,0) -- cycle;
\draw [red, ultra thick, densely dotted, xyp=0] (1,0.5) -- (0.7,0.5);
\draw [red, ultra thick, xyp=1] (1,0.5) -- (0.7,0.5);
\draw [red, ultra thick, xzp=0] (1,0.5) -- (0.7,0.5);
\draw [red, ultra thick, densely dotted, xzp=1] (1,0.5) -- (0.7,0.5);
\fill [gray!30!white, draw=none, opacity=0.5, yzp=1, rounded corners=5mm] (0,0) -- (0,1) -- (1,1) -- (1,0) -- cycle;
\draw [red, ultra thick, yzp=1] (0.5,0) arc (180:90:0.5);
\draw [red, ultra thick, yzp=1] (0.5,1) arc (0:-90:0.5);
\begin{scope}[xshift=2 cm]
\fill [gray!30!white, draw=none, opacity=0.5, yzp=0, rounded corners=5mm] (0,0) -- (0,1) -- (1,1) -- (1,0) -- cycle;
\draw [yzp=1, rounded corners=5mm] (0,0) -- (0,1) -- (1,1) -- (1,0) -- cycle;
\draw [yzp=0, rounded corners=5mm] (0,0) -- (0,1) -- (1,1) -- (1,0) -- cycle;
\draw [xzp=1, rounded corners=5mm] (0,0) -- (0,1) -- (1,1) -- (1,0) -- cycle;
\draw [xzp=0, rounded corners=5mm] (0,0) -- (0,1) -- (1,1) -- (1,0) -- cycle;
\draw [xyp=1, rounded corners=5mm] (0,0) -- (0,1) -- (1,1) -- (1,0) -- cycle;
\draw [xyp=0, rounded corners=5mm] (0,0) -- (0,1) -- (1,1) -- (1,0) -- cycle;
\draw [red, ultra thick, densely dotted, xyp=0] (0,0.5) -- (0.3,0.5);
\draw [red, ultra thick, densely dotted, xzp=1] (0,0.5) -- (0.3,0.5);
\draw [red, ultra thick, xyp=1] (0,0.5) -- (0.3,0.5);
\draw [red, ultra thick, xzp=0] (0,0.5) -- (0.3,0.5);
\draw [red, ultra thick, yzp=0] (0.5,0) arc (0:90:0.5);
\draw [red, ultra thick, yzp=0] (0.5,1) arc (-180:-90:0.5);
\end{scope}
\begin{scope}[xshift = 5 cm]
\draw [yzp=1, rounded corners=5mm] (0,0) -- (0,1) -- (1,1) -- (1,0) -- cycle;
\draw [yzp=0, rounded corners=5mm] (0,0) -- (0,1) -- (1,1) -- (1,0) -- cycle;
\draw [xzp=1, rounded corners=5mm] (0,0) -- (0,1) -- (1,1) -- (1,0) -- cycle;
\draw [xzp=0, rounded corners=5mm] (0,0) -- (0,1) -- (1,1) -- (1,0) -- cycle;
\draw [xyp=1, rounded corners=5mm] (0,0) -- (0,1) -- (1,1) -- (1,0) -- cycle;
\draw [xyp=0, rounded corners=5mm] (0,0) -- (0,1) -- (1,1) -- (1,0) -- cycle;
\draw [red, ultra thick, densely dotted, xyp=0] (1,0.5) -- (0.7,0.5);
\draw [red, ultra thick, xyp=1] (1,0.5) -- (0.7,0.5);
\draw [red, ultra thick, xzp=0] (1,0.5) -- (0.7,0.5);
\draw [red, ultra thick, densely dotted, xzp=1] (1,0.5) -- (0.7,0.5);
\fill [gray!30!white, draw=none, opacity=0.5, yzp=1, rounded corners=5mm] (0,0) -- (0,1) -- (1,1) -- (1,0) -- cycle;
\draw [red, ultra thick, yzp=1] (0.5,0) arc (0:90:0.5);
\draw [red, ultra thick, yzp=1] (0.5,1) arc (-180:-90:0.5);
\end{scope}
\begin{scope}[xshift = 7 cm]
\fill [gray!30!white, draw=none, opacity=0.5, yzp=0, rounded corners=5mm] (0,0) -- (0,1) -- (1,1) -- (1,0) -- cycle;
\draw [yzp=1, rounded corners=5mm] (0,0) -- (0,1) -- (1,1) -- (1,0) -- cycle;
\draw [yzp=0, rounded corners=5mm] (0,0) -- (0,1) -- (1,1) -- (1,0) -- cycle;
\draw [xzp=1, rounded corners=5mm] (0,0) -- (0,1) -- (1,1) -- (1,0) -- cycle;
\draw [xzp=0, rounded corners=5mm] (0,0) -- (0,1) -- (1,1) -- (1,0) -- cycle;
\draw [xyp=1, rounded corners=5mm] (0,0) -- (0,1) -- (1,1) -- (1,0) -- cycle;
\draw [xyp=0, rounded corners=5mm] (0,0) -- (0,1) -- (1,1) -- (1,0) -- cycle;
\draw [red, ultra thick, densely dotted, xyp=0] (0,0.5) -- (0.3,0.5);
\draw [red, ultra thick, densely dotted, xzp=1] (0,0.5) -- (0.3,0.5);
\draw [red, ultra thick, xyp=1] (0,0.5) -- (0.3,0.5);
\draw [red, ultra thick, xzp=0] (0,0.5) -- (0.3,0.5);
\draw [red, ultra thick, yzp=0] (0.5,0) arc (180:90:0.5);
\draw [red, ultra thick, yzp=0] (0.5,1) arc (0:-90:0.5);
\end{scope}

\begin{scope}[yshift = -2 cm]
\draw [yzp=1, rounded corners=5mm] (-0.05,-0.1) -- (-0.05,1.1) -- (1.05,1.1) -- (1.05,-0.1) -- cycle;
\draw [yzp=0, rounded corners=5mm] (0,0) -- (0,1) -- (1,1) -- (1,0) -- cycle;
\draw [xzp=1, rounded corners=5mm] (1,1) -- (0,1) -- (0,0) -- (1,0);
\draw [xzp=0, rounded corners=5mm] (1,1) -- (0,1) -- (0,0) -- (1,0);
\draw [xyp=1, rounded corners=5mm] (1,1) -- (0,1) -- (0,0) -- (1,0);
\draw [xyp=0, rounded corners=5mm] (1,1) -- (0,1) -- (0,0) -- (1,0);
\draw [red, ultra thick, densely dotted, xyp=0] (1,0.5) -- (0.7,0.5);
\draw [red, ultra thick, xyp=1] (1,0.5) -- (0.7,0.5);
\draw [red, ultra thick, xzp=0] (1,0.5) -- (0.7,0.5);
\draw [red, ultra thick, densely dotted, xzp=1] (1,0.5) -- (0.7,0.5);
\fill [gray!30!white, draw=none, yzp=1, opacity=0.5, rounded corners=5mm] (-0.05,-0.1) -- (-0.05,1.1) -- (1.05,1.1) -- (1.05,-0.1) -- cycle;
\draw [red, ultra thick, yzp=1] (0,0) to [bend right=45] (0,1);
\draw [red, ultra thick, yzp=1] (1,0) to [bend left=45] (1,1);
\end{scope}
\begin{scope}[yshift = -2 cm, xshift = 2 cm]
\fill [gray!30!white, draw=none, yzp=0, opacity=0.5, rounded corners=5mm] (-0.05,-0.1) -- (-0.05,1.1) -- (1.05,1.1) -- (1.05,-0.1) -- cycle;
\draw [yzp=0, rounded corners=5mm] (-0.05,-0.1) -- (-0.05,1.1) -- (1.05,1.1) -- (1.05,-0.1) -- cycle;
\draw [yzp=1, rounded corners=5mm] (0,0) -- (0,1) -- (1,1) -- (1,0) -- cycle;
\draw [xzp=1, rounded corners=5mm] (0,1) -- (1,1) -- (1,0) -- (0,0);
\draw [xzp=0, rounded corners=5mm] (0,1) -- (1,1) -- (1,0) -- (0,0);
\draw [xyp=1, rounded corners=5mm] (0,1) -- (1,1) -- (1,0) -- (0,0);
\draw [xyp=0, rounded corners=5mm] (0,1) -- (1,1) -- (1,0) -- (0,0);
\draw [red, ultra thick, densely dotted, xyp=0] (0,0.5) -- (0.3,0.5);
\draw [red, ultra thick, xyp=1] (0,0.5) -- (0.3,0.5);
\draw [red, ultra thick, xzp=0] (0,0.5) -- (0.3,0.5);
\draw [red, ultra thick, densely dotted, xzp=1] (0,0.5) -- (0.3,0.5);
\draw [red, ultra thick, yzp=0] (0,0) to [bend right=45] (0,1);
\draw [red, ultra thick, yzp=0] (1,0) to [bend left=45] (1,1);
\end{scope}
\begin{scope}[yshift = -2cm, xshift = 5cm]
\draw [yzp=1, rounded corners=5mm] (-0.05,-0.1) -- (-0.05,1.1) -- (1.05,1.1) -- (1.05,-0.1) -- cycle;
\draw [yzp=0, rounded corners=5mm] (0,0) -- (0,1) -- (1,1) -- (1,0) -- cycle;
\draw [xzp=1, rounded corners=5mm] (1,1) -- (0,1) -- (0,0) -- (1,0);
\draw [xzp=0, rounded corners=5mm] (1,1) -- (0,1) -- (0,0) -- (1,0);
\draw [xyp=1, rounded corners=5mm] (1,1) -- (0,1) -- (0,0) -- (1,0);
\draw [xyp=0, rounded corners=5mm] (1,1) -- (0,1) -- (0,0) -- (1,0);
\draw [red, ultra thick, densely dotted, xyp=0] (1,0.5) -- (0.7,0.5);
\draw [red, ultra thick, xyp=1] (1,0.5) -- (0.7,0.5);
\draw [red, ultra thick, xzp=0] (1,0.5) -- (0.7,0.5);
\draw [red, ultra thick, densely dotted, xzp=1] (1,0.5) -- (0.7,0.5);
\fill [gray!30!white, draw=none, yzp=1, opacity=0.5, rounded corners=5mm] (-0.05,-0.1) -- (-0.05,1.1) -- (1.05,1.1) -- (1.05,-0.1) -- cycle;
\draw [red, ultra thick, yzp=1] (0,0) to [bend left=45] (1,0);
\draw [red, ultra thick, yzp=1] (0,1) to [bend right=45] (1,1);
\end{scope}
\begin{scope}[yshift = -2 cm, xshift = 7 cm]
\fill [gray!30!white, draw=none, yzp=0, opacity=0.5, rounded corners=5mm] (-0.05,-0.1) -- (-0.05,1.1) -- (1.05,1.1) -- (1.05,-0.1) -- cycle;
\draw [yzp=0, rounded corners=5mm] (-0.05,-0.1) -- (-0.05,1.1) -- (1.05,1.1) -- (1.05,-0.1) -- cycle;
\draw [yzp=1, rounded corners=5mm] (0,0) -- (0,1) -- (1,1) -- (1,0) -- cycle;
\draw [xzp=1, rounded corners=5mm] (0,1) -- (1,1) -- (1,0) -- (0,0);
\draw [xzp=0, rounded corners=5mm] (0,1) -- (1,1) -- (1,0) -- (0,0);
\draw [xyp=1, rounded corners=5mm] (0,1) -- (1,1) -- (1,0) -- (0,0);
\draw [xyp=0, rounded corners=5mm] (0,1) -- (1,1) -- (1,0) -- (0,0);
\draw [red, ultra thick, densely dotted, xyp=0] (0,0.5) -- (0.3,0.5);
\draw [red, ultra thick, xyp=1] (0,0.5) -- (0.3,0.5);
\draw [red, ultra thick, xzp=0] (0,0.5) -- (0.3,0.5);
\draw [red, ultra thick, densely dotted, xzp=1] (0,0.5) -- (0.3,0.5);
\draw [red, ultra thick, yzp=0] (0,0) to [bend left=45] (1,0);
\draw [red, ultra thick, yzp=0] (0,1) to [bend right=45] (1,1);
\end{scope}

\begin{scope}[yshift = -4 cm, xshift = 0.5 cm]
\draw [yzp=1, rounded corners=5mm] (-0.05,-0.1) -- (-0.05,1.1) -- (1.05,1.1) -- (1.05,-0.1) -- cycle;
\draw [yzp=0, rounded corners=5mm] (0,0) -- (0,1) -- (1,1) -- (1,0) -- cycle;
\draw [xzp=1, rounded corners=5mm] (1,1) -- (0,1) -- (0,0) -- (1,0);
\draw [xzp=0, rounded corners=5mm] (1,1) -- (0,1) -- (0,0) -- (1,0);
\draw [xyp=1, rounded corners=5mm] (1,1) -- (0,1) -- (0,0) -- (1,0);
\draw [xyp=0, rounded corners=5mm] (1,1) -- (0,1) -- (0,0) -- (1,0);
\draw [xzp=1, rounded corners=5mm] (1,1) -- (2,1) -- (2,0) -- (1,0);
\fill [gray!30!white, draw=none, yzp=1, opacity=0.5, rounded corners=5mm] (-0.05,-0.1) -- (-0.05,1.1) -- (1.05,1.1) -- (1.05,-0.1) -- cycle;
\draw [red, ultra thick, densely dotted, xyp=0] (1,0.5) -- (0.7,0.5);
\draw [red, ultra thick, xyp=1] (1,0.5) -- (0.7,0.5);
\draw [red, ultra thick, xzp=0] (1,0.5) -- (0.7,0.5);
\draw [red, ultra thick, densely dotted, xzp=1] (1,0.5) -- (0.7,0.5);
\draw [red, ultra thick, yzp=1] (0,0) to [bend right=45] (0,1);
\draw [red, ultra thick, yzp=1] (1,0) to [bend left=45] (1,1);
\draw [yzp=2, rounded corners=5mm] (0,0) -- (0,1) -- (1,1) -- (1,0) -- cycle;
\draw [xzp=0, rounded corners=5mm] (1,1) -- (2,1) -- (2,0) -- (1,0);
\draw [xyp=1, rounded corners=5mm] (1,1) -- (2,1) -- (2,0) -- (1,0);
\draw [xyp=0, rounded corners=5mm] (1,1) -- (2,1) -- (2,0) -- (1,0);
\draw [red, ultra thick, densely dotted, xyp=0] (1,0.5) -- (1.3,0.5);
\draw [red, ultra thick, xyp=1] (1,0.5) -- (1.3,0.5);
\draw [red, ultra thick, xzp=0] (1,0.5) -- (1.3,0.5);
\draw [red, ultra thick, densely dotted, xzp=1] (1,0.5) -- (1.3,0.5);
\end{scope}
\begin{scope}[yshift = -4 cm, xshift = 5.5 cm]
\draw [yzp=1, rounded corners=5mm] (-0.05,-0.1) -- (-0.05,1.1) -- (1.05,1.1) -- (1.05,-0.1) -- cycle;
\draw [yzp=0, rounded corners=5mm] (0,0) -- (0,1) -- (1,1) -- (1,0) -- cycle;
\draw [xzp=1, rounded corners=5mm] (1,1) -- (0,1) -- (0,0) -- (1,0);
\draw [xzp=0, rounded corners=5mm] (1,1) -- (0,1) -- (0,0) -- (1,0);
\draw [xyp=1, rounded corners=5mm] (1,1) -- (0,1) -- (0,0) -- (1,0);
\draw [xyp=0, rounded corners=5mm] (1,1) -- (0,1) -- (0,0) -- (1,0);
\draw [red, ultra thick, densely dotted, xzp=1] (1,0.5) -- (1.3,0.5);
\fill [gray!30!white, draw=none, yzp=1, opacity=0.5, rounded corners=5mm] (-0.05,-0.1) -- (-0.05,1.1) -- (1.05,1.1) -- (1.05,-0.1) -- cycle;
\draw [red, ultra thick, densely dotted, xyp=0] (1,0.5) -- (0.7,0.5);
\draw [red, ultra thick, xyp=1] (1,0.5) -- (0.7,0.5);
\draw [red, ultra thick, xzp=0] (1,0.5) -- (0.7,0.5);
\draw [red, ultra thick, densely dotted, xzp=1] (1,0.5) -- (0.7,0.5);
\draw [red, ultra thick, yzp=1] (0,0) to [bend left=45] (1,0);
\draw [red, ultra thick, yzp=1] (0,1) to [bend right=45] (1,1);
\draw [yzp=2, rounded corners=5mm] (0,0) -- (0,1) -- (1,1) -- (1,0) -- cycle;
\draw [xzp=1, rounded corners=5mm] (1,1) -- (2,1) -- (2,0) -- (1,0);
\draw [xzp=0, rounded corners=5mm] (1,1) -- (2,1) -- (2,0) -- (1,0);
\draw [xyp=1, rounded corners=5mm] (1,1) -- (2,1) -- (2,0) -- (1,0);
\draw [xyp=0, rounded corners=5mm] (1,1) -- (2,1) -- (2,0) -- (1,0);
\draw [red, ultra thick, densely dotted, xyp=0] (1,0.5) -- (1.3,0.5);
\draw [red, ultra thick, xyp=1] (1,0.5) -- (1.3,0.5);
\draw [red, ultra thick, xzp=0] (1,0.5) -- (1.3,0.5);
\end{scope}
\end{tikzpicture}
\caption{Cutting and gluing side faces of cubes. From top to bottom is the gluing process; from bottom to top, the cutting process. Left: the face is unused. Right: the face is used. Top: the faces are rounded. Middle: the faces are sharpened. Bottom: the faces are identified.}
\label{fig:cube_gluing}
\end{center}
\end{figure}
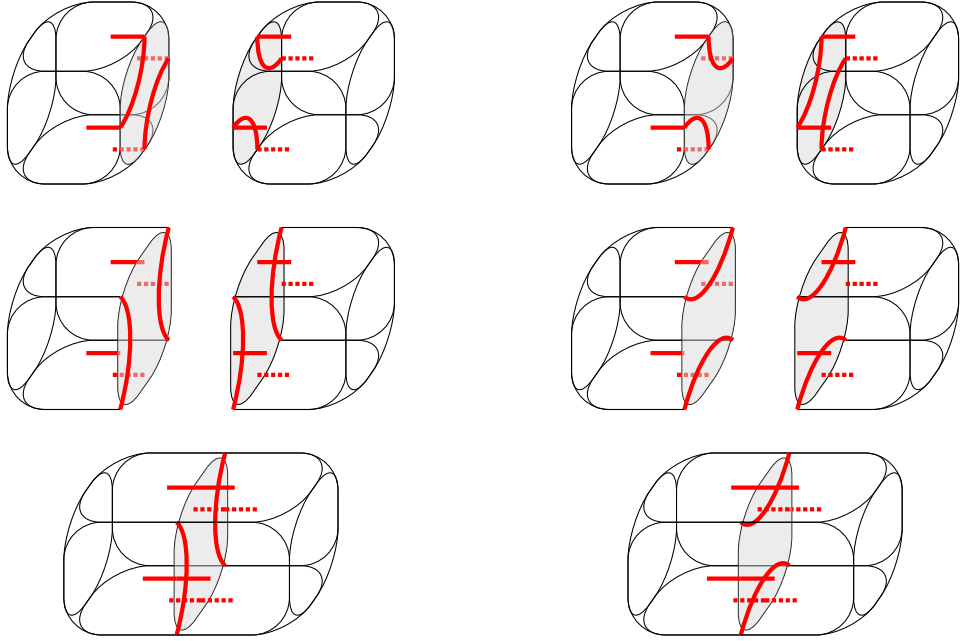

Given a collection of standard cubes, we may successively glue together their side faces in pairs by valid identifications. We may thus obtain a cubulated 3-manifold $(\Sigma, Q) \times [0,1]$. If each standard cube has a contact structure, we obtain a cubulated contact structure on $(\Sigma, Q) \times [0,1]$ where the top and bottom dividing sets are basic with respect to $Q$.

It turns out that if each cube is tight, then the result is tight, as we prove now.
\begin{lem}
\label{lem:cubulated_implies_tight}
Let $(\Sigma, Q)$ be a quadrangulated surface, and let $\xi$ be a cubulated contact structure on $(\Sigma, Q) \times [0,1]$, where each cube has a tight contact structure. Then $\xi$ is tight.
\end{lem}

\begin{proof}
Letting the dividing sets on $\Sigma \times \{0,1\}$ be $\Gamma_0, \Gamma_1$, we see $\xi$ is a contact structure on $M(\Gamma_0, \Gamma_1)$. As discussed above, we can cut $M(\Gamma_0, \Gamma_1)$ into a collection of standard tight cubes. We show that we can glue the cubes back together so that the contact structure remains tight, by applying Honda's theorem on gluing contact structures \cite{Hon02}. The idea is that the glued faces are ``too small" for any bypasses to pass through them.

Suppose we have a tight contact manifold $(M', \eta')$, consisting of some tight standard cubes, with some faces glued by valid identifications. We then glue two further side faces together, by a valid identification, to obtain a contact manifold $(M, \eta)$; we will show $\eta$ is tight. Both $M$ and $M'$ have dividing sets on their boundary; they are compact oriented irreducible sutured 3-manifolds which we denote $(M, \Gamma)$ and $(M',\Gamma')$. In $M$ the glued face forms an incompressible surface $S$ with Legendrian boundary intersecting the dividing set in four points. Moreover $S$ has a dividing set $\Gamma_S$ which interleaves with that on $\partial M'$, as shown in figure \ref{fig:cube_gluing}; it consists of two arcs. The dividing set $\Gamma_S$ on $S$, together with the (isotopy class of) tight contact structure $\eta$ on $M$, forms a potentially allowable configuration $(\Gamma_S, \eta')$, in the sense of \cite{Hon02}.

As the dividing set $\Gamma_S$ on $S$ consists of two arcs, there are no nontrivial bypass surgeries on it. Hence there are no nontrivial state transitions possible from $(\Gamma_S, \eta')$ to other configurations. 
So $(\Gamma_S, \eta)$ is an isolated vertex in the configuration graph, and hence by the gluing theorem, the result $\eta$ of gluing $\eta'$ along $S$ is tight.

In other words, as we glue two cubes together by a valid identification, if we had a tight contact structure beforehand, we still have one after gluing. Hence if each cube is tight, the cubulated contact structure $\xi$ obtained on $(\Sigma, Q) \times I$ is tight.
\end{proof}

Lemmas \ref{lem:all_contact_strs_cubulated} and \ref{lem:cubulated_implies_tight} are converses; together they immediately yield the following.
\begin{prop}
\label{prop:cubes_tight_means_tight}
Let $(\Sigma,Q)$ be a quadrangulated surface, and let $\Gamma_0, \Gamma_1$ be basic dividing sets. A contact structure on $M(\Gamma_0, \Gamma_1)$ is tight if and only if it is isotopic to a cubulated contact structure on $(\Sigma, Q) \times [0,1]$ where every cube has a tight contact structure.
\qed
\end{prop}

\subsection{Tightness of cubes}
\label{sec:tightness_of_cubes}

We now ask when a standard cube has a tight contact structure. In other words, we ask when the dividing set on a rounded standard cube is connected.

To this end we introduce some terminology and definitions for standard cubes in a cubulation $Q \times [0,1]$, where $Q$ is a quadrangulation of a marked surface $(\Sigma, V)$ consisting of squares $Q_1, \ldots, Q_k$. See figure \ref{fig:more_cube_examples} for illustrations of all this usage. The terminology may seem bizarre, but is adapted to the isomorphism of theorem \ref{thm:main_thm}.

First, we assign \emph{signs} to the vertices of cubes. Even when the cubes are rounded, we will still speak of vertices and think of them as points near the vertices of the bona fide cube. Each arc of $Q$ ends at a vertex of $V$, so each vertex of the cube $Q \times [0,1]$ is of the form $v \times \{0\}$ or $v \times \{1\}$ for some $v \in V$. We assign to this vertex the sign of $v$. Note these signs alternate around each $Q_i \times \{0\}$ and $Q_i \times \{1\}$, but do not alternate along vertical edges. We draw the positive vertices in green. (Note that these signs agree with signs of complementary regions of dividing sets on the top boundary, but disagree on the bottom.)

Second, we assign names to top and bottom dividing sets. A dividing set on the top or bottom square of a cube is called \emph{on} if it is a standard negative dividing set, and \emph{off} if it is a standard positive dividing set. The arc connecting the two positive vertices of the square $Q_i$ is called the \emph{principal diagonal} or \emph{positive diagonal}. So there is a principal diagonal on the top and bottom faces of the cube. The principal diagonal can be drawn on a face without intersecting the dividing set if and only if the face is on. When a face is on, we draw the corresponding diagonal and fill in the vertices connected by it; a face is off, we leave the diagonal out and leave the vertices hollow.

Third, we make some definitions for side faces. Swinging a principal diagonal from a positive vertex $v$ clockwise $45$ degrees, it hits an edge $e$ of the square; we say this edge is \emph{after} $v$, and the side face $e \times [0,1]$ is also \emph{after} $v$. Similarly, the edge and side face anticlockwise of $v$ are called \emph{before} $v$. Each side face can then be uniquely described as being before or after one of the two positive vertices of $Q$. The standard dividing set on a side face is called \emph{unused} if it spirals clockwise, as viewed from above, as it goes from top to bottom; otherwise it is called \emph{used}. We will often draw side faces as shaded when they are used, and clear if not.

Although these definitions assume the cube is rounded, so that the dividing set is smooth and standard on each face, they apply also when faces are sharpened. When we sharpen a side face, an unused dividing set becomes vertical, while a used dividing set becomes horizontal, as shown in figure \ref{fig:cube_gluing}. Thus, two side faces can be validly identified if and only if they are both used or both unused. 

A vertical dividing set on all sides of a cube with sharpened top and bottom corresponds, after rounding, to having all side faces unused. More generally, rounding the corners of $M(\Gamma_0, \Gamma_1)$ yields a dividing set in which every unglued side face of a cube (i.e. each side face around the boundary of $\Sigma \times [0,1]$) has unused dividing set, as shown in figure \ref{fig:vertical_side_boundary}.

\begin{figure}
\begin{center}
\begin{tikzpicture}[mypersp, scale=1.8]
\drawcube
\drawsigns
\leftclean
\rightclean
\backclean
\frontclean
\end{tikzpicture}   
\begin{tikzpicture}[mypersp, scale=1.8]
\drawcube
\drawsigns
\leftbyp
\rightbyp
\backbyp
\frontbyp
\end{tikzpicture}   
\begin{tikzpicture}[mypersp, scale=1.8]
\drawcube
\drawsigns
\leftclean
\backclean
\rightclean
\frontclean
\topon
\bottomoff
\end{tikzpicture}   
\begin{tikzpicture}[mypersp, scale=1.8]
\drawcube
\drawsigns
\leftclean
\backclean
\rightbyp
\frontbyp
\topoff
\bottomoff
\end{tikzpicture}   

\begin{tikzpicture}[mypersp, scale=1.8]
\drawcube
\drawsigns
\drawvertices
\bottomdiagonal
\draw [xyp=1, color=black!50!green, densely dotted, ultra thick] (0,0) -- (1,1);
\topoff
\bottomon
\end{tikzpicture}   
\begin{tikzpicture}[mypersp, scale=1.8]
\drawcube
\drawsigns
\drawvertices
\leftclean
\backclean
\rightclean
\frontclean
\topon
\bottomoff
\topdiagonal
\end{tikzpicture}   
\begin{tikzpicture}[mypersp, scale=1.8]
\drawcube
\drawsigns
\drawvertices
\labelvertices
\leftclean
\backclean
\frontbyp
\rightbyp
\topon
\bottomon
\topdiagonal
\bottomdiagonal
\end{tikzpicture}
\caption{Contact cubes; rounding not shown. Top, left to right: All side faces unused; all side faces used; bottom face off, top face on, all sides unused; front and right sides used, other sides unused, top and bottom off. Bottom, left to right: Top off, bottom on, so principal diagonals intersect the top dividing set but not the bottom. Centre: all sides unused, top on, bottom off. Right: the face after $v$ and the face before $w$ are used, other faces unused, top and bottom both on.}
\label{fig:cube_examples}
\label{fig:more_cube_examples}
\label{fig:principal_diagonal}
\end{center}
\end{figure}
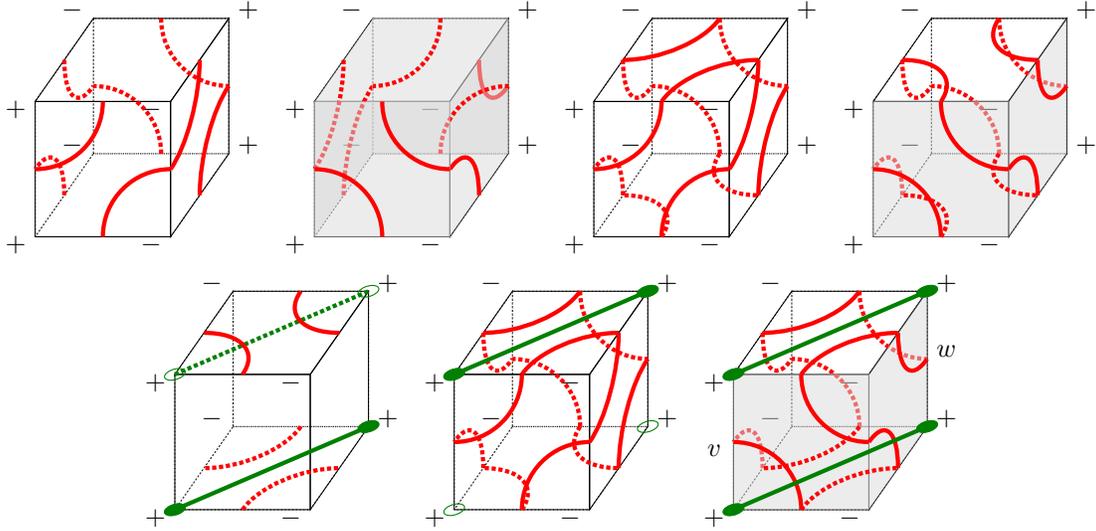


It is not difficult to run through the various cases and come up with a complete list of which standard cubes have a connected dividing set and hence a tight contact structure.
\begin{lem}
\label{lem:basic_cube_tight_conditions_1}
A standard cube has a tight contact structure if and only if it is one of the cases depicted in figure \ref{fig:tight_cubes}.
\qed
\end{lem}
In figure \ref{fig:tight_cubes}, to declutter the diagram we have not indicated signs of vertices explicitly, but the positive vertices are drawn in green. The two positive vertices of the square are labelled $v$ and $w$, and there is a symmetry of the square, preserving its product structure and signs of vertices, which rotates $180^\circ$ about a vertical axis, exchanging $v$ and $w$. Where two dividing sets are related by such a symmetry, we have only drawn one of them. Thus, if the vertices are labelled, we may have to rotate the cube before finding it in figure \ref{fig:tight_cubes}.

We can enumerate the various tight cubes in words too: they are as follows.
\begin{enumerate}
\item
All four side faces are unused; top and bottom are both on or both off.
\item
One side face $f$ used.
\begin{enumerate}
\item
$f$ is after a positive vertex, bottom on, top off.
\item
$f$ is before a positive vertex, bottom off, top on.
\end{enumerate}
\item
Two adjacent side faces are used.
\begin{enumerate}
\item
Used faces are before and after distinct positive vertices; top, bottom both on.
\item
Used faces are before and after the same positive vertex; top, bottom both off.
\end{enumerate}
\item
Three side faces are used.
\begin{enumerate}
\item
The unused side is after a positive vertex; bottom off, top on.
\item
The unused side is before a positive vertex; bottom on, top off.
\end{enumerate}
\item
All four side faces are used; top and bottom are both on or both off.
\end{enumerate}

\begin{figure}
\begin{center}
\begin{tikzpicture}[mypersp, scale=1.5]
\drawcube
\drawvertices
\labelvertices
\leftclean
\backclean
\frontclean
\rightclean
\topoff
\bottomoff
\draw (0.8,-0.7) node[align=center] {All sides unused \\ Top, bottom off};
\begin{scope}[xzp=0, xshift=2.3 cm, yshift = 0.2 cm]
\strandbackgroundshading
\cubestrandsetup
\leftoff
\rightoff
\draw (0.5,-0.5)  node[align=center] {$v,w \in {\bf a} \backslash (\supp h)$ \\ $M(v) \notin s \cup t$};
\end{scope}
\begin{scope}[xshift= 5 cm]
\drawcube
\drawvertices
\labelvertices
\leftclean
\backclean
\frontclean
\rightclean
\topon
\bottomon
\topdiagonal
\bottomdiagonal
\draw (0.8,-0.7) node[align=center] {All sides unused \\ Top, bottom on};
\end{scope}
\begin{scope}[xzp=0, xshift = 7.3 cm, yshift = 0.2 cm]
\strandbackgroundshading
\cubestrandsetup
\lefton
\righton
\draw [ultra thick, dotted]  (0,0.25) -- (1,0.25);
\draw [ultra thick, dotted]  (0,1.25) -- (1,1.25);
\draw (0.5,-0.5)  node[align=center] {$v,w \in {\bf a} \backslash (\supp h)$ \\ $M(v) \in s \cup t$};
\end{scope}

\begin{scope}[yshift = -3 cm]
\drawcube
\drawvertices
\labelvertices
\leftclean
\backclean
\frontbyp
\rightclean
\topoff
\bottomon
\bottomdiagonal
\draw (0.8,-0.7) node[align=center] {Side after $v$ used \\ Top off, bottom on};
\end{scope}
\begin{scope}[xzp=0, xshift=2.3 cm, yshift = -2.8 cm]
\strandbackgroundshading
\aftervused
\cubestrandsetup
\lefton
\rightoff
\draw [ultra thick]  (0,0.25) to [out=0, in=-90] (0.5,0.5);
\draw (0.5,-0.5)  node[align=center] {$v \in \partial^-(\supp h)$ \\ $w \in {\bf a} \backslash (\supp h)$ \\ $M(v) \in t \backslash s$};
\end{scope}
\begin{scope}[xshift = 5 cm, yshift=-3 cm]
\drawcube
\drawvertices
\labelvertices
\leftbyp
\backclean
\frontclean
\rightclean
\topon
\bottomoff
\topdiagonal
\draw (0.8,-0.7) node[align=center] {Side before $v$ used \\ Top on, bottom off};
\end{scope}
\begin{scope}[xzp=0, xshift=7.3 cm, yshift = -2.8 cm]
\strandbackgroundshading
\beforevused
\cubestrandsetup
\leftoff
\righton
\draw [ultra thick]  (0.5,0) to [out=90, in=180] (1,0.25);
\draw (0.5,-0.5)  node[align=center] {$v \in \partial^+ (\supp h)$ \\ $w \in {\bf a} \backslash (\supp h)$ \\ $M(v) \in t \backslash s$};
\end{scope}

\begin{scope}[yshift = -6 cm]
\drawcube
\drawvertices
\labelvertices
\leftclean
\backclean
\frontbyp
\rightbyp
\topon
\bottomon
\topdiagonal
\bottomdiagonal
\draw (0.8,-0.7)  node[align=center] {Sides after $v$, before $w$ used \\ Top, bottom on};
\end{scope}
\begin{scope}[xzp=0, xshift=2.3 cm, yshift = -5.8 cm]
\strandbackgroundshading
\aftervused
\beforewused
\cubestrandsetup
\lefton
\righton
\draw [ultra thick]  (0,0.25) to [out=0, in=-90] (0.5,0.5);
\draw [ultra thick]  (0.5,1) to [out=90, in=180] (1,1.25);
\draw (0.5,-0.5)  node[align=center] {$v \in \partial^- (\supp h)$ \\ $w \in \partial^+ (\supp h)$ \\ $M(v) \in s \cap t$};
\end{scope}
\begin{scope}[xshift = 5 cm, yshift = -6 cm]
\drawcube
\drawvertices
\labelvertices
\leftbyp
\backclean
\frontbyp
\rightclean
\topoff
\bottomoff
\draw (0.8,-0.7) node[align=center] {Sides before, after $v$ used \\ Top, bottom off};
\end{scope}
\begin{scope}[xzp=0, xshift=7.3 cm, yshift = -5.8 cm]
\strandbackgroundshading
\beforevused
\aftervused
\cubestrandsetup
\leftoff
\rightoff
\draw [ultra thick]  (0.4,0) to [out=90, in=-90] (0.6,0.5);
\draw (0.5,-0.5)  node[align=center] {$v \in \Int(\supp h)$ \\ $w \in {\bf a} \backslash (\supp h)$ \\ $M(v) \notin s \cup t$};
\end{scope}

\begin{scope}[yshift = -9 cm]
\drawcube
\drawvertices
\labelvertices
\leftbyp
\backbyp
\frontclean
\rightbyp
\topon
\bottomoff
\topdiagonal
\draw (0.8,-0.7) node[align=center] {Side after $v$ unused \\ Top on, bottom off};
\end{scope}
\begin{scope}[xzp=0, xshift=2.3 cm, yshift=-8.8 cm]
\strandbackgroundshading
\beforevused
\beforewused
\afterwused
\cubestrandsetup
\leftoff
\righton
\draw [ultra thick]  (0.5,0) to [out=90, in=180] (1,0.25);
\draw [ultra thick]  (0.4,1) to [out=90, in=-90] (0.6,1.5);
\draw (0.5,-0.5) node[align=center] {$v \in \partial^+ (\supp h)$ \\ $w \in \Int (\supp h)$ \\ $M(v) \in t \backslash s$};
\end{scope}
\begin{scope}[xshift = 5 cm, yshift = -9 cm]
\drawcube
\drawvertices
\labelvertices
\leftclean
\backbyp
\frontbyp
\rightbyp
\topoff
\bottomon
\bottomdiagonal
\draw (0.8,-0.7) node[align=center] {Side before $v$ unused \\ Top off, bottom on};
\end{scope}
\begin{scope}[xzp=0, xshift=7.3 cm, yshift = -8.8 cm]
\strandbackgroundshading
\aftervused
\beforewused
\afterwused
\cubestrandsetup
\lefton
\rightoff
\draw [ultra thick]  (0,0.25) to [out=0, in=-90] (0.5,0.5);
\draw [ultra thick]  (0.4,1) to [out=90, in=-90] (0.6,1.5);
\draw (0.5,-0.5)  node[align=center] {$v \in \partial^-  (\supp h)$ \\ $w \in \Int (\supp h)$ \\ $M(v) \in s \backslash t$};
\end{scope}

\begin{scope}[xshift = -1 cm, yshift = -12 cm]
\drawcube
\drawvertices
\labelvertices
\leftbyp
\backbyp
\frontbyp
\rightbyp
\topoff
\bottomoff
\draw (0.8,-0.7) node[align=center] {All sides used \\ Top, bottom off};\end{scope}
\begin{scope}[xzp=0, xshift=1.3 cm, yshift = -11.8 cm]
\strandbackgroundshading
\beforevused
\aftervused
\beforewused
\afterwused
\cubestrandsetup
\leftoff
\rightoff
\draw [ultra thick]  (0.4,0) to [out=90, in=-90] (0.6,0.5);
\draw [ultra thick]  (0.4,1) to [out=90, in=-90] (0.6,1.5);
\draw (0.5,-0.5)  node[align=center] {$v,w \in \Int(\supp h)$ \\ $M(v) \notin s \cup t$};
\end{scope}
\begin{scope}[xshift = 4 cm, yshift = -12 cm]
\drawcube
\drawvertices
\labelvertices
\leftbyp
\backbyp
\frontbyp
\rightbyp
\topon
\bottomon
\topdiagonal
\bottomdiagonal
\draw (0.8,-0.7) node[align=center] {All sides used \\ Top, bottom on};
\end{scope}
\begin{scope}[xzp=0, xshift=6.3 cm, yshift = -11.8 cm]
\strandbackgroundshading
\beforevused
\aftervused
\beforewused
\afterwused
\cubestrandsetup
\lefton
\righton
\draw [ultra thick]  (0,0.25) to [out=0, in=-90] (0.5,0.5);
\draw [ultra thick]  (0.5,0) to [out=90, in=180] (1,0.25);
\draw [ultra thick]  (0.4,1) to [out=90, in=-90] (0.6,1.5);
\draw (1.2,0.75) node {or};
\draw (1.2,-0.5)  node[align=center] {$v,w \in \Int(\supp h)$ \\ $M(p) \notin s \cup t$};
\end{scope}
\begin{scope}[xzp=0, xshift=8 cm, yshift = -11.8 cm]
\strandbackgroundshading
\beforevused
\aftervused
\beforewused
\afterwused
\cubestrandsetup
\lefton
\righton
\draw [ultra thick]  (0,1.25) to [out=0, in=-90] (0.5,1.5);
\draw [ultra thick]  (0.5,1) to [out=90, in=180] (1,1.25);
\draw [ultra thick]  (0.4,0) to [out=90, in=-90] (0.6,0.5);
\end{scope}
\end{tikzpicture} 

\caption{The various tight cubes, and corresponding strand diagrams.}
\label{fig:tight_cubes}
\label{fig:homology_generators_local}
\end{center}
\end{figure}
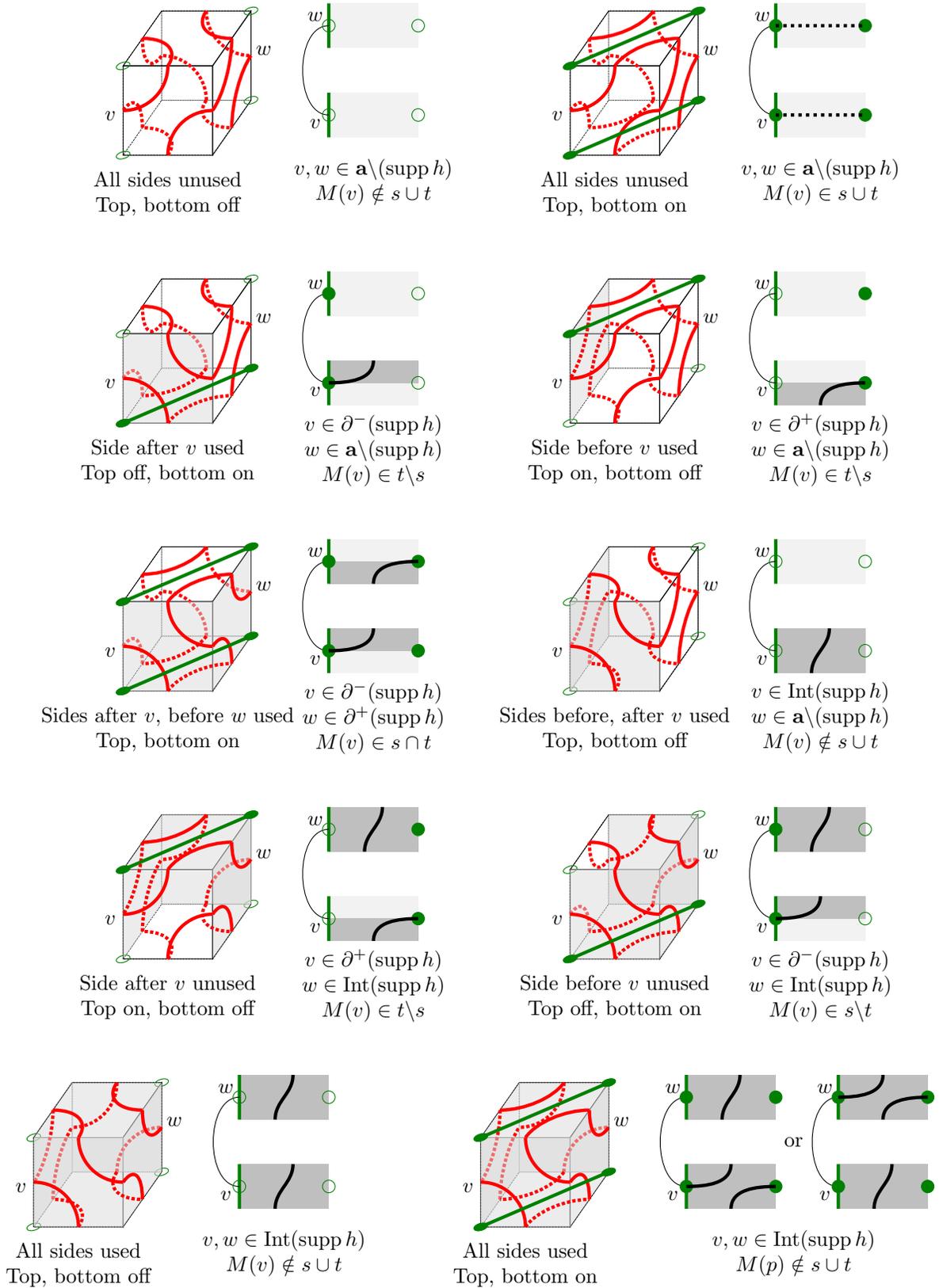

\subsection{Classification of tight contact structures}

Let $(\Sigma, Q)$ be a quadrangulated surface and let $\Gamma_0, \Gamma_1$ be basic dividing sets. We have seen (proposition \ref{prop:cubes_tight_means_tight}) that the tight contact structures on $M(\Gamma_0, \Gamma_1)$ are precisely the cubulated contact structures on $(\Sigma, Q) \times [0,1]$ with all cubes tight, i.e. satisfying the conditions of lemma \ref{lem:basic_cube_tight_conditions_1} and appearing in figure \ref{fig:tight_cubes}.

On each cube of the cubulation, the top and bottom dividing sets are prescribed by $\Gamma_0$ and $\Gamma_1$. Each side face which is not glued to another appears in the side boundary of $M(\Gamma_0, \Gamma_1)$ and hence is unused.

So the only way in which such contact structures may differ is by whether glued side faces are used or unused. The glued side faces are precisely those of the form $A \times [0,1]$, where $A$ is a decomposing arc of $Q$. If we choose, for each decomposing arc $A$, whether it is used or unused, then we have specified a dividing set on each face of each cube. These dividing sets may or may not be tight, but the choices of used and unused decomposing arcs which make all cubes tight yield tight contact structures on $M(\Gamma_0, \Gamma_1)$, and all tight contact structures on $M(\Gamma_0, \Gamma_1)$ are of this form.

To obtain a complete classification of contact structures in terms of a cubulation, it remains to check that distinct choices of used and unused glued faces yield distinct (i.e. non-isotopic) contact structures.
\begin{lem}
Suppose $\xi_0, \xi_1$ are two cubulated contact structures on $(\Sigma, Q) \times [0,1]$ such the face $A \times [0,1]$ is unused in $\xi_0$ and used in $\xi_1$. Then $\xi_0$ is not isotopic to $\xi_1$.
\end{lem}

\begin{proof}
The two dividing sets obtained on $A \times I$ are the two distinct ways of matching four boundary points on a disc. So when we calculate the relative Euler class $e(\xi_0)$ on $[A \times I]$, it differs from that of $e(\xi_1)$, and hence the contact structures cannot be isotopic.
\end{proof}

We now immediately obtain the following proposition.
\begin{prop}
\label{prop:contact_arc_labelling}
The isotopy classes of tight contact structures on $M(\Gamma_0, \Gamma_1)$ are in bijective correspondence with labellings of decomposing arcs of $Q$ as ``used" of ``unused", so that each cube satisfies the conditions of lemma \ref{lem:basic_cube_tight_conditions_1}.
\qed
\end{prop}

\subsection{Stacking cubulated contact structures}

In the contact category, as we know, morphisms correspond to stacking contact structures on $\Sigma \times [0,1]$ on top of each other. We now investigate the stacking of cubulated contact structures.

To this end let $(\Sigma,Q)$ be a quadrangulated surface, let $\Gamma_0, \Gamma_1, \Gamma_2$ be basic dividing sets, let $\xi_0$ be a tight contact structure on $M(\Gamma_0, \Gamma_1)$, and let $\xi_1$ be a tight contact structure on $M(\Gamma_1, \Gamma_2)$. Stacking $\xi_1$ on $\xi_0$ yields a contact structure $\xi$ on $M(\Gamma_0, \Gamma_2)$. Let $A$ be the set of decomposing arcs of $Q$. By proposition \ref{prop:contact_arc_labelling}, $\xi_0$ is defined by the subset of used arcs of $A$; let this set by $U_0 \subseteq A$. Similarly, let $U_1 \subseteq A$ be the set of used arcs in $\xi_1$. We can then describe $\xi$ as follows.
\begin{prop} \
\label{prop:stacking_cubulated_structures}
\begin{enumerate}
\item
If $U_0 \cap U_1 \neq \emptyset$ then $\xi$ is overtwisted.
\item
If $U_0 \cap U_1 = \emptyset$ then $\xi$ is a cubulated contact structure with used arcs $U_0 \cup U_1$. In this case, $\xi$ is tight if and only if there is a tight cubulated contact structure with used arcs $U_0 \cup U_1$.
\end{enumerate}
\end{prop}
Note that if $U_0 \cap U_1 = \emptyset$, $\xi$ may or may not be tight. The final statement in the second part is necessary because $\xi$ being cubulated is just a statement about its dividing sets; we additionally assert that $\xi$ is tight, if it is possible to be so.

\begin{proof}
Let $A$ be a decomposing arc of $Q$. We first consider decomposing $\xi_0$ along $A \times [0,1]$. Rounding the corners of $M(\Gamma_0, \Gamma_1)$ as in figure \ref{fig:vertical_side_boundary} and cutting along a rounded convex $A \times [0,1]$ with Legendrian boundary, as in figure \ref{fig:cube_gluing}, we obtain a horizontal or vertical dividing set on $A \times [0,1]$, accordingly as $A$ is used or not in $\xi_0$. We can do the same for $\xi_1$.

If $A$ is used in both $\xi_0$ and $\xi_1$, then we obtain horizontal dividing sets on $A \times I$ in both contact structures. These two dividing sets stack together to give a convex surface in $\xi$ with a contractible loop: see figure \ref{fig:stacking_cubes}(above). Thus $\xi$ is overtwisted. 

If $A$ is used in one of $\xi_0, \xi_1$ but not the other, then we obtain a horizontal dividing set on $A \times [0,1]$ in one contact structure, and a vertical dividing set in the other. These piece together to give a horizontal dividing set on $A \times I$ in $\xi$: see figure \ref{fig:stacking_cubes}(below). By a similar argument, if $A$ is used in neither of $\xi_0, \xi_1$, then the face $A \times I$ is unused in $\xi$.

If $U_0 \cap U_1$ is nonempty, then there is some arc $A$ used in both $\xi_0$ and $\xi_1$, so $\xi$ is overtwisted.

Suppose now that $U_0 \cap U_1 = \emptyset$. Then we obtain a cubulation of $\xi$ by stacking together the cubes of $\xi_0$ and $\xi_1$. The used faces of $\xi$ are precisely those in which are used in $\xi_0$ or $\xi_1$, i.e. $U_0 \cup U_1$. Moreover, when we stack a tight cube $C_1$ from $\xi_1$ on top of a corresponding cube $C_0$ from $\xi_0$ to obtain a cube $C$ of $\xi$, we glue two tight 3-balls together along a square face containing a standard dividing set. If the cube $C$ has (after rounding edges) a connected dividing set on its boundary, then it has a tight contact structure. This can be seen from Honda’s gluing theorem \cite{Hon02} as in the proof of lemma \ref{lem:cubulated_implies_tight}; again, no bypass can be passed through the square with a standard dividing set. By proposition \ref{prop:cubes_tight_means_tight}, $\xi$ is tight if and only if every cube is tight.

So $\xi$ is cubulated with used arcs $U_0 \cup U_1$, and moreover, $\xi$ is tight if and only if each cube has a connected dividing set. This occurs if and only if there is a tight cubulated contact structure with used arcs $U_0 \cup U_1$.
\end{proof}

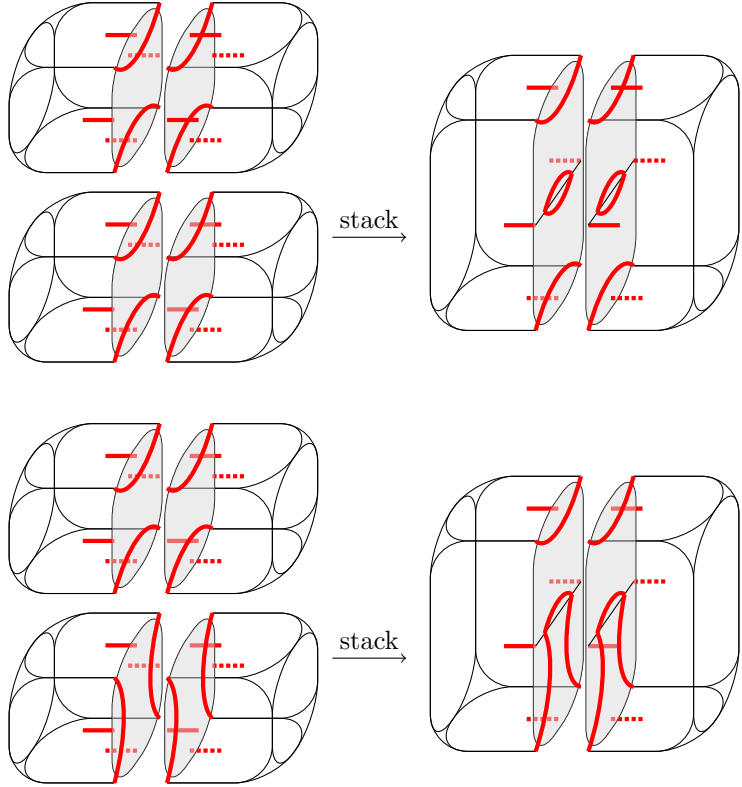
\begin{figure}
\begin{center}

\begin{tikzpicture}[mypersp, scale=1.4]
\draw [yzp=1, rounded corners=5mm] (-0.05,-0.1) -- (-0.05,1.1) -- (1.05,1.1) -- (1.05,-0.1) -- cycle;
\draw [yzp=0, rounded corners=5mm] (0,0) -- (0,1) -- (1,1) -- (1,0) -- cycle;
\draw [xzp=1, rounded corners=5mm] (1,1) -- (0,1) -- (0,0) -- (1,0);
\draw [xzp=0, rounded corners=5mm] (1,1) -- (0,1) -- (0,0) -- (1,0);
\draw [xyp=1, rounded corners=5mm] (1,1) -- (0,1) -- (0,0) -- (1,0);
\draw [xyp=0, rounded corners=5mm] (1,1) -- (0,1) -- (0,0) -- (1,0);
\draw [red, ultra thick, densely dotted, xyp=0] (1,0.5) -- (0.7,0.5);
\draw [red, ultra thick, xyp=1] (1,0.5) -- (0.7,0.5);
\draw [red, ultra thick, xzp=0] (1,0.5) -- (0.7,0.5);
\draw [red, ultra thick, densely dotted, xzp=1] (1,0.5) -- (0.7,0.5);
\fill [gray!30!white, draw=none, yzp=1, opacity=0.5, rounded corners=5mm] (-0.05,-0.1) -- (-0.05,1.1) -- (1.05,1.1) -- (1.05,-0.1) -- cycle;
\draw [red, ultra thick, yzp=1] (0,0) to [bend left=45] (1,0);
\draw [red, ultra thick, yzp=1] (0,1) to [bend right=45] (1,1);
\begin{scope}[xshift=1.5cm]
\draw [yzp=0, rounded corners=5mm] (-0.05,-0.1) -- (-0.05,1.1) -- (1.05,1.1) -- (1.05,-0.1) -- cycle;
\draw [yzp=1, rounded corners=5mm] (0,0) -- (0,1) -- (1,1) -- (1,0) -- cycle;
\draw [xzp=1, rounded corners=5mm] (0,1) -- (1,1) -- (1,0) -- (0,0);
\draw [xzp=0, rounded corners=5mm] (0,1) -- (1,1) -- (1,0) -- (0,0);
\draw [xyp=1, rounded corners=5mm] (0,1) -- (1,1) -- (1,0) -- (0,0);
\draw [xyp=0, rounded corners=5mm] (0,1) -- (1,1) -- (1,0) -- (0,0);
\fill [gray!30!white, draw=none, yzp=0, opacity=0.5, rounded corners=5mm] (-0.05,-0.1) -- (-0.05,1.1) -- (1.05,1.1) -- (1.05,-0.1) -- cycle;
\draw [red, ultra thick, densely dotted, xyp=0] (0,0.5) -- (0.3,0.5);
\draw [red, ultra thick, xyp=1] (0,0.5) -- (0.3,0.5);
\draw [red, ultra thick, xzp=0] (0,0.5) -- (0.3,0.5);
\draw [red, ultra thick, densely dotted, xzp=1] (0,0.5) -- (0.3,0.5);
\draw [red, ultra thick, yzp=0] (0,0) to [bend left=45] (1,0);
\draw [red, ultra thick, yzp=0] (0,1) to [bend right=45] (1,1);
\end{scope}
\begin{scope}[yshift=-1.8cm]
\draw [yzp=1, rounded corners=5mm] (-0.05,-0.1) -- (-0.05,1.1) -- (1.05,1.1) -- (1.05,-0.1) -- cycle;
\draw [yzp=0, rounded corners=5mm] (0,0) -- (0,1) -- (1,1) -- (1,0) -- cycle;
\draw [xzp=1, rounded corners=5mm] (1,1) -- (0,1) -- (0,0) -- (1,0);
\draw [xzp=0, rounded corners=5mm] (1,1) -- (0,1) -- (0,0) -- (1,0);
\draw [xyp=1, rounded corners=5mm] (1,1) -- (0,1) -- (0,0) -- (1,0);
\draw [xyp=0, rounded corners=5mm] (1,1) -- (0,1) -- (0,0) -- (1,0);
\draw [red, ultra thick, densely dotted, xyp=0] (1,0.5) -- (0.7,0.5);
\draw [red, ultra thick, xyp=1] (1,0.5) -- (0.7,0.5);
\draw [red, ultra thick, xzp=0] (1,0.5) -- (0.7,0.5);
\draw [red, ultra thick, densely dotted, xzp=1] (1,0.5) -- (0.7,0.5);
\fill [gray!30!white, draw=none, yzp=1, opacity=0.5, rounded corners=5mm] (-0.05,-0.1) -- (-0.05,1.1) -- (1.05,1.1) -- (1.05,-0.1) -- cycle;
\draw [red, ultra thick, yzp=1] (0,0) to [bend left=45] (1,0);
\draw [red, ultra thick, yzp=1] (0,1) to [bend right=45] (1,1);
\end{scope}
\begin{scope}[xshift=1.5 cm, yshift=-1.8cm]
\draw [yzp=0, rounded corners=5mm] (-0.05,-0.1) -- (-0.05,1.1) -- (1.05,1.1) -- (1.05,-0.1) -- cycle;
\draw [yzp=1, rounded corners=5mm] (0,0) -- (0,1) -- (1,1) -- (1,0) -- cycle;
\draw [xzp=1, rounded corners=5mm] (0,1) -- (1,1) -- (1,0) -- (0,0);
\draw [xzp=0, rounded corners=5mm] (0,1) -- (1,1) -- (1,0) -- (0,0);
\draw [xyp=1, rounded corners=5mm] (0,1) -- (1,1) -- (1,0) -- (0,0);
\draw [xyp=0, rounded corners=5mm] (0,1) -- (1,1) -- (1,0) -- (0,0);
\draw [red, ultra thick, densely dotted, xyp=0] (0,0.5) -- (0.3,0.5);
\draw [red, ultra thick, xyp=1] (0,0.5) -- (0.3,0.5);
\draw [red, ultra thick, xzp=0] (0,0.5) -- (0.3,0.5);
\draw [red, ultra thick, densely dotted, xzp=1] (0,0.5) -- (0.3,0.5);
\fill [gray!30!white, draw=none, yzp=0, opacity=0.5, rounded corners=5mm] (-0.05,-0.1) -- (-0.05,1.1) -- (1.05,1.1) -- (1.05,-0.1) -- cycle;
\draw [red, ultra thick, yzp=0] (0,0) to [bend left=45] (1,0);
\draw [red, ultra thick, yzp=0] (0,1) to [bend right=45] (1,1);
\end{scope}

\draw [->] (3.5,-1) -- (4.2,-1) node [midway, above] {stack};

\begin{scope}[xshift=4cm, yshift=-1.5cm]
\draw [yzp=1, rounded corners=5mm] (-0.05,-0.1) -- (-0.05,2.1) -- (1.05,2.1) -- (1.05,-0.1) -- cycle;
\draw [yzp=0, rounded corners=5mm] (0,0) -- (0,2) -- (1,2) -- (1,0) -- cycle;
\draw [xzp=1, rounded corners=5mm] (1,2) -- (0,2) -- (0,0) -- (1,0);
\draw [xzp=0, rounded corners=5mm] (1,2) -- (0,2) -- (0,0) -- (1,0);
\draw [xyp=2, rounded corners=5mm] (1,1) -- (0,1) -- (0,0) -- (1,0);
\draw [xyp=0, rounded corners=5mm] (1,1) -- (0,1) -- (0,0) -- (1,0);
\draw [red, ultra thick, densely dotted, xyp=0] (1,0.5) -- (0.7,0.5);
\draw [red, ultra thick, xyp=2] (1,0.5) -- (0.7,0.5);
\draw [red, ultra thick, xzp=0] (1,1) -- (0.7,1);
\draw [red, ultra thick, densely dotted, xzp=1] (1,1) -- (0.7,1);
\fill [gray!30!white, draw=none, opacity=0.5, yzp=1, rounded corners=5mm] (-0.05,-0.1) -- (-0.05,2.1) -- (1.05,2.1) -- (1.05,-0.1) -- cycle;
\draw [xyp=1] (1,0) -- (1,1);
\draw [red, ultra thick, yzp=1] (0,0) to [bend left=45] (1,0);
\draw [red, ultra thick, yzp=1] (0.2,1) to [bend right=45] (0.8,1);
\draw [red, ultra thick, yzp=1] (0,2) to [bend right=45] (1,2);
\draw [red, ultra thick, yzp=1] (0.2,1) to [bend left=45] (0.8,1);
\end{scope}
\begin{scope}[xshift=5.5cm, yshift=-1.5cm]
\draw [yzp=0, rounded corners=5mm] (-0.05,-0.1) -- (-0.05,2.1) -- (1.05,2.1) -- (1.05,-0.1) -- cycle;
\draw [yzp=1, rounded corners=5mm] (0,0) -- (0,2) -- (1,2) -- (1,0) -- cycle;
\draw [xzp=1, rounded corners=5mm] (0,2) -- (1,2) -- (1,0) -- (0,0);
\draw [xzp=0, rounded corners=5mm] (0,2) -- (1,2) -- (1,0) -- (0,0);
\draw [xyp=2, rounded corners=5mm] (0,1) -- (1,1) -- (1,0) -- (0,0);
\draw [xyp=0, rounded corners=5mm] (0,1) -- (1,1) -- (1,0) -- (0,0);
\fill [gray!30!white, draw=none, yzp=0, opacity=0.5, rounded corners=5mm] (-0.05,-0.1) -- (-0.05,2.1) -- (1.05,2.1) -- (1.05,-0.1) -- cycle;
\draw [red, ultra thick, densely dotted, xyp=0] (0,0.5) -- (0.3,0.5);
\draw [red, ultra thick, xyp=2] (0,0.5) -- (0.3,0.5);
\draw [red, ultra thick, xzp=0] (0,1) -- (0.3,1);
\draw [red, ultra thick, densely dotted, xzp=1] (0,1) -- (0.3,1);
\draw [xyp=1] (0,0) -- (0,1);
\draw [red, ultra thick, yzp=0] (0,0) to [bend left=45] (1,0);
\draw [red, ultra thick, yzp=0] (0.2,1) to [bend right=45] (0.8,1);
\draw [red, ultra thick, yzp=0] (0,2) to [bend right=45] (1,2);
\draw [red, ultra thick, yzp=0] (0.2,1) to [bend left=45] (0.8,1);
\end{scope}

\begin{scope}[yshift = -4 cm]
\draw [yzp=1, rounded corners=5mm] (-0.05,-0.1) -- (-0.05,1.1) -- (1.05,1.1) -- (1.05,-0.1) -- cycle;
\draw [yzp=0, rounded corners=5mm] (0,0) -- (0,1) -- (1,1) -- (1,0) -- cycle;
\draw [xzp=1, rounded corners=5mm] (1,1) -- (0,1) -- (0,0) -- (1,0);
\draw [xzp=0, rounded corners=5mm] (1,1) -- (0,1) -- (0,0) -- (1,0);
\draw [xyp=1, rounded corners=5mm] (1,1) -- (0,1) -- (0,0) -- (1,0);
\draw [xyp=0, rounded corners=5mm] (1,1) -- (0,1) -- (0,0) -- (1,0);
\draw [red, ultra thick, densely dotted, xyp=0] (1,0.5) -- (0.7,0.5);
\draw [red, ultra thick, xyp=1] (1,0.5) -- (0.7,0.5);
\draw [red, ultra thick, xzp=0] (1,0.5) -- (0.7,0.5);
\draw [red, ultra thick, densely dotted, xzp=1] (1,0.5) -- (0.7,0.5);
\fill [gray!30!white, draw=none, yzp=1, opacity=0.5, rounded corners=5mm] (-0.05,-0.1) -- (-0.05,1.1) -- (1.05,1.1) -- (1.05,-0.1) -- cycle;
\draw [red, ultra thick, yzp=1] (0,0) to [bend left=45] (1,0);
\draw [red, ultra thick, yzp=1] (0,1) to [bend right=45] (1,1);
\end{scope}
\begin{scope}[yshift = -4 cm, xshift=1.5cm]
\draw [yzp=0, rounded corners=5mm] (-0.05,-0.1) -- (-0.05,1.1) -- (1.05,1.1) -- (1.05,-0.1) -- cycle;
\draw [yzp=1, rounded corners=5mm] (0,0) -- (0,1) -- (1,1) -- (1,0) -- cycle;
\draw [xzp=1, rounded corners=5mm] (0,1) -- (1,1) -- (1,0) -- (0,0);
\draw [xzp=0, rounded corners=5mm] (0,1) -- (1,1) -- (1,0) -- (0,0);
\draw [xyp=1, rounded corners=5mm] (0,1) -- (1,1) -- (1,0) -- (0,0);
\draw [xyp=0, rounded corners=5mm] (0,1) -- (1,1) -- (1,0) -- (0,0);
\draw [red, ultra thick, densely dotted, xyp=0] (0,0.5) -- (0.3,0.5);
\draw [red, ultra thick, xyp=1] (0,0.5) -- (0.3,0.5);
\draw [red, ultra thick, xzp=0] (0,0.5) -- (0.3,0.5);
\draw [red, ultra thick, densely dotted, xzp=1] (0,0.5) -- (0.3,0.5);
\fill [gray!30!white, draw=none, yzp=0, opacity=0.5, rounded corners=5mm] (-0.05,-0.1) -- (-0.05,1.1) -- (1.05,1.1) -- (1.05,-0.1) -- cycle;
\draw [red, ultra thick, yzp=0] (0,0) to [bend left=45] (1,0);
\draw [red, ultra thick, yzp=0] (0,1) to [bend right=45] (1,1);
\end{scope}
\begin{scope}[yshift=-5.8cm]
\draw [yzp=1, rounded corners=5mm] (-0.05,-0.1) -- (-0.05,1.1) -- (1.05,1.1) -- (1.05,-0.1) -- cycle;
\draw [yzp=0, rounded corners=5mm] (0,0) -- (0,1) -- (1,1) -- (1,0) -- cycle;
\draw [xzp=1, rounded corners=5mm] (1,1) -- (0,1) -- (0,0) -- (1,0);
\draw [xzp=0, rounded corners=5mm] (1,1) -- (0,1) -- (0,0) -- (1,0);
\draw [xyp=1, rounded corners=5mm] (1,1) -- (0,1) -- (0,0) -- (1,0);
\draw [xyp=0, rounded corners=5mm] (1,1) -- (0,1) -- (0,0) -- (1,0);
\draw [red, ultra thick, densely dotted, xyp=0] (1,0.5) -- (0.7,0.5);
\draw [red, ultra thick, xyp=1] (1,0.5) -- (0.7,0.5);
\draw [red, ultra thick, xzp=0] (1,0.5) -- (0.7,0.5);
\draw [red, ultra thick, densely dotted, xzp=1] (1,0.5) -- (0.7,0.5);
\fill [gray!30!white, draw=none, yzp=1, opacity=0.5, rounded corners=5mm] (-0.05,-0.1) -- (-0.05,1.1) -- (1.05,1.1) -- (1.05,-0.1) -- cycle;
\draw [red, ultra thick, yzp=1] (0,0) to [bend right=45] (0,1);
\draw [red, ultra thick, yzp=1] (1,0) to [bend left=45] (1,1);
\end{scope}
\begin{scope}[xshift=1.5 cm, yshift=-5.8cm]
\draw [yzp=0, rounded corners=5mm] (-0.05,-0.1) -- (-0.05,1.1) -- (1.05,1.1) -- (1.05,-0.1) -- cycle;
\draw [yzp=1, rounded corners=5mm] (0,0) -- (0,1) -- (1,1) -- (1,0) -- cycle;
\draw [xzp=1, rounded corners=5mm] (0,1) -- (1,1) -- (1,0) -- (0,0);
\draw [xzp=0, rounded corners=5mm] (0,1) -- (1,1) -- (1,0) -- (0,0);
\draw [xyp=1, rounded corners=5mm] (0,1) -- (1,1) -- (1,0) -- (0,0);
\draw [xyp=0, rounded corners=5mm] (0,1) -- (1,1) -- (1,0) -- (0,0);
\draw [red, ultra thick, densely dotted, xyp=0] (0,0.5) -- (0.3,0.5);
\draw [red, ultra thick, xyp=1] (0,0.5) -- (0.3,0.5);
\draw [red, ultra thick, xzp=0] (0,0.5) -- (0.3,0.5);
\draw [red, ultra thick, densely dotted, xzp=1] (0,0.5) -- (0.3,0.5);
\fill [gray!30!white, draw=none, yzp=0, opacity=0.5, rounded corners=5mm] (-0.05,-0.1) -- (-0.05,1.1) -- (1.05,1.1) -- (1.05,-0.1) -- cycle;
\draw [red, ultra thick, yzp=0] (0,0) to [bend right=45] (0,1);
\draw [red, ultra thick, yzp=0] (1,0) to [bend left=45] (1,1);
\end{scope}

\draw [yshift = -4 cm, ->] (3.5,-1) -- (4.2,-1) node [midway, above] {stack};

\begin{scope}[xshift=4cm, yshift=-5.5cm]
\draw [yzp=1, rounded corners=5mm] (-0.05,-0.1) -- (-0.05,2.1) -- (1.05,2.1) -- (1.05,-0.1) -- cycle;
\draw [yzp=0, rounded corners=5mm] (0,0) -- (0,2) -- (1,2) -- (1,0) -- cycle;
\draw [xzp=1, rounded corners=5mm] (1,2) -- (0,2) -- (0,0) -- (1,0);
\draw [xzp=0, rounded corners=5mm] (1,2) -- (0,2) -- (0,0) -- (1,0);
\draw [xyp=2, rounded corners=5mm] (1,1) -- (0,1) -- (0,0) -- (1,0);
\draw [xyp=0, rounded corners=5mm] (1,1) -- (0,1) -- (0,0) -- (1,0);
\draw [red, ultra thick, densely dotted, xyp=0] (1,0.5) -- (0.7,0.5);
\draw [red, ultra thick, xyp=2] (1,0.5) -- (0.7,0.5);
\draw [red, ultra thick, xzp=0] (1,1) -- (0.7,1);
\draw [red, ultra thick, densely dotted, xzp=1] (1,1) -- (0.7,1);
\fill [gray!30!white, draw=none, opacity=0.5, yzp=1, rounded corners=5mm] (-0.05,-0.1) -- (-0.05,2.1) -- (1.05,2.1) -- (1.05,-0.1) -- cycle;
\draw [xyp=1] (1,0) -- (1,1);
\draw [red, ultra thick, yzp=1] (0,2) to [bend right=45] (1,2);
\draw [red, ultra thick, yzp=1] (0.2,1) to [bend left=45] (0.8,1);
\draw [red, ultra thick, yzp=1] (0,0) to [bend right=45] (0.2,1);
\draw [red, ultra thick, yzp=1] (1,0) to [bend left=45] (0.8,1);
\end{scope}
\begin{scope}[xshift=5.5cm, yshift=-5.5cm]
\draw [yzp=0, rounded corners=5mm] (-0.05,-0.1) -- (-0.05,2.1) -- (1.05,2.1) -- (1.05,-0.1) -- cycle;
\draw [yzp=1, rounded corners=5mm] (0,0) -- (0,2) -- (1,2) -- (1,0) -- cycle;
\draw [xzp=1, rounded corners=5mm] (0,2) -- (1,2) -- (1,0) -- (0,0);
\draw [xzp=0, rounded corners=5mm] (0,2) -- (1,2) -- (1,0) -- (0,0);
\draw [xyp=2, rounded corners=5mm] (0,1) -- (1,1) -- (1,0) -- (0,0);
\draw [xyp=0, rounded corners=5mm] (0,1) -- (1,1) -- (1,0) -- (0,0);
\draw [red, ultra thick, densely dotted, xyp=0] (0,0.5) -- (0.3,0.5);
\draw [red, ultra thick, xyp=2] (0,0.5) -- (0.3,0.5);
\draw [red, ultra thick, xzp=0] (0,1) -- (0.3,1);
\draw [red, ultra thick, densely dotted, xzp=1] (0,1) -- (0.3,1);
\fill [gray!30!white, draw=none, yzp=0, opacity=0.5, rounded corners=5mm] (-0.05,-0.1) -- (-0.05,2.1) -- (1.05,2.1) -- (1.05,-0.1) -- cycle;
\draw [xyp=1] (0,0) -- (0,1);
\draw [red, ultra thick, yzp=0] (0,2) to [bend right=45] (1,2);
\draw [red, ultra thick, yzp=0] (0.2,1) to [bend left=45] (0.8,1);
\draw [red, ultra thick, yzp=0] (0,0) to [bend right=45] (0.2,1);
\draw [red, ultra thick, yzp=0] (1,0) to [bend left=45] (0.8,1);
\end{scope}
\end{tikzpicture}

\caption{Above: Stacking two used side faces on top of each other yields an overtwisted contact structure. Below: A used face on top of an unused face yields a used face.}
\label{fig:stacking_cubes}
\end{center}
\end{figure}

\section{The strand algebra and its homology}
\label{sec:strands_algebra}

\subsection{Arc diagrams and tape graphs}
\label{sec:arc_diagrams_tape_graphs}

We now turn to the Heegaard-Floer side of the story, and discuss the strand algebra. We mostly follow Zarev's exposition in \cite{Zarev09}, but also refer to work of Lipshitz--Ozsv\'{a}th--Thurston \cite{LOT08, LOT11_Bimodules}, and refine and introduce some terminology for our purposes.

\begin{defn}
\label{def:arc_diagram}
An \emph{arc diagram} is a triple $\mathcal{Z} = ({\bf Z}, {\bf a}, M)$ where
\begin{enumerate}
\item
${\bf Z} = (Z_1, Z_2, \ldots, Z_l)$ is a sequence of oriented line segments,
\item
${\bf a} = (a_1, a_2, \ldots, a_{2k})$ is a sequence of distinct points of ${\bf Z}$ in order along ${\bf Z}$, and
\item
$M$ is a 2-to-1 function ${\bf a} \To \{1, 2, \ldots, k\}$.
\end{enumerate}
We require that after performing oriented surgery on ${\bf Z}$ at each 0-sphere $M^{-1}(i)$, the resulting 1-manifold should consist of arcs; no circles are allowed.
\end{defn}
The points of ${\bf a}$ are called \emph{places}. The function $M$ matches the points of $M$ in $k$ \emph{matched pairs} which we also call \emph{twins}. We reserve the letter $k$ for the number of marked pairs. Oriented surgery preserves the orientation of the $1$-manifold; see figure \ref{fig:0-surgery}. We draw arc diagrams by drawing the segments $Z_i$ vertically, oriented upwards; the matched pairs are indicated by \emph{arcs} between them. If we permute the sequence of line segments $Z_i$, reorder the $a_i$ and adjust the matching $M$ accordingly, we obtain an equivalent arc diagram.

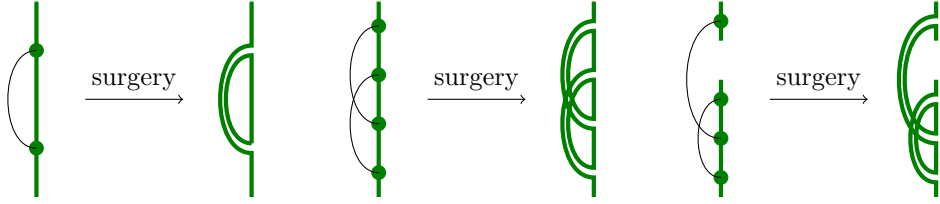
\begin{figure}
\begin{center}
\begin{tikzpicture}[scale=1.3]
\draw [ultra thick, color=black!50!green] (0,0) -- (0,2);
\draw [color=black!50!green, fill=black!50!green] (0,0.5) circle (2 pt);
\draw [color=black!50!green, fill=black!50!green] (0,1.5) circle (2 pt);
\draw (0,0.5) to [bend left=90] (0,1.5);
\draw [->] (0.5,1) -- (1.5,1) node [midway, above] {surgery};
\begin{scope}[xshift = 2.2 cm]
\draw [ultra thick, color=black!50!green] (0,0) -- (0,0.45) to [bend left=90] (0,1.55) -- (0,2);
\draw [ultra thick, color=black!50!green] (0,0.55) to [bend left=90] (0,1.45) -- (0,0.55);
\end{scope}
\begin{scope}[xshift=3.5 cm]
\draw [ultra thick, color=black!50!green] (0,0) -- (0,2);
\draw [color=black!50!green, fill=black!50!green] (0,0.25) circle (2 pt);
\draw [color=black!50!green, fill=black!50!green] (0,0.75) circle (2 pt);
\draw [color=black!50!green, fill=black!50!green] (0,1.25) circle (2 pt);
\draw [color=black!50!green, fill=black!50!green] (0,1.75) circle (2 pt);
\draw (0,0.25) to [bend left=90] (0,1.25);
\draw (0,0.75) to [bend left=90] (0,1.75);
\draw [->] (0.5,1) -- (1.5,1) node [midway, above] {surgery};
\end{scope}
\begin{scope}[xshift = 5.7 cm]
\draw [ultra thick, color=black!50!green] (0,0) -- (0,0.2) to [bend left=90] (0,1.3) -- (0,1.7) to [bend right=90] (0,0.8) -- (0,1.2) to [bend right=90] (0,0.3) -- (0,0.7) to [bend left=90] (0,1.8) -- (0,2);
\end{scope}
\begin{scope}[xshift=7 cm]
\draw [ultra thick, color=black!50!green] (0,0) -- (0,1.2);
\draw [ultra thick, color=black!50!green] (0,1.6) -- (0,2);
\draw [color=black!50!green, fill=black!50!green] (0,0.2) circle (2 pt);
\draw [color=black!50!green, fill=black!50!green] (0,0.6) circle (2 pt);
\draw [color=black!50!green, fill=black!50!green] (0,1) circle (2 pt);
\draw [color=black!50!green, fill=black!50!green] (0,1.8) circle (2 pt);
\draw (0,0.2) to [bend left=90] (0,1);
\draw (0,0.6) to [bend left=90] (0,1.8);
\draw [->] (0.5,1) -- (1.5,1) node [midway, above] {surgery};
\end{scope}
\begin{scope}[xshift = 9.2 cm]
\draw [ultra thick, color=black!50!green] (0,0) -- (0,0.15) to [bend left=90] (0,1.05) -- (0,1.2);
\draw [ultra thick, color=black!50!green] (0,1.6) -- (0,1.75) to [bend right=90] (0,0.65) -- (0,0.95) to [bend right=90] (0,0.25) -- (0,0.55) to [bend left=90] (0,1.85) -- (0,2);
\end{scope}
\end{tikzpicture}
\caption{Oriented surgery. Left: surgery yields a closed loop, so this is not an arc diagram. Centre and right: two examples of arc diagrams.}
\label{fig:0-surgery}
\end{center}
\end{figure}

An arc diagram is equivalent to a tape graph. From an arc diagram, we may collapse each line segment $Z_i$ to a vertex, and regard matched pairs as connecting pairs of vertices by edges. The orientation on each line segment provides a total ordering on total ordering of the half-edges incident to each vertex. Conversely, from a tape graph we may ``blow up" each vertex as in \cite[sec. 4.2]{Me14_twisty_itsy_bitsy} into an oriented line segment, with incident half-edges ordered along it, as in figure \ref{fig:blowing_up}(left).

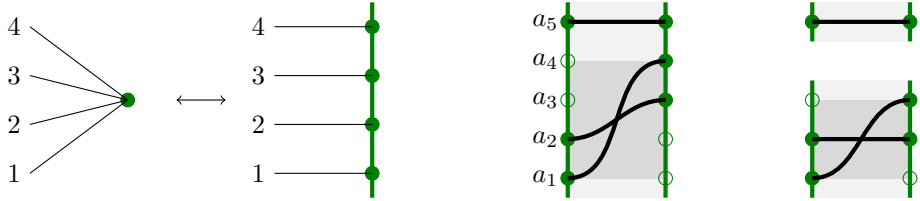
\begin{figure}
\begin{center}
\begin{tikzpicture}[scale=1.3]
\draw [color=black!50!green, fill=black!50!green] (0,0) circle (2 pt);
\draw (0,0) -- (-1,0.75) node[left] {$4$};
\draw (0,0) -- (-1,0.25) node[left] {$3$};
\draw (0,0) -- (-1,-0.25) node[left] {$2$};
\draw (0,0) -- (-1,-0.75) node[left] {$1$};
\draw [<->] (0.5,0) -- (1,0);
\begin{scope}[xshift=2.5cm]
\draw [ultra thick, color=black!50!green] (0,-1) -- (0,1);
\draw [color=black!50!green, fill=black!50!green] (0,0.75) circle (2 pt);
\draw [color=black!50!green, fill=black!50!green] (0,0.25) circle (2 pt);
\draw [color=black!50!green, fill=black!50!green] (0,-0.25) circle (2 pt);
\draw [color=black!50!green, fill=black!50!green] (0,-0.75) circle (2 pt);
\draw (0,0.75) -- (-1,0.75) node[left] {$4$};
\draw (0,0.25) -- (-1,0.25) node[left] {$3$};
\draw (0,-0.25) -- (-1,-0.25) node[left] {$2$};
\draw (0,-0.75) -- (-1,-0.75) node[left] {$1$};
\end{scope}

\begin{scope}[xshift=4.5cm, yshift = -1 cm]
\draw [draw=none, fill=gray!10!white] (0,0) -- (0,2) -- (1,2) -- (1,0) -- cycle;
\draw [draw=none, fill=gray!30!white] (0,0.2) -- (0,1.4) -- (1,1.4) -- (1,0.2) -- cycle;
\draw [ultra thick, color=black!50!green] (0,0) -- (0,2);
\draw [ultra thick, color=black!50!green] (1,0) -- (1,2);
\draw [color=black!50!green, fill=black!50!green] (0,0.2) circle (2 pt) node[left, color=black] {$a_1$};
\draw [color=black!50!green, fill=black!50!green] (0,0.6) circle (2 pt) node [left, color=black] {$a_2$};
\draw [color=black!50!green, fill=none] (0,1) circle (2 pt) node[left, color=black] {$a_3$};
\draw [color=black!50!green, fill=none] (0,1.4) circle (2 pt) node[left, color=black] {$a_4$};
\draw [color=black!50!green, fill=black!50!green] (0,1.8) circle (2 pt) node[left, color=black] {$a_5$};
\draw [color=black!50!green, fill=none] (1,0.2) circle (2 pt);
\draw [color=black!50!green, fill=none] (1,0.6) circle (2 pt);
\draw [color=black!50!green, fill=black!50!green] (1,1) circle (2 pt);
\draw [color=black!50!green, fill=black!50!green] (1,1.4) circle (2 pt);
\draw [color=black!50!green, fill=black!50!green] (1,1.8) circle (2 pt);
\draw [ultra thick]  (0,0.2) to [out=0, in=180] (1,1.4);
\draw [ultra thick]  (0,0.6) to [out=0, in=180] (1,1);
\draw [ultra thick]  (0,1.8) to [out=0, in=180] (1,1.8);
\end{scope}

\begin{scope}[xshift=7 cm, yshift = -1 cm]
\draw [draw=none, fill=gray!10!white] (0,0) -- (0,1.2) -- (1,1.2) -- (1,0) -- cycle;
\draw [draw=none, fill=gray!10!white] (0,1.6) -- (0,2) -- (1,2) -- (1,1.6) -- cycle;
\draw [draw=none, fill=gray!30!white] (0,0.2) -- (0,1) -- (1,1) -- (1,0.2) -- cycle;
\draw [ultra thick, color=black!50!green] (0,0) -- (0,1.2);
\draw [ultra thick, color=black!50!green] (0,1.6) -- (0,2);
\draw [ultra thick, color=black!50!green] (1,0) -- (1,1.2);
\draw [ultra thick, color=black!50!green] (1,1.6) -- (1,2);
\draw [color=black!50!green, fill=black!50!green] (0,0.2) circle (2 pt);
\draw [color=black!50!green, fill=black!50!green] (0,0.6) circle (2 pt);
\draw [color=black!50!green, fill=none] (0,1) circle (2 pt);
\draw [color=black!50!green, fill=black!50!green] (0,1.8) circle (2 pt);
\draw [color=black!50!green, fill=none] (1,0.2) circle (2 pt);
\draw [color=black!50!green, fill=black!50!green] (1,0.6) circle (2 pt);
\draw [color=black!50!green, fill=black!50!green] (1,1) circle (2 pt);
\draw [color=black!50!green, fill=black!50!green] (1,1.8) circle (2 pt);
\draw [ultra thick]  (0,0.2) to [out=0, in=180] (1,1);
\draw [ultra thick]  (0,0.6) to [out=0, in=180] (1,0.6);
\draw [ultra thick]  (0,1.8) to [out=0, in=180] (1,1.8);
\end{scope}
\end{tikzpicture}
\caption{Left: equivalence of tape graphs and arc diagrams. Centre: strand diagram with $3$ strands, $5$ places and $1$ inversion. Right: strand diagram with $3$ strands, $(3,1)$ places and 1 inversion.}
\label{fig:unconstrained_strands_map}
\label{fig:blowing_up}
\end{center}
\end{figure}

\subsection{Algebra of strands}
\label{sec:algebra_of_strands}

Let $k \geq 0$ and $n \geq 1$ be integers. An \emph{(unconstrained) strand map with $k$ strands on $n$ places} is a triple $\mu = (S,T,\phi)$, where $S,T \subseteq \{1, 2, \ldots, n\}$, $|S| = |T| = k$, and $\phi: S \To T$ is a non-decreasing bijection. 

We can draw a \emph{strand diagram} (or Reeb chord description) of $\mu$ as follows. Label $n$ points in order on an oriented interval $Z$ as $\{a_1, a_2, \ldots, a_n\}$. The diagram is drawn in $Z \times [0,1]$ and consists of an arc with non-negative slope (since $\phi$ is non-decreasing) from $(i,0)$ to $(\phi(i), 1)$, for each $i \in S$. If each of these $k$ arcs, or \emph{strands}, is drawn transversely, and they meet efficiently without triple crossings, then the number of crossings in the diagram is the number of \emph{inversions} of $\phi$, i.e. the number of pairs $(i,j)$ such that $i<j$ and $\phi(i) > \phi(j)$. The set of inversions of $\mu$ is denoted by $\Inv(\mu)$ and its cardinality by $\inv(\mu)$. We say $\mu$ \emph{begins} at $S$ and \emph{ends} at $T$, or \emph{goes from $S$ to $T$}; we indicate $S$ and $T$ in a strand diagram by drawing the corresponding points $a_i$ filled-in. The points $a_i$ break $Z$ into consecutive sub-intervals which we call the \emph{steps} of $Z$; we always draw the $a_i$ in the interior of $Z$ so that there are \emph{interior} steps $[a_i, a_{i+1}]$ and \emph{exterior} steps at the ends of the interval. An interior step is \emph{used} if some strand's vertical coordinate passes through the step, i.e. if there is some $j \in S$ with $j \leq i$ and $\phi(j) \geq i+1$; otherwise we say it is \emph{unused}. We indicate a used step $[a_i, a_{i+1}]$ by shading $[a_i, a_{i+1}] \times [0,1]$ darker in a strand diagram. See figure \ref{fig:unconstrained_strands_map}(centre). In practice we use strand diagrams and maps interchangeably.

The \emph{(unconstrained) strand algebra with $k$ strands and $n$ places} $\A(n,k)$ is a $\Z_2$ algebra, generated freely as a $\Z_2$-module by the strand maps with $k$ strands on $n$ places. The product of two strand maps is roughly their composition, if it is defined and has no ``excess inversions", otherwise it is zero. More precisely, let $\mu = (S,T,\phi)$ and $\nu = (U,V,\psi)$. If $T = U$ and $\inv(\psi \circ \phi) = \inv(\phi) + \inv(\psi)$, then $\mu \cdot \nu = (S,V, \psi \circ \phi)$; otherwise $\mu \cdot \nu = 0$. The product can be obtained by concatenating strand diagrams from left to right, but if strands do not match, or two strands intersect twice, the result is zero. 

The strand diagrams consisting entirely of horizontal strands are the strand maps of the form $I(S) = (S,S,1_S)$, where $S \subseteq \{1, \ldots, n\}$. Each $I(S)$ is a left and right idempotent of $\A(n,k)$; the $\Z_2$-submodule $I(S) \cdot \A(n,k) \cdot I(T)$ has basis the strand diagrams going from $S$ to $T$.

The algebra $\A(n,k)$ has a differential $\partial$ which roughly ``resolves crossings" in strand diagrams. Each crossing/inversion may be resolved in a unique way to obtain another strand diagram with less inversions. Then $\partial \mu$ is the sum of all strand diagrams obtained from $\mu$ by resolving a crossing such that the number of inversions decreases by exactly $1$. If $\mu$ has no crossings then $\partial \mu = 0$. This differential satisfies $\partial^2 = 0$ and the Leibniz rule \cite{LOT08}.

The \emph{(unconstrained) strand algebra with $k$ strands and $(n_1, \ldots, n_l)$ places} is the $\Z_2$-algebra given by
\[
\A(n_1, n_2, \ldots, n_l ; k) = \bigoplus_{k_1, \ldots, k_l} \A (n_1, k_1) \otimes \cdots \otimes \A(n_l, k_l),
\]
where the direct sum is over integers $k_1, \ldots, k_l \geq 0$ such that $k_1 + \cdots + k_l = k$. We consider $n_1, \ldots, n_l$ places lying on separate intervals $Z_1, \ldots, Z_l$. As a $\Z_2$-module, $\A(n_1, \ldots, n_l; k)$ is generated by \emph{(unconstrained) strand maps with $k$ strands and $(n_1, \ldots, n_l)$ places}, which are strand maps $\mu = (S,T,\phi)$ with $k$ strands and $n_1 + \cdots + n_l$ places, such that for each $i \in S$, $i$ and $\phi(i)$ lie on the same interval. A strand diagram of such a strand map can naturally be drawn as in figure \ref{fig:unconstrained_strands_map}(right). Notions of inversion, multiplication, used steps, and the differential carry over to $\A(n_1, \ldots, n_l; k)$ immediately.

\subsection{Algebra associated to an arc diagram}

We now define an algebra associated to an arc diagram $\mathcal{Z} = ({\bf Z}, {\bf a}, M)$, where ${\bf Z} = (Z_1, \ldots, Z_l)$ and ${\bf a} = (a_1, \ldots, a_{2k})$. Roughly, $\ZZ$ \emph{constrains} the arc diagrams discussed above for the intervals $Z_j$: strands cannot start or end at both points of a pair matched by $M$; and they must be ``symmetric" with respect to matched pairs, in a certain sense. The places of ${\bf a}$, ordered along ${\bf Z}$, can be regarded as the places of a strand map; let there be $n_j$ places on $Z_j$, so $n_1 + \cdots + n_l = 2k$. The algebra $\A(\ZZ)$ is  related to the algebras $\A(n_1, \ldots, n_l; \cdot)$.

Given a set $s \subseteq \{1, \ldots, k\}$ of size $i$, there are $2^{i}$ subsets $S$ of ${\bf a}$ which are \emph{sections} of $s$, i.e. such that $M$ restricts to a bijection $S \To s$: $s$ can be regarded a subset of the arcs of $\ZZ$ joining matched points, and each $S$ corresponds to choosing an endpoint of each arc. Throughout, we will use lower case letters to refer to subsets of $\{1, \ldots, k\}$, and upper case to refer to their sections. Each $S \subseteq {\bf a}$ defines an idempotent $I(S)$ of $\A(n_1, \ldots, n_l; i)$, consisting of horizontal strands at $S$. Adding these up over all $2^i$ sections $S$ of $s$, we obtain another idempotent $I(s)$ of $\A(n_1, \ldots, n_l; i)$, ``symmetrised" with respect to the matching; and then adding up all \emph{these} over the subsets $s \subseteq \{ 1, \ldots, k \}$ of size $i$, we obtain an idempotent $I_i$:
\[
I(s) = \sum_{S \text{ a section of } s} I(S), \quad \quad \quad
I_i = \sum_{\stackrel{s \subseteq \{1, \ldots, k\}}{|s|=i}} I(s).
\]
The $I(s)$ and the $I_i$ are orthogonal: $I(s)I(t) = I(s)$ if $s=t$, and is otherwise zero; and $I_i I_j = I_i$, if $i=j$, and is otherwise zero. The \emph{ring $\mathcal{I}(\mathcal{Z},i)$ of $\ZZ$-constrained $i$-strand idempotents} is the $\Z_2$-subalgebra of $\A(|Z_1|, \ldots, |Z_l|;i)$ generated by $I(s)$, over all $i$-element sets $s \subset \{1, \ldots, k\}$. 

We say a strand map \emph{begins} at $s \subseteq \{1, \ldots, k\}$ if it begins at a section of $s$, and \emph{ends} at $t \subseteq \{1, \ldots, k\}$ if it ends at a section of $t$; similarly we say that it \emph{goes from $s$ to $t$}. A \emph{$\ZZ$-constrained strand map} is a strand map on $(n_1, \ldots, n_l)$ places which begins at some $s$ and ends at some $t$. In other words, a strand map is $\ZZ$-constrained if it begins and ends at subsets of ${\bf a}$ which contain no matched pairs, i.e. on which $M$ is injective. See figure \ref{fig:constrained_strands_diagrams}.

\begin{figure}
\begin{center}
\begin{tikzpicture}[scale=1.4]
\begin{scope}[xshift=-5 cm]
\draw [draw=none, fill=gray!10!white] (0,0) -- (0,1.2) -- (1,1.2) -- (1,0) -- cycle;
\draw [draw=none, fill=gray!10!white] (0,1.6) -- (0,2) -- (1,2) -- (1,1.6) -- cycle;
\draw [draw=none, fill=gray!30!white] (0,0.2) -- (0,0.6) -- (1,0.6) -- (1,0.2) -- cycle;
\draw [ultra thick, color=black!50!green] (0,0) -- (0,1.2);
\draw [ultra thick, color=black!50!green] (0,1.6) -- (0,2);
\draw [color=black!50!green, fill=black!50!green] (0,0.2) circle (2 pt);
\draw [color=black!50!green, fill=black!50!green] (0,0.6) circle (2 pt);
\draw [color=black!50!green, fill=black!50!green] (0,1) circle (2 pt);
\draw [color=black!50!green, fill=black!50!green] (0,1.8) circle (2 pt);
\draw (0,0.2) to [bend left=90] (0,1);
\draw (0,0.6) to [bend left=90] (0,1.8);
\draw [ultra thick, color=black!50!green] (0,0) -- (0,1.2);
\draw [ultra thick, color=black!50!green] (0,1.6) -- (0,2);
\draw [ultra thick, color=black!50!green] (1,0) -- (1,1.2);
\draw [ultra thick, color=black!50!green] (1,1.6) -- (1,2);
\draw [color=black!50!green, fill=black!50!green] (0,0.2) circle (2 pt);
\draw [color=black!50!green, fill=black!50!green] (0,0.6) circle (2 pt);
\draw [color=black!50!green, fill=black!50!green] (0,1) circle (2 pt);
\draw [color=black!50!green, fill=black!50!green] (0,1.8) circle (2 pt);
\draw [color=black!50!green, fill=black!50!green] (1,0.2) circle (2 pt);
\draw [color=black!50!green, fill=black!50!green] (1,0.6) circle (2 pt);
\draw [color=black!50!green, fill=black!50!green] (1,1) circle (2 pt);
\draw [color=black!50!green, fill=black!50!green] (1,1.8) circle (2 pt);
\draw [ultra thick]  (0,0.2) to [out=0, in=180] (1,0.6);
\draw [ultra thick]  (0,1) to [out=0, in=180] (1,1);
\end{scope}
\begin{scope}[xshift=-2.5 cm]
\draw [draw=none, fill=gray!10!white] (0,0) -- (0,1.2) -- (1,1.2) -- (1,0) -- cycle;
\draw [draw=none, fill=gray!10!white] (0,1.6) -- (0,2) -- (1,2) -- (1,1.6) -- cycle;
\draw [draw=none, fill=gray!30!white] (0,0.2) -- (0,1) -- (1,1) -- (1,0.2) -- cycle;
\draw [ultra thick, color=black!50!green] (0,0) -- (0,1.2);
\draw [ultra thick, color=black!50!green] (0,1.6) -- (0,2);
\draw [color=black!50!green, fill=black!50!green] (0,0.2) circle (2 pt);
\draw [color=black!50!green, fill=black!50!green] (0,0.6) circle (2 pt);
\draw [color=black!50!green, fill=black!50!green] (0,1) circle (2 pt);
\draw [color=black!50!green, fill=black!50!green] (0,1.8) circle (2 pt);
\draw (0,0.2) to [bend left=90] (0,1);
\draw (0,0.6) to [bend left=90] (0,1.8);
\draw [ultra thick, color=black!50!green] (0,0) -- (0,1.2);
\draw [ultra thick, color=black!50!green] (0,1.6) -- (0,2);
\draw [ultra thick, color=black!50!green] (1,0) -- (1,1.2);
\draw [ultra thick, color=black!50!green] (1,1.6) -- (1,2);
\draw [color=black!50!green, fill=black!50!green] (0,0.2) circle (2 pt);
\draw [color=black!50!green, fill=black!50!green] (0,0.6) circle (2 pt);
\draw [color=black!50!green, fill=black!50!green] (0,1) circle (2 pt);
\draw [color=black!50!green, fill=black!50!green] (0,1.8) circle (2 pt);
\draw [color=black!50!green, fill=black!50!green] (1,0.2) circle (2 pt);
\draw [color=black!50!green, fill=black!50!green] (1,0.6) circle (2 pt);
\draw [color=black!50!green, fill=black!50!green] (1,1) circle (2 pt);
\draw [color=black!50!green, fill=black!50!green] (1,1.8) circle (2 pt);
\draw [ultra thick]  (0,0.2) to [out=0, in=180] (1,1);
\draw [ultra thick]  (0,0.6) to [out=0, in=180] (1,0.6);
\end{scope}
\end{tikzpicture}
\caption{Left: A strand diagram which is not constrained by $\mathcal{Z}$, as it begins at two matched points. Right: A $\ZZ$-constrained strand diagram.}
\label{fig:constrained_strands_diagrams}
\end{center}
\end{figure}
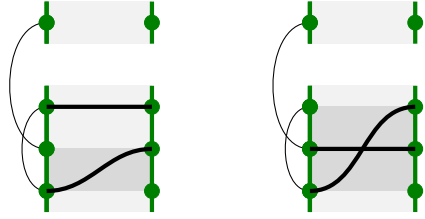

As $\Z_2$-module, $I(s) \cdot \A(n_1, \ldots, n_l; i) \cdot I(t)$ is freely generated by diagrams of $i$ strands on $(n_1, \ldots, n_l)$ places which begin at $s$ and end at $t$. Similarly, $I_i \cdot \A(n_1, \ldots, n_l; i) \cdot I_i$ has basis given by $i$-strand $\ZZ$-constrained diagrams. Thus $I_i \mu I_i = \mu$ or $0$, accordingly as $\mu$ is $i$-strand $\ZZ$-constrained or not. 

We also require that horizontal strands be ``symmetrised" with respect to  matched pairs. To this end, suppose $\mu =(A,B,\phi)$ is an unconstrained strand map on $(n_1, \ldots, n_l)$ places, where $\phi$ is strictly increasing (i.e. has no horizontal strands), and consider adding horizontal strands to $\phi$ at some places $U$, so as to obtain a strand map of $i$ strands. Such $U$ are subsets of $\{1,\ldots, 2k\}$ disjoint from $A$ and $B$, of size $i - |A| = i - |B|$. Adjoining horizontal strands to $\phi$ at $U$ results in a strand map $(A \cup U, B \cup U, \phi_U)$, where $\phi_U|_A = \phi$ and $\phi_U|_U = 1_U$. Then we can define the following sum over such $U$
\[
a_i (\mu) = a_i (A,B,\phi) = \sum_{U} \left( A \cup U, B \cup U, \phi_U \right) \in \A (n_1, \ldots, n_l ; i).
\]
Thus, $a_i(\mu)$ is the sum of all possible $i$-strand diagrams on $(n_1, \ldots, n_l)$ places obtained from $\mu$ by adding horizontal strands. ``Constraining" by multiplying by $I_i$ on left and right means that $I_i \cdot a_i (\mu) \cdot I_i$ is the sum of all possible $\ZZ$-constrained $i$-strand diagrams obtained from $\mu$ by adding horizontal strands. Constraining further by multiplying, we see that $I(s) a_i (\mu) I(t)$ is them sum of such diagrams which begin at $s$ and end at $t$. This leads to the following definition.
\begin{defn}
The \emph{$i$-strand algebra $\mathcal{A}(\mathcal{Z},i)$ of the arc diagram} $\mathcal{Z}$ is the $\Z_2$ subalgebra of $\A(n_1, \ldots, n_l; i)$ generated by $\mathcal{I}(\mathcal{Z},i)$ and the elements $I_i \cdot a_i (\mu) \cdot I_i$, over all strictly increasing strand maps $\mu$ with at most $i$ strands on $(n_1, \ldots, n_l)$ places.

The \emph{strand algebra of the arc diagram} $\mathcal{Z}$ is the direct sum
\[
\A ( \mathcal{Z} ) = \bigoplus_{i=0}^k \A (\mathcal{Z},i).
\]
\end{defn}
As a $\Z_2$-module, $\mathcal{A}(\mathcal{Z},i)$ is generated by elements of the form $I(s) a_i (\mu) I(t)$, where $\mu$ is as in the definition above and $s,t \subseteq \{1, \ldots, k\}$. Suppose $\nu$ is a $\ZZ$-constrained $i$-strand diagram appearing in $I(s) a_i (\mu) I(t)$, i.e. obtained from $\mu$ by adding horizontal strands so as to begin at $s$ and end at $t$; suppose further that $\nu$ has a horizontal strand at $p$. Then no strand begins or ends at its twin $p'$, so removing the horizontal strand at $p$ and replacing it with a horizontal strand at $p'$ results in another diagram appearing in $I(s) a_i (\mu) I(t)$. Indeed, $I(s) a_i (\mu) I(t)$ consists precisely of the diagrams obtained from $\nu$ by replacing horizontal strands in this way. If $\mu$ has $i-j$ increasing strands, then $\nu$ has $j$ horizontal strands, and $I(s) \cdot a_i (\mu) \cdot I(t)$ is the sum of the $2^j$ diagrams obtained by replacements of horizontal strands. We draw a diagram of $I(s) a_i (\mu) I(t)$ by drawing the $j$ pairs of horizontal strands \emph{dotted}; we refer to this as a \emph{symmetrised $\ZZ$-constrained strand diagram}. See figure \ref{fig:symmetrised_diagram}.

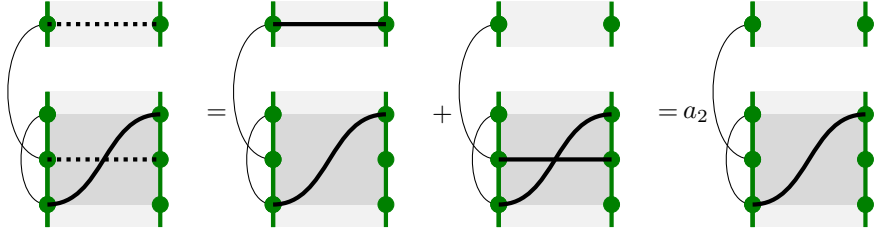
\begin{figure}
\begin{center}
\begin{tikzpicture}[scale=1.5]
\draw [draw=none, fill=gray!10!white] (0,0) -- (0,1.2) -- (1,1.2) -- (1,0) -- cycle;
\draw [draw=none, fill=gray!10!white] (0,1.6) -- (0,2) -- (1,2) -- (1,1.6) -- cycle;
\draw [draw=none, fill=gray!30!white] (0,0.2) -- (0,1) -- (1,1) -- (1,0.2) -- cycle;
\draw [ultra thick, color=black!50!green] (0,0) -- (0,1.2);
\draw [ultra thick, color=black!50!green] (0,1.6) -- (0,2);
\draw [color=black!50!green, fill=black!50!green] (0,0.2) circle (2 pt);
\draw [color=black!50!green, fill=black!50!green] (0,0.6) circle (2 pt);
\draw [color=black!50!green, fill=black!50!green] (0,1) circle (2 pt);
\draw [color=black!50!green, fill=black!50!green] (0,1.8) circle (2 pt);
\draw (0,0.2) to [bend left=90] (0,1);
\draw (0,0.6) to [bend left=90] (0,1.8);
\draw [ultra thick, color=black!50!green] (0,0) -- (0,1.2);
\draw [ultra thick, color=black!50!green] (0,1.6) -- (0,2);
\draw [ultra thick, color=black!50!green] (1,0) -- (1,1.2);
\draw [ultra thick, color=black!50!green] (1,1.6) -- (1,2);
\draw [color=black!50!green, fill=black!50!green] (0,0.2) circle (2 pt);
\draw [color=black!50!green, fill=black!50!green] (0,0.6) circle (2 pt);
\draw [color=black!50!green, fill=black!50!green] (0,1) circle (2 pt);
\draw [color=black!50!green, fill=black!50!green] (0,1.8) circle (2 pt);
\draw [color=black!50!green, fill=black!50!green] (1,0.2) circle (2 pt);
\draw [color=black!50!green, fill=black!50!green] (1,0.6) circle (2 pt);
\draw [color=black!50!green, fill=black!50!green] (1,1) circle (2 pt);
\draw [color=black!50!green, fill=black!50!green] (1,1.8) circle (2 pt);
\draw [ultra thick]  (0,0.2) to [out=0, in=180] (1,1);
\draw [ultra thick, dotted]  (0,0.6) to [out=0, in=180] (1,0.6);
\draw [ultra thick, dotted]  (0,1.8) to [out=0, in=180] (1,1.8);
\draw (1.5,1) node {$=$};
\begin{scope}[xshift = 2cm]
\draw [draw=none, fill=gray!10!white] (0,0) -- (0,1.2) -- (1,1.2) -- (1,0) -- cycle;
\draw [draw=none, fill=gray!10!white] (0,1.6) -- (0,2) -- (1,2) -- (1,1.6) -- cycle;
\draw [draw=none, fill=gray!30!white] (0,0.2) -- (0,1) -- (1,1) -- (1,0.2) -- cycle;
\draw [ultra thick, color=black!50!green] (0,0) -- (0,1.2);
\draw [ultra thick, color=black!50!green] (0,1.6) -- (0,2);
\draw [color=black!50!green, fill=black!50!green] (0,0.2) circle (2 pt);
\draw [color=black!50!green, fill=black!50!green] (0,0.6) circle (2 pt);
\draw [color=black!50!green, fill=black!50!green] (0,1) circle (2 pt);
\draw [color=black!50!green, fill=black!50!green] (0,1.8) circle (2 pt);
\draw (0,0.2) to [bend left=90] (0,1);
\draw (0,0.6) to [bend left=90] (0,1.8);
\draw [ultra thick, color=black!50!green] (0,0) -- (0,1.2);
\draw [ultra thick, color=black!50!green] (0,1.6) -- (0,2);
\draw [ultra thick, color=black!50!green] (1,0) -- (1,1.2);
\draw [ultra thick, color=black!50!green] (1,1.6) -- (1,2);
\draw [color=black!50!green, fill=black!50!green] (0,0.2) circle (2 pt);
\draw [color=black!50!green, fill=black!50!green] (0,0.6) circle (2 pt);
\draw [color=black!50!green, fill=black!50!green] (0,1) circle (2 pt);
\draw [color=black!50!green, fill=black!50!green] (0,1.8) circle (2 pt);
\draw [color=black!50!green, fill=black!50!green] (1,0.2) circle (2 pt);
\draw [color=black!50!green, fill=black!50!green] (1,0.6) circle (2 pt);
\draw [color=black!50!green, fill=black!50!green] (1,1) circle (2 pt);
\draw [color=black!50!green, fill=black!50!green] (1,1.8) circle (2 pt);
\draw [ultra thick]  (0,0.2) to [out=0, in=180] (1,1);
\draw [ultra thick]  (0,1.8) to [out=0, in=180] (1,1.8);
\draw (1.5,1) node {$+$};
\end{scope}
\begin{scope}[xshift = 4cm]
\draw [draw=none, fill=gray!10!white] (0,0) -- (0,1.2) -- (1,1.2) -- (1,0) -- cycle;
\draw [draw=none, fill=gray!10!white] (0,1.6) -- (0,2) -- (1,2) -- (1,1.6) -- cycle;
\draw [draw=none, fill=gray!30!white] (0,0.2) -- (0,1) -- (1,1) -- (1,0.2) -- cycle;
\draw [ultra thick, color=black!50!green] (0,0) -- (0,1.2);
\draw [ultra thick, color=black!50!green] (0,1.6) -- (0,2);
\draw [color=black!50!green, fill=black!50!green] (0,0.2) circle (2 pt);
\draw [color=black!50!green, fill=black!50!green] (0,0.6) circle (2 pt);
\draw [color=black!50!green, fill=black!50!green] (0,1) circle (2 pt);
\draw [color=black!50!green, fill=black!50!green] (0,1.8) circle (2 pt);
\draw (0,0.2) to [bend left=90] (0,1);
\draw (0,0.6) to [bend left=90] (0,1.8);
\draw [ultra thick, color=black!50!green] (0,0) -- (0,1.2);
\draw [ultra thick, color=black!50!green] (0,1.6) -- (0,2);
\draw [ultra thick, color=black!50!green] (1,0) -- (1,1.2);
\draw [ultra thick, color=black!50!green] (1,1.6) -- (1,2);
\draw [color=black!50!green, fill=black!50!green] (0,0.2) circle (2 pt);
\draw [color=black!50!green, fill=black!50!green] (0,0.6) circle (2 pt);
\draw [color=black!50!green, fill=black!50!green] (0,1) circle (2 pt);
\draw [color=black!50!green, fill=black!50!green] (0,1.8) circle (2 pt);
\draw [color=black!50!green, fill=black!50!green] (1,0.2) circle (2 pt);
\draw [color=black!50!green, fill=black!50!green] (1,0.6) circle (2 pt);
\draw [color=black!50!green, fill=black!50!green] (1,1) circle (2 pt);
\draw [color=black!50!green, fill=black!50!green] (1,1.8) circle (2 pt);
\draw [ultra thick]  (0,0.2) to [out=0, in=180] (1,1);
\draw [ultra thick]  (0,0.6) to [out=0, in=180] (1,0.6);
\draw (1.5,1) node {$=$};
\end{scope}
\begin{scope}[xshift = 6.25cm]
\draw (-0.5,1) node {$a_2$};
\draw [draw=none, fill=gray!10!white] (0,0) -- (0,1.2) -- (1,1.2) -- (1,0) -- cycle;
\draw [draw=none, fill=gray!10!white] (0,1.6) -- (0,2) -- (1,2) -- (1,1.6) -- cycle;
\draw [draw=none, fill=gray!30!white] (0,0.2) -- (0,1) -- (1,1) -- (1,0.2) -- cycle;
\draw [ultra thick, color=black!50!green] (0,0) -- (0,1.2);
\draw [ultra thick, color=black!50!green] (0,1.6) -- (0,2);
\draw [color=black!50!green, fill=black!50!green] (0,0.2) circle (2 pt);
\draw [color=black!50!green, fill=black!50!green] (0,0.6) circle (2 pt);
\draw [color=black!50!green, fill=black!50!green] (0,1) circle (2 pt);
\draw [color=black!50!green, fill=black!50!green] (0,1.8) circle (2 pt);
\draw (0,0.2) to [bend left=90] (0,1);
\draw (0,0.6) to [bend left=90] (0,1.8);
\draw [ultra thick, color=black!50!green] (0,0) -- (0,1.2);
\draw [ultra thick, color=black!50!green] (0,1.6) -- (0,2);
\draw [ultra thick, color=black!50!green] (1,0) -- (1,1.2);
\draw [ultra thick, color=black!50!green] (1,1.6) -- (1,2);
\draw [color=black!50!green, fill=black!50!green] (0,0.2) circle (2 pt);
\draw [color=black!50!green, fill=black!50!green] (0,0.6) circle (2 pt);
\draw [color=black!50!green, fill=black!50!green] (0,1) circle (2 pt);
\draw [color=black!50!green, fill=black!50!green] (0,1.8) circle (2 pt);
\draw [color=black!50!green, fill=black!50!green] (1,0.2) circle (2 pt);
\draw [color=black!50!green, fill=black!50!green] (1,0.6) circle (2 pt);
\draw [color=black!50!green, fill=black!50!green] (1,1) circle (2 pt);
\draw [color=black!50!green, fill=black!50!green] (1,1.8) circle (2 pt);
\draw [ultra thick]  (0,0.2) to [out=0, in=180] (1,1);
\end{scope}
\end{tikzpicture}
\caption{ Right: A symmetrised $\ZZ$-constrained strand diagram.}
\label{fig:symmetrised_diagram}
\end{center}
\end{figure}

As a $\Z_2$-module, $\mathcal{A}(\mathcal{Z},i)$ has basis the symmetrised $\ZZ$-constrained $i$-strand diagrams; and $\mathcal{A}(\mathcal{Z})$ has basis all symmetrised $\ZZ$-constrained strand diagrams. Multiplication is obtained by concatenating diagrams, the differential resolves crossings, and we can speak of used and unused steps, as before.

\subsection{Gradings on the strand algebra}
\label{sec:gradings}

The strand algebra $\mathcal{A}(\mathcal{Z},i)$ has some rather involved gradings. We only need some of these notions; see \cite{LOT08, LOT11_Bimodules, Zarev09} for details.

Fix an arc diagram $\mathcal{Z} = ({\bf Z}, {\bf a}, M)$, where ${\bf Z} = (Z_1, \ldots, Z_l)$ and ${\bf a} = (a_1, \ldots, a_{2k})$. A $\ZZ$-constrained strand map has, for our purposes, two gradings: a \emph{homological} or \emph{spin-c grading}, valued in $H_1 ({\bf Z}, {\bf a})$; and a \emph{Maslov grading} in $\frac{1}{2} \Z$. We will follow \cite{Zarev09}, which is slightly different from \cite{LOT08}.

Let $\mu = (S,T,\phi)$ be a $\ZZ$-constrained strand map. Then for each $a_i \in S$, both $a_i$ and $a_{\phi(a)}$ are places in ${\bf a}$, in the same interval $Z_j$, with $a<\phi(a)$, so there is a sub-interval $[a, \phi(a)] \subset Z_j$, giving a homology class in $H_1({\bf Z}, {\bf a})$. The \emph{homological grading} of $\mu$, denoted $[\mu]$, is given by the sum of these intervals $[a, \phi(a)]$.

The group $H_1 ({\bf Z}, {\bf a})$ has basis the interior steps of ${\bf Z}$. The number of times a step occurs in an element $\alpha$ of $H_1 ({\bf Z}, {\bf a})$ is called its \emph{multiplicity}. The steps with nonzero multiplicity form the \emph{support} of $\alpha$, denoted $\supp \alpha$. For a place $a_i \in {\bf a}$, the \emph{multiplicity} of $\alpha$ at $a_i$ is the average of the multiplicity of $\alpha$ on the steps immediately before and after $a_i$ (so lies in $\frac{1}{2} \Z$). Extending linearly, we obtain the multiplicity of $\alpha$ at any linear combination of points of ${\bf a}$, giving a linear map $m : H_0 ({\bf a}) \times H_1 ({\bf Z}, {\bf a}) \To \frac{1}{2} \Z$.

The homological grading $[\mu]$ of the $\ZZ$-constrained strand map $\mu$ is a non-negative integer combination of interior steps; its support consists of the used steps. Since horizontal strands contribute zero to the homological grading, symmetrised constrained strand maps have a well-defined homological grading. Moreover, homological grading is preserved by the differential, so $H(\A(\ZZ))$ splits over homological gradings.

The \emph{Maslov grading} of a $\ZZ$-constrained strand map $\mu = (S,T,\phi)$ is $\iota(\mu) = \inv(\phi) - m(S, [\mu]) \in \frac{1}{2} \Z$. One can check that the Maslov grading is preserved when we replace a horizontal strand at a point by a horizontal strand at its twin, so the Maslov grading is well-defined for symmetrised $\ZZ$-constrained strand diagrams. The full grading of $\mu$ is given by $(\iota(\mu), [\mu])$ and these pairs form a grading group $\Gr(\ZZ)$ which is $\frac{1}{2} \Z \times H_1 ({\bf Z}, {\bf a})$ with a certain non-abelian operation. With this grading, $\A(\ZZ)$ is a differential graded algebra \cite[prop. 2.14]{Zarev09}. 

The differential reduces the Maslov grading of a strand diagram by $1$. If we specify the homological grading $[\mu]$ and the beginning $S$ of a $\ZZ$-constrained strand diagram $\mu$, then $m(S,[\mu])$ is fixed so that $\iota(\mu) = \inv(\phi) - \text{constant}$.

\subsection{The homology of the strand algebra}

In \cite[thm. 9]{LOT11_Bimodules}, Lipshitz--Ozsv\'{a}th--Thurston gave a description of the homology of $\mathcal{A}(\mathcal{Z},i)$ for a pointed matched circle $\mathcal{Z}$. In order to generalise this description to general arc diagrams, we consider the relationship between the strand algebras of arc diagrams and pointed matched circles.

For our purposes, a \emph{pointed matched circle} is an arc diagram with one line segment, i.e. ${\bf Z} = (Z)$.

Let $\mathcal{Z}= ({\bf Z}, {\bf a}, M)$ be an arc diagram, where ${\bf Z} = (Z_1, \ldots, Z_l)$, $|Z_j| = n_j$, and ${\bf a} = (a_1, \ldots, a_{2k})$.

By gluing the intervals of $Z_j$ of ${\bf Z}$ together, we can obtain a single interval with places along it, matched in pairs; however this might not be a pointed matched circle, because of the surgery condition in the definition of an arc diagram. But by adding some additional intervals if necessary, with places matched in pairs, and gluing these together with the $Z_j$ in an appropriate fashion, one may obtain a pointed matched circle $\widehat{\mathcal{Z}} = ((Z), \widehat{{\bf a}}, \widehat{M})$ which contains $\mathcal{Z}$ as a sub-diagram. Specifically, $\mathcal{Z}$ is obtained from $\widehat{\mathcal{Z}}$ by splitting the interval $Z$ at various steps, and then removing some of the intervals, together with their matched pairs of places. Each of the intervals $Z_j$ can be regarded as sub-intervals of $Z$.

Now $\A(\ZZ,i)$ is a summand of $\A(\widehat{\ZZ},i)$ which can be defined purely in terms of homological grading and idempotents. A $\ZZ$-constrained strand diagram of $i$ strands can be regarded as a strand diagram, also of $i$ strands, constrained by $\widehat{\ZZ}$. And a $\widehat{\ZZ}$-constrained $i$-strand diagram can be regarded as a $\ZZ$-constrained $i$-strand diagram if and only if its homological grading is supported on the sub-intervals $Z_j$ of $Z$, and the strands begin and end at places on the $Z_j$. Symmetrising by taking $I(s) a_i (\mu) I(t)$ for $s,t$ corresponding to places on the $Z_j$, we then see that symmetrised $\ZZ$-constrained $i$-strand diagrams from $s$ to $t$, correspond precisely to symmetrised $\widehat{\ZZ}$-constrained $i$-strand diagrams from $s$ to $t$.

Letting $\A(\ZZ,i;h)$ denote the summand of $\A(\ZZ,i)$ with homological grading $h$, $\widehat{I}_i$ denote the sum of the idempotents $I(s)$ corresponding to places on the $Z_j$ where $|s|=i$, and noting that the differential preserves homological grading and beginning and ending idempotents, we thus have
\begin{equation}
\label{eqn:homology_from_pointed_matched_circle}
\A(\ZZ,i) = \bigoplus_{h,s,t} I(s) \A(\widehat{\ZZ},i;h) I(t)
\quad \text{and} \quad
H(\A(\ZZ,i)) = \bigoplus_{h,s,t} I(s) H(\A(\widehat{\ZZ},i;h)) I(t),
\end{equation}
where both direct sums are over $h$ supported on the $Z_j$, and $s,t$ corresponding to places on the $Z_j$.

Now as $\widehat{\ZZ}$ is a pointed matched circle, $H(\A(\widehat{\ZZ},i))$ is described by the theorem of Lipshitz--Ozsv\'{a}th--Thurston; so we obtain $H(\A(\ZZ,i))$ as the summand described above. We now state this theorem, adapted to our context.
\begin{thm}[\cite{LOT11_Bimodules} thm. 9]
\label{thm:LOT_theorem}
Let $\ZZ = ({\bf Z}, {\bf a}, M)$ be an arc diagram, let $s,t \subseteq \{1, \ldots, k\}$, and let $h \in H_1({\bf Z}, {\bf a})$. The homological-degree $h$ summand of $I(s) \cdot H(\A(\ZZ)) \cdot I(t)$ is nonzero if and only if the following conditions hold.
\begin{enumerate} 
\item The multiplicity of $h$ on each step of ${\bf Z}$ is $0$ or $1$.
\item If $v,w \in {\bf a}$ are matched by $M$, $v \in \Int(\supp(h))$, and $w \notin \Int(\supp(h))$, then $M(v) \notin s \cap t$.
\item There exists a symmetrised $\ZZ$-constrained strand diagram from $s$ to $t$ with homological grading $h$ without crossings.
\end{enumerate}
Moreover, if $h,s,t$ satisfy the conditions above, then the homological-degree $h$ summand of $I(s) \cdot H(\A(\ZZ)) \cdot I(t)$ is 1-dimensional, represented by any symmetrised $\ZZ$-constrained strand diagram from $s$ to $t$ with homological grading $h$ without crossings.
\end{thm}

\begin{proof}
Because of the decomposition in equation \ref{eqn:homology_from_pointed_matched_circle}, it suffices to check that the equivalence of our conditions with those of \cite{LOT11_Bimodules} for the appropriate homological gradings (i.e. those supported on the $Z_j$), for the algebra of the pointed matched circle $\widehat{\ZZ}$. 

Suppose the homological-degree $h$ summand of $I(s) \cdot H(\A(\widehat{\ZZ})) \cdot I(t)$ is nonzero; so there exists at least one $l \in \frac{1}{2} \Z$ such that the $(l,h)$-graded summand is nonzero. Then \cite{LOT11_Bimodules} states that (i) and (ii) hold. Moreover \cite{LOT11_Bimodules} asserts that when the $(l,h)$-graded summand is nonzero, it is 1-dimensional, represented by any crossingless diagram of that grading. Hence there exists a symmetrised $\widehat{\ZZ}$-constrained strand diagram from $s$ to $t$ with homological grading $h$ without crossings, so (iii) holds. 

The theorem in \cite{LOT11_Bimodules} also states that for the degree $(l,h)$ summand of $I(s) \cdot H(\A(\widehat{\ZZ})) \cdot I(t)$ to be nonzero, the Maslov degree must be minimal among symmetrised $\widehat{\ZZ}$-constrained strand diagrams from $s$ to $t$ with homological grading $h$. As mentioned in section \ref{sec:gradings}, once $h$, $s$ and $t$ are fixed, the Maslov grading is given by the number of crossings, minus a constant, and so the minimal Maslov degree is precisely the one with zero crossings. Thus there is precisely one Maslov grading $l$ such that the $(l,h)$ graded summand is nonzero. Hence the $h$-graded summand of $I(s) \cdot H(\A(\widehat{\ZZ})) \cdot I(t)$ is 1-dimensional, represented by any crossingless symmetrised $\ZZ$-constrained diagram from $s$ to $t$ with homological grading $h$.

Now suppose $h,s,t$ satisfy the conditions above. By (iii) there exists a symmetrised $\widehat{\ZZ}$-constrained strand diagram $\mu$ from $s$ to $t$ with homological grading $h$ and no crossings. Let $\mu$ have Maslov grading $l$. We show that $h,s,t,l$ satisfy conditions 1--4 of \cite[thm. 9]{LOT11_Bimodules}; conditions 2 and 3 are satisfied immediately by our conditions (i) and (ii). The existence of $\mu$ implies that $h$ is compatible with $I(s)$ and $I(t)$, in the sense of \cite[defn. 3.7]{LOT11_Bimodules}, so condition 1 is satisfied. Again, with $h,s,t$ are fixed, Maslov grading is number of crossings minus a constant; so since $\mu$ has no crossings, $l$ is minimal among all symmetrised $\widehat{\ZZ}$-constrained strand diagrams from $s$ to $t$ with homological grading $h$. Thus condition 4 holds. By \cite[thm. 9]{LOT11_Bimodules} then the degree $(l,h)$ summand of $I(s) \cdot H(\A(\widehat{\ZZ})) \cdot I(t)$ is nonzero, hence also the homological-degree $h$ summand.
\end{proof}

\subsection{Local description of homology}

We can now give a ``local" description of $H(\A(\ZZ))$. The idea is that if $h \in H_1({\bf Z}, {\bf a})$ and $s,t \subseteq \{1, \ldots, k\}$ satisfy conditions (i)---(iii) of theorem \ref{thm:LOT_theorem}, then we can say what a symmetrised $\ZZ$-constrained strand diagram must look like ``locally" near a pair of matched points.

More precisely, suppose $h,s,t$ satisfy conditions (i)--(iii) of theorem \ref{thm:LOT_theorem}, and let $v, w \in {\bf a}$ be two matched points, so $M(v)=M(w)$. We consider the fragment of $\ZZ$ at $v$, $w$ and their adjacent steps. 

As $h$ has multiplicity $0$ or $1$ on each step, it is specified by its set of used steps, or support; $h$ can be regarded as a collection of oriented sub-intervals of ${\bf Z}$. For any interval $I$, we write $\partial^+ I$ for its positive boundary (i.e. maximum), and $\partial^- I$ for its negative boundary (i.e. minimum); hence in homology $\partial I = \partial^+ I - \partial^- I$. Each point of ${\bf a}$ thus lies in exactly one of $\partial^- (\supp h)$, $\partial^+ (\supp h)$, $\Int (\supp h)$, or ${\bf a} \backslash (\supp h)$. 

Condition (iii) tells us that, if a place $v \in {\bf a}$ has the property that the step immediately before $v$ is not used by $h$, but the step immediately before $v$ is, then $M(v) \in s$. In other words, these conditions on $h$ imply that some strand must begin at $v$. Similarly, if the step before $p$ is used by $h$, but the step after $v$ is not, then $M(v)  \in t$. 

In general, $M(v)$ satisfies exactly one of $M(v) \in s \cap t$, $M(v) \in s \backslash t$, $M(v) \in t \backslash s$, or $M(v) \notin s \cup t$. We call the data of whether $v,w \in \partial^+(\supp h), \partial^- (\supp h), \Int(\supp h), {\bf a} \backslash (\supp h)$, and whether $M(v) \in s \cap t, s \backslash t, t \backslash s$ or $M(v) \notin s \cup t$, \emph{the data of $h,s,t$ near $v,w$}.

Considering the various possibilities, the fragment of $\ZZ$ at the $2$ places $v,w$ and the $4$ adjacent steps, must fall into precisely one of the following cases, which are also illustrated in figure \ref{fig:homology_generators_local}. We compute the data of $h,s,t$ near $v,w$ in each case.
\begin{enumerate}
\item
No steps used. In this case any strand appearing must be horizontal, hence symmetrised (dotted).
\begin{enumerate}
\item
There are symmetrised horizontal strands at $v, w$. Then $v,w \in {\bf a} \backslash (\supp h)$ and $M(v) \in s \cap t$.
\item
No strand begins or ends at $v$ or $w$. Then $v,w \in {\bf a} \backslash (\supp h)$ and $M(v) \notin s \cup t$.
\end{enumerate}
\item
One step used; say it is adjacent to $v$.
\begin{enumerate}
\item
Step after $v$ used, so a strand begins at $v$. Then $v \in \partial^- (\supp h)$ and $M(v) \in s \backslash t$.
\item
Step before $v$ used, so a strand ends at $v$. Then  $v \in \partial^+ (\supp h)$ and $M(v) \in t \backslash s$.
\end{enumerate}
\item
Two steps used. In this case the steps before $v$ and $w$ cannot both be used, for then a strand must end at $v$ and another strand must end at $w$, contradicting the $\ZZ$-constrained property. Similarly, the steps after $v$ and $w$ cannot both be used.
\begin{enumerate}
\item
Steps used are before and after distinct places, say after $v$ and before $w$, so a strand begins at $v$ and a strand ends at $w$. Then $v \in \partial^- (\supp h)$, $w \in \partial^+ (\supp h)$, and $M(v) \in s \cap t$.
\item
Steps used are before and after a single place, say $v$. Then $v \in \Int(\supp h )$ but $w \notin \Int(\supp h)$, so by condition (ii) then $M(v) \notin s \cap t$. Any strand at $w$ would have to be horizontal, contradicting this condition; and a strand begins at $v$ if and only if a strand ends at $v$, again contradicting this condition; hence no strand can begin or end at $v$ or $w$. So $v \in \Int(\supp h)$, $w \in {\bf a} \backslash (\supp h)$, and $M(v) \notin s \cup t$.
\end{enumerate}
\item
Three steps used; say the unused step is adjacent to $v$.
\begin{enumerate}
\item
The unused step is after $v$. Then a strand must end at $v$; the $\ZZ$-constrained condition them implies no strand ends at $w$; and so no strand can begin at $w$ either. Thus $v \in \partial^+ (\supp h)$, $w \in \Int(\supp h)$ and $M(v) \in t \backslash s$.
\item
The unused step is before $v$. Then a strand must begin at $v$; $\ZZ$-constraint then implies no strand begins at $w$; then no strand ends at $w$ either. Thus $v  \in \partial^-  (\supp h)$, $w \in \Int(\supp h)$ and $M(v) \in s \backslash t$.
\end{enumerate}
\item
All four steps used. Then a strand begins at $v$ if and only if a strand ends at $v$; and similarly for $w$. So $M(v)$ lies in neither or both of $s$ and $t$.
\begin{enumerate}
\item
No strand begins or ends at $v$ or $w$. So $v,w \in \Int(\supp h)$ and $(v) \notin s \cup t$.
\item
If $M(v)$ lies in both $s$ and $t$, there are two possibilities: either a strand begins at $v$, a strand ends at $v$, and no strand begins or ends at $w$; or a strand begins at $w$, a strand ends at $w$, and no strand begins or ends at $v$. Either way, $v,w \in \Int(\supp h)$ and $M(p) \notin s \cup t$.
\end{enumerate}
\end{enumerate}
(This classification into cases parallels the classification in section \ref{sec:tightness_of_cubes}.)

In other words, in a symmetrised $\ZZ$-constrained strand diagram which is nonzero in $H(\A(\ZZ))$, near every pair of matched places the diagram must look like one of the cases in figure \ref{fig:homology_generators_local} (up to a possible relabelling of $v$ and $w$). We also observe that, conversely, if a strand diagram (a priori unconstrained) looks near every pair of matched places like one of these cases, then it is $\ZZ$-constrained, symmetrised, and satisfies the conditions of theorem \ref{thm:LOT_theorem}, hence is nonzero in $H(\A(\ZZ))$.

If a diagram exists which is nonzero in $H(\A(\ZZ))$, goes from $s$ to $t$, and has homological grading $h$, then the diagram is completely determined near each pair of matched places $v,w$ by the data of $h,s,t$ near $v,w$ --- \emph{except} in the very final case where both $v,w \in \Int(\supp h)$ and $M(p) \notin s \cup t$. There are then two possibilities, which are shown at the bottom right of figure \ref{fig:homology_generators_local}. However, \emph{in homology} these two diagrams are equal, because of the equation shown in figure \ref{fig:equal_in_homology}. (Although this only shows the strand diagram near $v$ and $w$, if we have larger strand diagrams which are equal elsewhere, have no further crossings, and are as shown near $v$ and $w$, then the equation still holds.) 

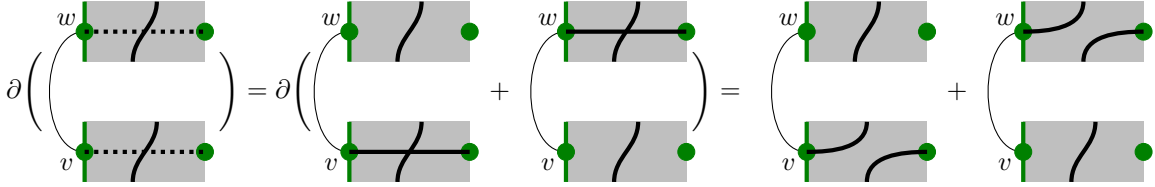
\begin{figure}
\begin{center}
\begin{tikzpicture}[mypersp, scale=1.6]
\begin{scope}[xzp=0, xshift = -2 cm]
\draw (-0.5,0.75) node {$\partial \Bigg($};
\strandbackgroundshading
\beforevused
\aftervused
\beforewused
\afterwused
\cubestrandsetup
\lefton
\righton
\draw [ultra thick]  (0.4,0) to [out=90, in=-90] (0.6,0.5);
\draw [ultra thick, dotted]  (0,0.25) -- (1,0.25); 
\draw [ultra thick, dotted]  (0,1.25) -- (1,1.25); 
\draw [ultra thick]  (0.4,1) to [out=90, in=-90] (0.6,1.5);
\draw (1.5,0.75) node {$\Bigg) = \partial \Bigg($};
\end{scope}
\begin{scope}[xzp=0, xshift=0.2 cm]
\strandbackgroundshading
\beforevused
\aftervused
\beforewused
\afterwused
\cubestrandsetup
\lefton
\righton
\draw [ultra thick]  (0.4,0) to [out=90, in=-90] (0.6,0.5);
\draw [ultra thick]  (0,0.25) -- (1,0.25); 
\draw [ultra thick]  (0.4,1) to [out=90, in=-90] (0.6,1.5);
\draw (1.25,0.75) node {$+$};
\end{scope}
\begin{scope}[xzp=0, xshift=2 cm]
\strandbackgroundshading
\beforevused
\aftervused
\beforewused
\afterwused
\cubestrandsetup
\lefton
\righton
\draw [ultra thick]  (0.4,0) to [out=90, in=-90] (0.6,0.5);
\draw [ultra thick]  (0,1.25) -- (1,1.25); 
\draw [ultra thick]  (0.4,1) to [out=90, in=-90] (0.6,1.5);
\draw (1.25,0.75) node {$\Bigg) =$};
\end{scope}
\begin{scope}[xshift = 4 cm, xzp=0]
\strandbackgroundshading
\beforevused
\aftervused
\beforewused
\afterwused
\cubestrandsetup
\lefton
\righton
\draw [ultra thick]  (0,0.25) to [out=0, in=-90] (0.5,0.5);
\draw [ultra thick]  (0.5,0) to [out=90, in=180] (1,0.25);
\draw [ultra thick]  (0.4,1) to [out=90, in=-90] (0.6,1.5);
\draw (1.25,0.75) node {+};
\end{scope}
\begin{scope}[xzp=0, xshift=5.8 cm]
\strandbackgroundshading
\beforevused
\aftervused
\beforewused
\afterwused
\cubestrandsetup
\lefton
\righton
\draw [ultra thick]  (0,1.25) to [out=0, in=-90] (0.5,1.5);
\draw [ultra thick]  (0.5,1) to [out=90, in=180] (1,1.25);
\draw [ultra thick]  (0.4,0) to [out=90, in=-90] (0.6,0.5);
\end{scope}
\end{tikzpicture} 
\caption{The two possible choices of strand diagrams are equal in homology.}
\label{fig:equal_in_homology}
\end{center}
\end{figure}

Thus, if we know $h,s,t$, then there is at most one corresponding generator of $H(\A(\ZZ))$, and it is given locally by figure \ref{fig:homology_generators_local}. We summarise this discussion by the following proposition.
\begin{prop}
\label{prop:strands_homology_geometrically}
Let $\mathcal{Z}$ be an arc diagram, let $s,t$ be two subsets of $\{1, \ldots, k\}$, and let $h \in H_1 ({\bf Z}, {\bf a})$. 
The $h$-graded summand of $I(s) \cdot H(\A(\ZZ)) \cdot I(t)$ is nonzero if and only if the data of $h,s,t$ near every pair of points $v,w \in {\bf a}$ matched by $M$ are as in one of the cases in figure \ref{fig:homology_generators_local} (up to a possible relabelling of $v$ and $w$). In this case the $h$-graded summand of $I(s) \cdot H(\A(\mathcal{Z})) \cdot I(t)$ is one-dimensional, generated by the corresponding diagram in figure \ref{fig:homology_generators_local}, which is unique in homology.
\qed
\end{prop}

\subsection{Multiplication in homology}

As $\A(\ZZ)$ is a differential graded algebra, multiplication is well-defined in $H(\A(\ZZ))$. Suppose we have two generators of $H(\A(\ZZ))$ represented by symmetrised $\ZZ$-constrained strand diagrams $\mu_0, \mu_1$. Suppose $\mu_i$ goes from $s_i$ to $t_i$ and has homological grading $h_i$, having support on used steps. If $t_0 \neq s_1$ then $\mu_0 \mu_1 = 0$; so assume $t_0 = s_1$. Then $\mu_0 \mu_1$ goes from $s_0$ to $t_1$ and has homological grading $h_0 + h_1$. We describe the homology class of $\mu_0 \mu_1$ in $H(\A(\ZZ))$ in the following proposition.
\begin{prop}\
\label{prop:multiplication_in_homology}
\begin{enumerate}
\item
If $\mu_0, \mu_1$ have a common used step, then $\mu_0 \mu_1$ is zero in homology.
\item
If $\mu_0, \mu_1$ have no common used step, and the $(h_0 + h_1)$-graded summand of $I(s_0) \cdot H(\A(\ZZ)) \cdot I(t_1)$ is nonzero, then $\mu_0 \mu_1$ represents the generator of this summand.
\end{enumerate}
\end{prop}

\begin{proof}
First, if $\mu_0, \mu_1$ have a common used step, then $(h_0 + h_1)$ has multiplicity $2$ on this step, so by theorem \ref{thm:LOT_theorem} the $(h_0 + h_1)$-graded summand of $H(\A(\ZZ))$ is zero. Second, if the $(h_0 + h_1)$-graded summand of $I(s_0) \cdot H(\A(\ZZ)) \cdot I(t_1)$ is zero, then $\mu_0 \mu_1$ must be zero in homology. Thus we can assume $\mu_0, \mu_1$ have no common used step, and that the $(h_0 + h_1)$-graded summand of $I(s_0) \cdot H(\A(\ZZ)) \cdot I(t_1)$ is nonzero; we show that $\mu_0 \mu_1$ generates this summand.

Since $\mu_0, \mu_1$ are nonzero in homology, by proposition \ref{prop:strands_homology_geometrically}, the local data of $h_0, s_0, t_0 = s_1$ and $h_1, s_1, t_1$ near every pair of matched places $v,w$ must be one of the cases shown in figure \ref{fig:homology_generators_local}. Moreover, the local data of $(h_0 + h_1), s_0, t_1$ near $v,w$ also appears in the figure, since the corresponding summand is nonzero.

Now if $\mu_0$ and $\mu_1$ do not concatenate along some $v,w$ then (after possibly relabelling $v,w$), $\mu_0$ must have a non-horizontal strand ending at $v$, and $\mu_1$ must have a non-horizontal strand beginning at $w$. Since $h_0, h_1$ have no step in common then the step before $v$ is used in $h_0$ but not $h_1$, and the step after $w$ is used in $h_1$ but not $h_0$. If the step after $v$ were used in $h_1$ then $\mu_1$ would have to have a strand beginning at $v$, but it already has one starting at $w$; so this step is not used in $h_1$. 
A similar argument shows that the step before $w$ is not used in $h_0$.
If the step before $w$ is used in $h_1$ then we have $w \in \Int(\supp h_1 )$ and $v \in {\bf a} \backslash (\supp h_1)$, so (by reference to figure \ref{fig:homology_generators_local}) $M(v) \notin s_1 \cup t_1$; but $M(v) \in s_1$ as $\mu_1$ has a strand beginning at $w$. This is a contradiction, so the step before $w$ is not used in $h_1$. A similar argument shows the step after $v$ is not used in $h_0$. Thus $h_0$ is supported only on the step before $v$, and hence $M(v) \notin s_0$. Similarly, $h_1$ is supported only on the step after $w$, and hence $M(v) \notin t_1$. Thus $(h_0 + h_1)$ is supported on the steps before $v$ and after $w$, so $\mu_0 \mu_1$ must have a strand beginning at $w$ and a strand ending at $v$, and yet $M(v) \notin s_0 \cup t_1$, yielding a contradiction. We conclude that $\mu_0, \mu_1$ must in fact concatenate along each $v,w$. As $\mu_0 \mu_1$ is obtained by such concatenation and has the starting places of $\mu_0$ and the ending places of $\mu_1$, it is a symmetrised $\ZZ$-constrained strand diagram.

As $\mu_0 \mu_1$ is a symmetrised $\ZZ$-constrained strand diagram, and the local data of $(h_0 + h_1), s_0, t_1$ near every matched pair $v,w$ appears in figure \ref{fig:homology_generators_local}, its homology class is nonzero and generates the $(h_0 + h_1)$-summand of $I(s_0) \cdot H(\A(\ZZ)) \cdot I(t_1)$.
\end{proof}

\section{The correspondence}
\label{sec:correspondence}

Having seen several similarities between cubulated contact structures and arc and strand diagrams, we now make the correspondence precise, establishing a ``dictionary" between contact structures and strand diagrams as shown in figure \ref{fig:dictionary}.

\begin{figure}
\begin{center}
\setlength\dashlinedash{0.2pt}
\setlength\dashlinegap{4.5pt}
\setlength\arrayrulewidth{0.2pt}
\begin{tabular}{c|c}
\hline \hline
{\bf Contact geometry} & {\bf Heegaard-Floer} \\ \hline \hline \\
Quadrangulated surface $(\Sigma, Q)$ & Arc diagram $\ZZ = ({\bf Z}, {\bf a}, M)$ \\ \hdashline
Positive vertex of $(\Sigma, V)$ & Interval $Z_j \in {\bf Z}$ \\ \hdashline
Square of $Q$ & Element of $\{1, \ldots, k\}$, Arc in $\ZZ$ \\ \hdashline
Positive vertex $v$ of square & Place $v \in {\bf a}$ \\ \hdashline
Pair of positive vertices $(v,w)$ of square & Matched places $M(v)=M(w)$, Endpoints of arc  \\ \hdashline
Index $I(\Sigma,Q) = \#$ squares of $Q$ & $k = \#$ matched pairs \\
$=\#$ cubes of cubulation $(\Sigma,Q) \times [0,1]$ & $= \frac{1}{2} |{\bf a}| = \frac{1}{2} \#$ places \\ 
\hdashline
Side faces of cubes & Steps of ${\bf Z}$ \\ \hdashline
Boundary edges of $(\Sigma,V)$ & Exterior steps of ${\bf Z}$ \\
Faces on side boundary of $(\Sigma,Q) \times [0,1]$ \\ \hdashline
Decomposing arcs of $Q$ & Interior steps of ${\bf Z}$ \\
Glued side faces of cubes \\ \hdashline
Side faces of cube before, after vertex & Steps before, after place \\
\hline
Basic dividing set $\Gamma$ & Subset $s \subseteq \{1, \ldots, k\}$, Idempotent $I(s)$ \\
\hdashline
Bottom squares which are on/negative & Beginning $s \subseteq \{1, \ldots, k\}$,  Left idempotent $I(s)$ \\
\hdashline
Top squares which are on/negative & Ending $t \subseteq \{1, \ldots, k\}$, Right idempotent $I(t)$ \\
\hdashline
Euler class $e(\Gamma) = I(\Sigma,V) - 2i$ & Sizes of beginning, ending sets $i = |s|=|t|$ \\
\hline
Used faces & Used steps = $\supp h$ \\
\hdashline
Data of top/bottom faces on/off, & Data of used/unused steps (or $h$), \\
side faces before/after $v,w$ used/unused & $s,t$ near $v,w$ \\
\hdashline
Data satisfies lemma \ref{lem:basic_cube_tight_conditions_1} 
or figure \ref{fig:tight_cubes}  & Data satisfies proposition \ref{prop:strands_homology_geometrically} 
or figure \ref{fig:homology_generators_local} \\
\hdashline
Given $\Gamma_0, \Gamma_1$, choice of used faces & Given $s,t$, choice of used steps \\
so each cube appears in  figure \ref{fig:tight_cubes} &  so each matched pair appears in figure \ref{fig:homology_generators_local} \\
\hdashline
All cubes tight & Symmetrised $\ZZ$-constrained strand diagram \\
Tight contact structure on $M(\Gamma_0, \Gamma_1)$ & which is nonzero in homology \\
\hline
$CA(\Sigma,Q)$ & $H(\A(\ZZ))$ \\
\hdashline
Summand $1_{\Gamma_0} \cdot CA(\Sigma,Q) \cdot 1_{\Gamma_1}$ & Summand $I(s) \cdot H(\A(\ZZ)) \cdot I(t)$ \\
\hdashline
Summand $CA_e (\Sigma,Q)$ & Summand $H(\A(\ZZ,i))$ \\
\hdashline
Relative Euler class $e(\xi)$ & Homological/spin-c grading \\
\hline
Stacking cubes & Multiplication in $H(\A(\ZZ))$ \\
\hdashline
Stack two used faces (=overtwisted) & Multiply diagrams with common used step (=0) \\
\hline
Bypass & Strand
\end{tabular}
\end{center}
\caption{Dictionary between contact geometry and Heegaard-Floer notions.}
\label{fig:dictionary}
\end{figure}

\subsection{From arc diagrams to quadrangulated surfaces and back}

Let $(\Sigma,Q)$ be a quadrangulated surface. As discussed in section \ref{sec:quadrangulations}, drawing the positive diagonal in each square yields the positive spine $G_Q^+$, which is a tape graph, and onto which $\Sigma$ retracts; the half-edges of $G_Q^+$ incident at a vertex are totally ordered clockwise around the vertex. And as discussed in section \ref{sec:arc_diagrams_tape_graphs}, a tape graph is equivalent to an arc diagram by ``blowing up" the vertices into line segments.

Conversely, from an arc diagram $\ZZ = ({\bf Z}, {\bf a}, M)$, we may collapse the line segments into vertices and obtain a tape graph $G_\ZZ$; then, with incident half-edges oriented clockwise around each vertex, we may thicken this tape graph into an oriented surface --- indeed, into a quadrangulated surface. In \cite[prop. 4.5]{Me14_twisty_itsy_bitsy} we gave a precise condition for when an oriented tape graph is the spine of a quadrangulated surface. Translated into the present context, the condition is that oriented surgery on ${\bf Z}$ at each $M^{-1}(i)$ results in a 1-manifold consisting of arcs, with no closed loops. (Each boundary component of the thickening contains a vertex of the graph, along with its adjacent barrier half-sides, in the language of \cite{Me14_twisty_itsy_bitsy}.) This condition is part of the definition of an arc diagram. Thus, as discussed in \cite{Me14_twisty_itsy_bitsy}, we naturally obtain a quadrangulated surface by thickening each edge of $G_\ZZ$ into the diagonal of a square, and gluing together sides of squares corresponding to adjacent half-edges. See figure \ref{fig:arc_diagram_to_surface}.

Thus, from $(\Sigma,Q)$ we obtain $\ZZ$ by retracting onto the positive spine, and then blowing up vertices into segments. From $\ZZ$ we obtain $(\Sigma,Q)$ by collapsing segments into vertices and then thickening each edge into a square.  

It is clear that these two processes are inverses of each other, and so there is a bijective correspondence between quadrangulated surfaces $(\Sigma,Q)$ and arc diagrams $\ZZ$, under which positive vertices of $(\Sigma,V)$ correspond to intervals of ${\bf Z}$, squares of $Q$ (or cubes of the cubulation) correspond to arcs of $\ZZ$ or elements of $\{1, \ldots, k\}$, a positive vertex of a particular square corresponds to a place in ${\bf a}$, and the two positive vertices of a square correspond to two matched places $v,w$ with $M(v) = M(w)$. The number of squares $I(\Sigma,V)$ in $Q$ (or cubes in the cubulation) is equal to the number $k = \frac{1}{2} |{\bf a}|$ of arcs or matched pairs. Moreover, the edges in $(\Sigma,Q)$ (or side faces of cubes) correspond to the steps of ${\bf Z}$: decomposing arcs of $Q$ (or glued side faces) correspond to interior steps of ${\bf Z}$, and boundary edges of $(\Sigma,Q)$ (or unglued side faces of cubes, those on the side boundary of $\Sigma \times [0,1]$) correspond to exterior steps of ${\bf Z}$. The faces of a cube before and after a vertex $v$ correspond to the steps before and after the corresponding place in ${\bf a}$.

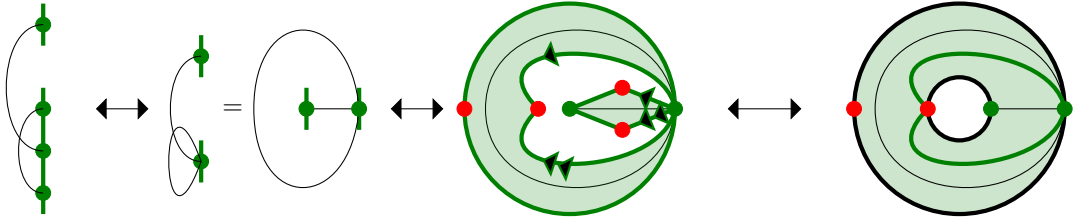
\begin{figure}
\begin{center}
\begin{tikzpicture}[scale=1.4,
decomposition/.style={thick, draw=green!50!black}, 
vertex/.style={draw=green!50!black, fill=green!50!black},
>=triangle 90, 
decomposition glued1/.style={thick, draw=green!50!black, postaction={nomorepostaction,decorate, decoration={markings,mark=at position 0.7 with {\arrow{>}}}}},
decomposition glued2/.style={thick, draw = green!50!black, postaction={nomorepostaction, decorate, decoration={markings,mark=at position 0.7 with {\arrow{>>}}}}}
]
\draw [ultra thick, color=black!50!green] (0,0) -- (0,1.2);
\draw [ultra thick, color=black!50!green] (0,1.6) -- (0,2);
\draw [color=black!50!green, fill=black!50!green] (0,0.2) circle (2 pt);
\draw [color=black!50!green, fill=black!50!green] (0,0.6) circle (2 pt);
\draw [color=black!50!green, fill=black!50!green] (0,1) circle (2 pt);
\draw [color=black!50!green, fill=black!50!green] (0,1.8) circle (2 pt);
\draw (0,0.2) to [bend left=90] (0,1);
\draw (0,0.6) to [bend left=90] (0,1.8);
\draw [<->] (0.5,1) -- (1,1);

\begin{scope}[xshift=1.5cm]
\draw [color=black!50!green, ultra thick] (0,0.3) -- (0,0.7);
\draw [color=black!50!green, ultra thick] (0,1.3) -- (0,1.7);
\draw [color=black!50!green, fill=black!50!green] (0,0.5) circle (2 pt);
\draw [color=black!50!green, fill=black!50!green] (0,1.5) circle (2 pt);
\draw (0,0.5) to [bend left=90] (0,1.5);
\draw (0,0.5) .. controls ($ (0,0.5) + (110:1.2) $) and ($ (0,0.5) + (-110:1.2) $) .. (0,0.5);
\draw (0.3,1) node {$=$};
\end{scope}
\begin{scope}[xshift=2cm]
\draw (1,1) -- (0.5,1);
\draw [color=black!50!green, ultra thick] (1,0.8) -- (1,1.2);
\draw [color=black!50!green, ultra thick] (0.5,0.8) -- (0.5,1.2);
\draw (1,1) .. controls ($ (1,1) + (95:1) $) and ($ (0,1) + (90:1) $) .. (0,1) .. controls ($ (0,1) + (-90:1) $) and ($ (1,1) + (-95:1) $) .. (1,1);
\draw [color=black!50!green, fill=black!50!green] (1,1) circle (2 pt);
\draw [color=black!50!green, fill=black!50!green] (0.5,1) circle (2 pt);
\draw [<->] (1.3,1) -- (1.8,1);
\end{scope}

\begin{scope}[xshift=5 cm]
\draw (1,1) -- (0,1);
\draw [fill=black!50!green, opacity=0.2] (1,1) -- (0.5,1.2) -- (0,1) -- (0.5,0.8) -- (1,1);
\draw [color=black!50!green, ultra thick] (1,1) -- (0.5,1.2) -- (0,1) -- (0.5,0.8) -- (1,1);

\draw (1,1) .. controls ($ (1,1) + (95:1) $) and ($ (-0.8,1) + (90:1) $) .. (-0.8,1) .. controls ($ (-0.8,1) + (-90:1) $) and ($ (1,1) + (-95:1) $) .. (1,1);
\draw [fill=black!50!green, opacity=0.2] (1,1) .. controls ($ (1,1) + (100:0.7) $) and ($ (-0.3,1) + (135:1) $) .. (-0.3,1) .. controls ($ (-0.3,1) + (-135:1) $) and ($ (1,1) + (-100:0.7) $) .. (1,1) arc (360:0:1);
\draw [color=black!50!green, ultra thick] (1,1) .. controls ($ (1,1) + (100:0.7) $) and ($ (-0.3,1) + (135:1) $) .. (-0.3,1) .. controls ($ (-0.3,1) + (-135:1) $) and ($ (1,1) + (-100:0.7) $) .. (1,1) arc (360:0:1);

\draw [decomposition glued1] (1,1) -- (0.5,1.2);
\draw [decomposition glued1] (1,1) .. controls ($ (1,1) + (100:0.7) $) and ($ (-0.3,1) + (135:1) $) .. (-0.3,1);

\draw [decomposition glued2] (1,1) -- (0.5,0.8);
\draw [decomposition glued2] (1,1) .. controls ($ (1,1) + (-100:0.7) $) and ($ (-0.3,1) + (-135:1) $) .. (-0.3,1);

\draw [color=black!50!green, fill=black!50!green] (1,1) circle (2 pt);
\draw [color=black!50!green, fill=black!50!green] (0,1) circle (2 pt);
\draw [color=red, fill=red] (0.5,1.2) circle (2 pt);
\draw [color=red, fill=red] (0.5,0.8) circle (2 pt);
\draw [color=red, fill=red] (-0.3,1) circle (2 pt);
\draw [color=red, fill=red] (-1,1) circle (2 pt);
\draw [<->] (1.5,1) -- (2.2,1);
\end{scope}

\begin{scope}[xshift = 8.7 cm]
\draw [fill=black!50!green, opacity=0.2] (1,1) arc (0:360:1);
\draw [fill=white] (0.3,1) arc (0:360:0.3);
\draw [ultra thick] (0,1) circle (1 cm);
\draw [ultra thick] (0,1) circle (0.3 cm);
\draw (1,1) -- (0.3,1);
\draw (1,1) .. controls ($ (1,1) + (95:1) $) and ($ (-0.8,1) + (90:1) $) .. (-0.8,1) .. controls ($ (-0.8,1) + (-90:1) $) and ($ (1,1) + (-95:1) $) .. (1,1);
\draw [color=black!50!green, ultra thick] (1,1) .. controls ($ (1,1) + (100:0.7) $) and ($ (-0.3,1) + (135:1) $) .. (-0.3,1) .. controls ($ (-0.3,1) + (-135:1) $) and ($ (1,1) + (-100:0.7) $) .. (1,1);
\draw [color=black!50!green, fill=black!50!green] (1,1) circle (2 pt);
\draw [color=black!50!green, fill=black!50!green] (0.3,1) circle (2 pt);
\draw [color=red, fill=red] (-1,1) circle (2 pt);
\draw [color=red, fill=red] (-0.3,1) circle (2 pt);
\end{scope}
\end{tikzpicture}

\caption{From arc diagram to quadrangulated surface. Green and red vertices are positive and negative.}
\label{fig:arc_diagram_to_surface}
\end{center}
\end{figure}

We note that a version of this construction appears in \cite[sec. 2.1]{Zarev09}, which constructs a surface from an arc diagram $\mathcal{Z}= ({\bf Z}, {\bf A}, M)$ by thickening each segment $Z_i$ into a rectangle $Z_i \times [0,1]$, and thickening the 1-handles so they are also rectangles, attached at $M^{-1}(i) \times \{0\}$. Our construction is an equivalent thickening.

Since the squares of $(\Sigma, Q)$ correspond to the elements of $\{1, \ldots, k\}$, we may associate to each subset $s \subseteq \{1, \ldots, k\}$ the basic dividing set $\Gamma$ where the squares corresponding to $s$ are on (have standard negative dividing set), and other squares are off (have positive dividing set). This dividing set is equivalent to the elementary dividing sets associated to $s$ in \cite[sec. 6.1]{Zarev10}. If $|s| = i$, then $e(\Gamma) = k-2i = I(\Sigma,V)-2i$. This gives a bijective correspondence between subsets of $\{1, \ldots, k\}$, and basic dividing sets on $(\Sigma,Q)$. If $\Gamma_0, \Gamma_1$ are basic dividing sets corresponding to subsets $s,t \subseteq \{1, \ldots, k\}$, then in $M(\Gamma_0, \Gamma_1)$, the cubes have bottom faces corresponding to $s$ and top faces corresponding to $t$.

\subsection{From contact structures to strand diagrams and back}

Let $(\Sigma,Q)$ be a quadrangulated surface corresponding to an arc diagram $\mathcal{Z} = ({\bf Z}, {\bf A}, M)$, and let $\Gamma_0, \Gamma_1$ be basic dividing sets corresponding to $s,t \subseteq \{1, \ldots, k\}$. Let $\xi$ be a tight contact structure on $M(\Gamma_0, \Gamma_1)$, which is cubulated by $(\Sigma,Q) \times [0,1]$. From proposition \ref{prop:contact_arc_labelling}, $\xi$ corresponds to a labelling of each decomposing arc of $Q$ as used or unused, so that the dividing set on each face of each cube is determined. Each cube satisfies the conditions of lemma \ref{lem:basic_cube_tight_conditions_1} and thus is one of the cases depicted in figure \ref{fig:tight_cubes}. 

To (the isotopy class of) $\xi$ we now associate a generator of $H(\A(\ZZ))$ (i.e. a homology class of symmetrised $\ZZ$-constrained strand diagram) as follows. We declare that the strand diagram goes from $s$ to $t$ and, since the decomposing arcs of $Q$ correspond to the interior steps of ${\bf Z}$, we declare the used interior steps to be those corresponding to the used decomposing arcs. These determine a homology class $h \in H_1 ({\bf Z}, {\bf a})$ supported on the used steps, each with multiplicity $1$.

A cube of the cubulation corresponds to an element of $\{1, \ldots, k\}$,  and hence to a matched pair of places. The positive vertices $v,w$ on the cube correspond to the two places, which we also call $v,w \in {\bf a}$. The bottom face is on or off accordingly as $M(v) = M(w)$ lies in $s$ or not; the top face is on or off accordingly as $M(v)=M(w)$ lies in $t$ or not; and the side faces before and after $v$ and $w$ are used or not accordingly as the steps before and after $v$ and $w$ are used or not. Thus the data of $\Gamma_0, \Gamma_1$ and the used faces of the cube determine the data of $h,s,t$ near $v,w$, and vice versa. We observe that a set of cube data appears in figure \ref{fig:tight_cubes} if and only if the corresponding data of $h,s,t$ near $v,w$ also appears; indeed, they are written next to each other in the figure. Thus the data of $h,s,t$ near each matched pair $v,w$ appears in figure \ref{fig:tight_cubes}, and by proposition \ref{prop:strands_homology_geometrically} the $h$-graded summand of $I(s) \cdot H(\A(\ZZ)) \cdot I(t)$ is one-dimensional, generated by the unique homology class of diagram given locally near each pair of marked points by figure \ref{fig:tight_cubes}. We associate to $\xi$ this homology class of symmetrised $\ZZ$-constrained strand diagram.

Conversely, to a generator of $H(\A(\ZZ))$, represented by a symmetrised $\ZZ$-constrained strand diagram $\mu$, we can associate an (isotopy class of) tight contact structure $\xi$. Let $\mu$ go from $s$ to $t$ and have homological grading $h$. We take $\xi$ on $M(\Gamma_0, \Gamma_1)$, where $\Gamma_0, \Gamma_1$ correspond to $s$ and $t$, such that the used faces of $\xi$ correspond to the used steps of $\mu$. Since $\mu$ is nonzero in $H(\A(\ZZ))$, near every pair of matched places $v,w$ the data of $h,s,t$ is one of the cases depicted in figure \ref{fig:tight_cubes}; hence for each cube, the cube data of $\xi$ also appears in figure \ref{fig:tight_cubes}, and hence each cube is tight. So we associate to $\mu$ the (isotopy class of) tight contact structure $\xi$ constructed from these tight cubes.

The correspondence between tight contact structures and generators of homology is clearly bijective. We simply pass back and forth between the local description of a contact structure near a cube, and the local description of a strand diagram near a pair of matched points, using figure \ref{fig:tight_cubes}. Thus we have a $\Z_2$-module isomorphism $CA(\Sigma,Q) \cong H(\A(\ZZ))$.

A contact structure on $M(\Gamma_0, \Gamma_1)$ corresponds to a generator from $s$ to $t$, and so this isomorphism restricts to isomorphisms of summands $1_{\Gamma_0} \cdot CA(\Sigma,Q) \cdot 1_{\Gamma_1} \cong I(s) \cdot H(\A(\ZZ)) \cdot I(t)$. And a contact structure with Euler class $e$ corresponds to a generator where $|s|=|t|=i$, where $e = k-2i$, so the isomorphism restricts to summands $CA_e (\Sigma,Q) \cong H(\A(\ZZ,i))$. An $h$-graded summand of $H(\A(\ZZ))$ corresponds to those contact structures with used faces given by the support of $h$, and hence to a specified relative Euler class summand of $CA(\Sigma,Q)$.

\subsection{Multiplication: the name of the game}

We now show that the $\Z_2$-module isomorphism $CA(\Sigma,Q) \To H(\mathcal{Z})$ preserves multiplication. So let $\Gamma_0, \Gamma_1, \Gamma_2$ be three basic dividing sets corresponding to subsets $s_0, s_1, s_2 \subseteq \{1, \ldots, k\}$, let $\xi_0$ be a tight contact structure on $M(\Gamma_0, \Gamma_1)$ corresponding to a generator of $I(s_0) \cdot H(\A(\ZZ)) \cdot I(s_1)$, represented by a strand diagram $\mu_0$, and let $\xi_1$ be a tight contact structure on $M(\Gamma_1, \Gamma_2)$ corresponding to a generator of $I(s_1) \cdot H(\A(\ZZ)) \cdot I(s_2)$ represented by a strand diagram $\mu_1$. Let $\xi_i$ have used faces $U_i$ and let $\mu_i$ have homological grading $h_i$.

We have seen that the used faces $U_i$ of each $\xi_i$ correspond to the used steps of each $\mu_i$. If $\xi_0, \xi_1$ have a common used face, then $\mu_0, \mu_1$ have a common used step. In this case stacking $\xi_0$ and $\xi_1$ yields an overtwisted contact structure, by proposition \ref{prop:stacking_cubulated_structures}(i), so $\xi_0 \xi_1 = 0$ in $CA(\Sigma,Q)$; and correspondingly, $\mu_0 \mu_1$ is zero in homology, by proposition \ref{prop:multiplication_in_homology}(i). We can now assume the used faces $U_i$ of the $\xi_i$ are disjoint, and the used steps of the $\mu_i$ are disjoint.

In this case, by proposition \ref{prop:stacking_cubulated_structures}(ii) $\xi_0 \xi_1$ is the cubulated contact structure on $M(\Gamma_0, \Gamma_2)$ with used faces given by $U_0 \cup U_1$. This contact structure is tight if and only if each cube is one of the cases depicted in figure \ref{fig:tight_cubes}. Similarly, by proposition \ref{prop:multiplication_in_homology}(ii), $\mu_0 \mu_1$ is nonzero in homology if and only if the $(h_0 + h_1)$-graded summand of $I(s_0) \cdot H(\A(\ZZ)) \cdot I(s_2)$ is nonzero, and by proposition \ref{prop:strands_homology_geometrically} this occurs if and only if the data of $h_0 + h_1, s_0, s_2$ near each matched pair $v,w$ is one of the cases depicted in figure \ref{fig:tight_cubes}. Thus $\xi_0 \xi_1$ is tight if and only if $\mu_0 \mu_1$ is nonzero in homology, and in this case, since the dividing sets $\Gamma_0, \Gamma_2$ of $\xi_0 \xi_1$ correspond to the beginning and end $s_0, s_2$ of $\mu_0 \mu_1$, and the used faces $U_0 \cup U_1$ of $\xi_0 \xi_1$ correspond to the used steps of $\mu_0 \mu_1$, the contact structure $\xi_0 \xi_1$ maps to the homology class of $\mu_0 \mu_1$ under the module isomorphism $CA(\Sigma,Q) \cong H(\A(\ZZ))$.

We have now proved theorem \ref{thm:main_thm}. In fact we have also proved isomorphisms of $\Z_2$-submodule summands
\[
CA_e (\Sigma,Q) \cong H(\A(\ZZ,i)) \quad
1_{\Gamma_0} \cdot CA(\Sigma,Q) \cdot 1_{\Gamma_1} \cong I(s_0) \cdot H(\A(\ZZ)) \cdot I(s_1)
\]
where $e = I(\Sigma,V) - 2i$ and $\Gamma_0, \Gamma_1$ are basic dividing sets corresponding to $s_0, s_1 \subseteq \{1, \ldots, k\}$.

We remark that the strands of a strand diagram can be interpreted as \emph{bypass additions}, which are compatible with the quadrangulation in an appropriate sense. Alternatively, if we draw the principal diagonals in faces which are on, then the upwards movement of a strand diagram corresponds to a clockwise rotation of such a diagonal about one of its endpoints. The ordering of the points on an arc diagram also provide an interesting refinement of the notion of \emph{partial orders} on objects of a contact category, as discussed in \cite{Me09Paper, Me16ContactCategories}, reminiscent of the clock theorem of formal knot theory \cite{Kauffman_FKT83}.

\section{Relation to sutured Floer homology}
\label{sec:SFH}

We finally prove corollary \ref{cor:SFH_dimension}, giving the dimension of sutured Floer homology related of the sutured manifold corresponding to our construction. As usual, let $\mathcal{Z}$ be an arc diagram corresponding to a quadrangulated surface $(\Sigma, Q)$. 

The corollary essentially now follows immediately from theorem 6.4 of \cite{Zarev10}, giving an algebra isomorphism
\[
H ( A( \mathcal{Z} ) ) \cong
\bigoplus_{\Gamma_0, \Gamma_1 \text{ basic}} SFH( -M(\Gamma_0, \Gamma_1) ).
\]
(We have translated Zarev's notation into our own: the $(F \times [0,1], \Gamma_{I \rightarrow J})$ of \cite[sec. 6]{Zarev10} has sutures $\Gamma_I \times \{0\}$ and $\Gamma_J \times \{1\}$, with a negative twist along the side boundary, opposite to the behaviour of a vertical dividing set. Reversing the orientation on $F$, but not on $[0,1]$, produces a sutured manifold $-F \times [0,1]$ with sutures $-\Gamma_I \times \{0\}$ and $-\Gamma_J \times \{1\}$, with a positive twist on the side boundary. These sutures do behave like a vertical dividing set and we denote this sutured manifold $-M(\Gamma_I, \Gamma_J)$.)

Moreover, if $s,t$ are subsets of $\{1, \ldots, k\}$ corresponding to basic dividing sets $\Gamma_s, \Gamma_t$, Zarev shows that the isomorphism above restricts to summands as an isomorphism
\[
I(s) \cdot H(A(\mathcal{Z})) \cdot I(t) \cong SFH(-M(\Gamma_s, \Gamma_t)).
\]

\begin{proof}[Proof of corollary \ref{cor:SFH_dimension}]
Combining theorem \ref{thm:main_thm} and Zarev's isomorphism, we have
\[
CA(\Sigma,Q) \cong 
\bigoplus_{\Gamma_0, \Gamma_1 \text{ basic}} SFH(-M(\Gamma_0, \Gamma_1))
\quad \text{and} \quad
1_{\Gamma_s} \cdot CA(\Sigma,Q) \cdot 1_{\Gamma_t} \cong 
SFH(-M(\Gamma_s, \Gamma_t)).
\]
The summand $1_{\Gamma_s} \cdot CA(\Sigma,Q) \cdot 1_{\Gamma_t}$ of the contact category algebra, as a $\Z_2$-module, has basis given by the isotopy classes contact structures on $M(\Gamma_s, \Gamma_t)$. Hence it has dimension equal to the number of isotopy classes of contact structures on $M(\Gamma_s, \Gamma_t)$.
\end{proof}



\bibliography{danbib}
\bibliographystyle{amsplain}

\end{document}